\documentclass[10pt]{article}

\usepackage{latexsym, amssymb, amsmath, amsthm, amscd}
\usepackage[all,cmtip]{xy}
\usepackage{amsmath}
\usepackage{amsthm}
\usepackage{amssymb}
\usepackage{mathtools}
\usepackage{amsfonts}
\usepackage{amsrefs}
\usepackage{color,authblk}
\usepackage{tikz}
\usepackage{amscd} 
\usepackage{comment}
\usepackage{mathrsfs}
\usepackage[all]{xy}
\usepackage{stackrel}
\usepackage{graphicx}
\usepackage{authblk}
\usepackage[hidelinks]{hyperref}
\usepackage{color,soul}
\usepackage{fullpage}
\usepackage[all,cmtip]{xy}
\usepackage{amsmath,amscd}
\usepackage{tikz-cd}
\usepackage{tikz}
\usetikzlibrary{matrix}

\usepackage{hyperref}

\numberwithin{equation}{section} \DeclareMathSizes{2}{10}{12}{13}
\newtheorem{thm}{Proposition}[section]
\newtheorem{Thm}[thm]{Theorem}
\newtheorem{rem}[thm]{Remark}

\newtheorem{cor}[thm]{Corollary}
\newtheorem{lem}[thm]{Lemma}
\newtheorem{defn}[thm]{Definition}

\numberwithin{thm}{section} 

\title{Categories of modules, comodules and contramodules over representations}

\author{Mamta Balodi \footnote{\scriptsize Department of Mathematics, Indian Institute of Science, Bangalore 560012, India. Email: mamta.balodi@gmail.com} $\qquad$ Abhishek Banerjee \footnote{\scriptsize Department of Mathematics, Indian Institute of Science, Bangalore 560012, India. Email: abhishekbanerjee1313@gmail.com} \footnote{\scriptsize AB was partially supported by SERB Matrics fellowship MTR/2017/000112} $\qquad$ Samarpita Ray \footnote{\scriptsize Department of Mathematics, Indian Institute of Science Education and Research, Pune 411008, India. Email: ray.samarpita31@gmail.com} \footnote{\scriptsize SR was partially supported by SERB National Postdoctoral Fellowship PDF/2020/000670}  }

\date{}

\begin{document}

\maketitle 

\medskip

\begin{abstract} We study and relate categories of modules, comodules and contramodules over a representation of a small category taking values in (co)algebras, in a manner similar to modules over a ringed space. As a result, we obtain a categorical framework which incorporates all the adjoint functors between these categories in a natural manner. Various classical properties of coalgebras and their morphisms arise naturally within this theory. We also consider cartesian objects in each of these categories, which may be viewed as counterparts of quasi-coherent sheaves over a scheme. We study their categorical properties using cardinality arguments. Our focus is on generators for these categories and on Grothendieck categories, because the latter may be treated as replacements for noncommutative spaces. 
\end{abstract}

\medskip
MSC(2020) Subject Classification:   16T15, 18E10


\hypersetup{linktocpage}

\tableofcontents

\section{Introduction}

The purpose of this paper is to obtain an algebraic geometry like categorical framework that studies modules, comodules and contramodules over a representation of a small category taking values in (co)algebras. In classical algebraic geometry, one usually has a ringed site, or more generally a ringed category $(X,\mathcal O)$ consisting of a small category $X$ and a presheaf $\mathcal O$ of commutative rings on $X$. Accordingly, a module $M$ over $(X,\mathcal O)$ corresponds to a family $\{M_x\}_{x\in X}$, where each $M_x$ is an $\mathcal O_x$ module, along with compatible morphisms. In more abstract settings, the idea of studying schemes by means of module categories linked with adjoint pairs given by extension and restriction of scalars is well developed in the literature.  This appears for instance in the relative algebraic geometry over symmetric monoidal categories (see Deligne \cite{Del}, To$\ddot{\text{e}}$n and Vaqui\'{e} \cite{TV}), in derived algebraic geometry (see Lurie \cite{Lurie}) and in homotopical algebraic geometry (see To$\ddot{\text{e}}$n and Vezzosi \cite{TV-1}, \cite{TV-2}).

\smallskip
In \cite{EV}, Estrada and Virili considered a representation $\mathcal A:\mathscr X\longrightarrow Add$ of a small category $\mathscr X$ taking values in the category $Add$ of small preadditive categories. Following the philosophy of Mitchell \cite{Mit}, the small preadditive categories play the role of ``algebras with several objects.'' An object $\mathcal M$ in the category $Mod\text{-}\mathcal A$ of $\mathcal A$-modules consists of the data of an $\mathcal A_x$-module $\mathcal M_x$ for each $x\in Ob(\mathscr X)$, along with compatible morphisms corresponding to extension or restriction of scalars.  The authors in \cite{EV} then establish a number of categorical properties of $\mathcal A$-modules, as also those of cartesian objects in $Mod\text{-}\mathcal A$, the latter being  similar to quasi-coherent modules over a scheme. As such, the study in \cite{EV} not only takes the philosophy of Mitchell one step further, but also provides a categorical framework for studying modules over ringed categories where the algebras are not necessarily commutative. 

\smallskip
The work of \cite{EV} is our starting point. For a small category $\mathscr X$, we consider either a representation $\mathcal C:\mathscr X\longrightarrow Coalg$ taking values in coalgebras or a representation $\mathcal A:\mathscr X\longrightarrow Alg$ taking values in algebras. In place of modules, we consider three different ``module like'' categories; those of modules, comodules and contramodules as well as incorporate all the adjoint functors between them into our theory. 
In doing so, we have two objectives. First of all, in each of these contexts, we also work with cartesian objects, which play a role similar to quasi-coherent sheaves over a scheme. By a classical result of Gabriel \cite{Gab} (see also Rosenberg \cite{Rose-1}, \cite{Rose}, \cite{Rose1}) we know that under certain conditions, a scheme can be reconstructed from its category of quasi-coherent sheaves. As such, the categories of cartesian comodules or cartesian contramodules may be  viewed as a step towards constructing a scheme like object related to comodules or contramodules over coalgebras. Our focus is on Grothendieck categories appearing in these contexts and more generally on generators of these categories. This is because Grothendieck categories may be treated as a replacement for noncommutative spaces as noted in \cite{LGS}. The latter is motivated by the work of \cite{Z1}, \cite{Z2}, \cite{Z3} as well as the observation in \cite{Low} that the Gabriel-Popescu theorem for Grothendieck categories may be viewed as an additive version of Giraud's theorem. We also note that module valued representations
of a small category have been studied by a number of authors (see, for instance, \cite{Oda}, \cite{EEES}, \cite{EEGR}). 

\smallskip
Secondly, our methods enable us to explore the richness of the theory of comodules and coalgebras with the flavor of algebraic geometry. In this framework, it becomes natural to include the theory of contramodules, which  have been somewhat neglected in the literature. Formally, the notion of a contramodule is also dual to that of a module, if we write the structure map of a module $P$ over an algebra $A$ as a morphism $P\longrightarrow Hom(A,P)$ instead of the usual $P\otimes A\longrightarrow P$. Accordingly, a contramodule $M$ over a coalgebra $C$ consists of a space $M$ as well as a morphism $Hom(C,M)\longrightarrow M$ satisfying certain coassociativity and counit conditions (see Section 5). While contramodules were introduced much earlier by Eilenberg and Moore \cite[$\S$ IV.5]{EM}, the subject has seen a lot of interest in recent years (see, for instance, \cite{BPS}, \cite{Pst2}, \cite{BBW}, \cite{semicontra}, \cite{Pmem}, \cite{P2}, \cite{Pst1}, \cite{P}, \cite{co-contra}, 
\cite{P1}, \cite{Sha}, \cite{Wis}). 

\smallskip One important aspect of our paper is that for comodules over a coalgebra representation $\mathcal C:\mathscr X\longrightarrow Coalg$ or modules over an algebra representation $\mathcal A:\mathscr X\longrightarrow Alg$, it becomes necessary to work with objects of  two different orientations, which we refer to as ``cis-objects'' and ``trans-objects.''  We shall see that cis-comodules over a coalgebra representation are related to trans-modules over its dual algebra representation and vice-versa. 

\smallskip
For a coalgebra $C$ over a field $K$, let $\mathbf M^C$ denote the category of right $C$-comodules. Given a morphism $\alpha:C\longrightarrow D$ of coalgebras, we consider a system of three different functors between comodule categories
\begin{equation}\label{1Icom3}
\alpha^!:\mathbf M^D\longrightarrow \mathbf M^C\qquad \alpha^*:\mathbf M^C\longrightarrow \mathbf M^D \qquad
\alpha_*:\mathbf M^D\longrightarrow \mathbf M^C
\end{equation}
Here, $\alpha^*:\mathbf M^C\longrightarrow \mathbf M^D$ is the corestriction of scalars and its right adjoint is given by the cotensor product $\alpha_*
=\_\_\square_DC:\mathbf M^D\longrightarrow \mathbf M^C$. In addition, if $\alpha:C\longrightarrow D$ makes $C$ quasi-finite as a right $D$-comodule (see Section 2.2), then the left adjoint $\alpha^!:\mathbf M^D\longrightarrow \mathbf M^C$ of the corestriction functor $\alpha^*$ also exists. In order to define a cis-comodule $\mathcal M$ over a coalgebra representation $\mathcal C:\mathscr X\longrightarrow Coalg$, we need a collection $\{\mathcal M_x\}_{x\in Ob(\mathscr X)}$ where each $\mathcal M_x$ is a $\mathcal C_x$-comodule, along with compatible morphisms $\mathcal M^\alpha:\alpha^*\mathcal M_x\longrightarrow\mathcal M_y$ of $\mathcal C_y$-comodules for $\alpha\in \mathscr X(x,y)$ (see Definition \ref{comod-rep}). Equivalently, we have morphisms $\mathcal M_\alpha:\mathcal M_x\longrightarrow \alpha_*\mathcal M_y$ of $\mathcal C_x$-comodules for each  $\alpha\in \mathscr X(x,y)$. By combining techniques on comodules with adapting the cardinality arguments of \cite{EV}, we study the category $Com^{cs}$-$\mathcal C$ of cis-comodules and give conditions for it to be a Grothendieck category. It turns out that the  relevant criterion is for the representation $\mathcal C:\mathscr X\longrightarrow Coalg$ to take values in semiperfect coalgebras, i.e., those  for which the category of comodules has enough projectives. The semiperfect coalgebras also return in the last section, where they make an interesting appearance with respect to torsion theories.

\smallskip
On the other hand, a trans-comodule $\mathcal M$ over $\mathcal C:\mathscr X\longrightarrow Coalg$ consists of morphisms $_{\alpha}\mathcal M:\mathcal M_y\longrightarrow \alpha^*\mathcal M_x$ for each $\alpha\in \mathscr X(x,y)$. Whenever the coalgebra representation   is quasi-finite, i.e., each morphism $\mathcal C_x
\longrightarrow \mathcal C_y$ of coalgebras induced by $\mathcal C:\mathscr X\longrightarrow Coalg$  makes $\mathcal C_x$ quasi-finite as a $\mathcal C_y$-comodule, this is equivalent to having morphisms $^{\alpha}\mathcal M:\alpha^!\mathcal M_y\longrightarrow \mathcal M_x$  for each $\alpha\in \mathscr X(x,y)$. We study the category $Com^{tr}$-$\mathcal C$ of trans-comodules in a manner similar to $Com^{cs}$-$\mathcal C$.  When the small category $\mathscr X$ is a poset, we show that the evaluation functor at each $x\in Ob(\mathscr X)$ has both a left and a right adjoint. This enables us to construct explicit projective generators for $Com^{cs}$-$\mathcal C$ and $Com^{tr}$-$\mathcal C$, by making use of the projective generators in the category of comodules over each of the semiperfect coalgebras $\{\mathcal C_x\}_{x\in Ob(\mathscr X)}$. 

\smallskip
A cartesian object in the category of cis-comodules consists of $\mathcal M\in Com^{cs}$-$\mathcal C$ such that for each $\alpha\in \mathscr X(x,y)$, the morphism $\mathcal M_\alpha:\mathcal M_x\longrightarrow \alpha_*\mathcal M_y$ is an isomorphism. In order to study these objects, we will suppose that the representation $\mathcal C:\mathscr X\longrightarrow Coalg$ is coflat, i.e., the cotensor products $\_\_\square_{\mathcal C_y}\mathcal C_x$ corresponding to any morphism $\mathcal C_x
\longrightarrow \mathcal C_y$  of coalgebras induced by $\mathcal C$ are exact. By using a transfinite induction argument adapted from \cite{EV}, we will show that for any coflat  and semiperfect representation $\mathcal C:\mathscr X\longrightarrow Coalg$ of a poset, there exists a cardinal $\kappa'$ such that any cartesian cis-comodule $\mathcal M$ over $\mathcal C$ can be expressed as a filtered union of cartesian subcomodules each of cardinality $\leq \kappa'$ (see Theorem \ref{T4.8b}). It follows in particular that the category
$Com^{cs}_c$-$\mathcal C$ of cartesian cis-comodules over such a representation is a Grothendieck category. Here, we also refer the reader to the classical result of Gabber (see, for instance, \cite[Tag 077K]{Stacks}), which shows that the category of quasi-coherent sheaves over a scheme is a Grothendieck category. We also obtain a right adjoint of the inclusion functor
$Com^{cs}_c$-$\mathcal C\hookrightarrow Com^{cs}$-$\mathcal C$, which may be viewed as a coalgebraic counterpart of the classical quasi-coherator construction (see Illusie \cite[Lemme 3.2]{Ill}). 

\smallskip
We will say that a quasi-finite morphism $\alpha:C\longrightarrow D$ of coalgebras is right $\Sigma$-injective if the direct sum $C^{(\Lambda)}$ is injective as a right $D$-comodule for any indexing set $\Lambda$. This is equivalent (see  \cite[Corollary 3.10]{takh}) to the functor $\alpha^!:\mathbf M^D\longrightarrow \mathbf M^C$ being exact. An object $\mathcal M\in Com^{tr}$-$\mathcal C$ is cartesian if the morphism $^{\alpha}\mathcal M:\alpha^!\mathcal M_y\longrightarrow \mathcal M_x$ is an isomorphism for each $\alpha\in \mathscr X(x,y)$. We then study the category $Com^{tr}_c$-$\mathcal C$ of cartesian trans-contramodules over a $\Sigma$-injective and semiperfect coalgebra representation in a manner similar to $Com^{cs}_c$-$\mathcal C$.

\smallskip
If $C$ is a coalgebra over a field $K$, we denote by $\mathbf M_{[C,\_\_]}$ the category of (right) $C$-contramodules. If $\alpha:C\longrightarrow D$ is a morphism of coalgebras, we have a pair of adjoint functors
\begin{equation}\label{1Icont2}
\alpha^\bullet: \mathbf M_{[D,\_\_]}\longrightarrow\mathbf M_{[C,\_\_]}\qquad \alpha_\bullet :\mathbf M_{[C,\_\_]}\longrightarrow \mathbf M_{[D,\_\_]}
\end{equation}
between corresponding categories of contramodules. Here, the functor $\alpha_\bullet$ is the contrarestriction of scalars, obtained by treating a $C$-contramodule $(M,\pi^C_M:Hom_K(C,M)\longrightarrow M)$ as a $D$-contramodule with induced structure map $Hom_K(D,M)\longrightarrow Hom_K(C,M)\xrightarrow{\pi^C_M}M$. Let $\mathcal C:\mathscr X\longrightarrow Coalg$ be a coalgebra representation. A trans-contramodule $\mathcal M$ over $\mathcal C$ consists of a $\mathcal C_x$-contramodule $\mathcal M_x$ for each $x\in 
Ob(\mathscr X)$ along with compatible morphisms $_{\alpha}\mathcal M:\mathcal M_y\longrightarrow\alpha_\bullet\mathcal M_x$ for  $\alpha\in \mathscr X(x,y)$. Equivalently, we have a morphism
$^\alpha\mathcal M:\alpha^\bullet\mathcal M_y\longrightarrow \mathcal M_x$ for $\alpha\in \mathscr X(x,y)$. We note that we have only a single pair of adjoint functors in \eqref{1Icont2}, unlike the system of three adjoint functors for comodule categories in \eqref{1Icom3}. As such, we consider only the category $Cont^{tr}$-$\mathcal C$ of trans-contramodules over a representation $\mathcal C:\mathscr X\longrightarrow Coalg$ but no cis-contramodules. 

\smallskip
In comparison to comodule categories, working with contramodules presents certain difficulties. The first among these is the fact that the category $\mathbf M_{[C,\_\_]}$ of contramodules over a $K$-coalgebra $C$ is not usually a Grothendieck category. Further, direct sums in the category of contramodules do not correspond in general to the direct sums of their underlying vector spaces. As a result, we work by considering morphisms from presentable generators in contramodule categories, using an adjunction between vector spaces and  contramodules. Interestingly, the category of contramodules $\mathbf M_{[C,\_\_]}$  does contain enough projectives. For a coalgebra representation $\mathcal C:\mathscr X\longrightarrow Coalg$ we show that $Cont^{tr}$-$\mathcal C$ has a set of generators. When the small  category $\mathscr X$ is a poset, we show that $Cont^{tr}$-$\mathcal C$  is in fact locally presentable and we construct a set of projective generators for $Cont^{tr}$-$\mathcal C$.

\smallskip
When $\alpha:C\longrightarrow D$ is a coflat morphism of coalgebras, one observes that the contraextension functor $\alpha^\bullet:\mathbf M_{[D,\_\_]}\longrightarrow \mathbf M_{[C,\_\_]}$ is exact. Accordingly, we say that a trans-contramodule $\mathcal M$ over a coflat representation $\mathcal C:\mathscr X\longrightarrow Coalg$ is cartesian if the morphism
$^\alpha\mathcal M:\alpha^\bullet\mathcal M_y\longrightarrow \mathcal M_x$ is an isomorphism for each $\alpha\in \mathscr X(x,y)$. For each $x\in Ob(\mathscr X)$, we choose a regular cardinal $\lambda_x$ such that the dual $\mathcal C_x^*=Hom_K(\mathcal C_x,K)$ is a $\lambda_x$-presentable generator in $\mathbf M_{[\mathcal C_x,\_\_]}$. Since colimits of contramodules do  not correspond to the colimits of underlying vector spaces, we rely extensively on  presentable objects in each $\mathbf M_{[\mathcal C_x,\_\_]}$  to set up a transfinite induction argument.  When $\mathscr X$ is a poset and   
$\mathcal C:\mathscr X\longrightarrow Coalg$ is a coflat representation, we show that there is a regular cardinal $\kappa''$ such that any element of a cartesian trans-contramodule over $\mathcal C$ lies in a cartesian subobject having cardinality $\leq \kappa''$ (see Theorem \ref{T6.11exes}). 

\smallskip
Finally, we consider modules over an algebra representation $\mathcal A:\mathscr X\longrightarrow Alg$ and proceed to relate the categories of modules, comodules and contramodules to each other. For any $K$-algebra $A$, we denote by $\mathbf M_A$ the category of right $A$-modules.  Corresponding to any morphism $\alpha:A\longrightarrow B$ of algebras, we have a system of three adjoint functors
\begin{equation}\label{1Imod3}
\alpha^\circ:\mathbf M_A\longrightarrow \mathbf M_B\qquad \alpha_\circ:\mathbf M_B\longrightarrow \mathbf M_A\qquad\alpha^\dagger:\mathbf M_A\longrightarrow \mathbf M_B
\end{equation} Here $\alpha_\circ$ is the usual restriction of scalars, $\alpha^\circ$ is its left adjoint, while $\alpha^\dagger$ is its right adjoint. Accordingly, we can define categories 
$Mod^{cs}$-$\mathcal A$ and $Mod^{tr}$-$\mathcal A$ respectively of right  cis-modules and right trans-modules over $\mathcal A$. We observe in particular that the cis-modules recover the modules of Estrada and Virili \cite{EV} (in the case where the representation in \cite{EV} takes values in $K$-algebras). 

\smallskip
We now relate comodules over a coalgebra representation $\mathcal C:\mathscr X\longrightarrow Coalg$ to modules over its linear dual representation $\mathcal C^*:\mathscr X^{op}
\longrightarrow Alg$ as well as modules over an algebra representation $\mathcal A:\mathscr X^{op}\longrightarrow Alg$ to comodules over its finite dual representation $\mathcal A^\circ:
\mathscr X\longrightarrow Coalg$. More generally, we define a rational pairing $(\mathcal C,\mathcal A,\Phi)$ of a coalgebra representation $ \mathcal C:\mathscr X\longrightarrow Coalg$  with an algebra representation $\mathcal A:\mathscr X^{op}\longrightarrow Alg$. We recall (see, for instance, \cite[$\S$ 4.18]{BW}) that if $\varphi:C\otimes A\longrightarrow K$ is a rational pairing of a coalgebra $C$ with an algebra $A$, the category of right $C$-comodules can be embedded as the category of rational left $A$-modules. Accordingly, we construct pairs of adjoint functors
\begin{equation}
\begin{tikzcd}[column sep=3cm, row sep=0.7cm]
Com^{cs}\text{-}\mathcal C\arrow[r,shift left,"I_\Phi^{tr}"]
    &
\mathcal A\text{-}Mod^{tr}\arrow[l,shift left,"R_\Phi^{tr}"]
\end{tikzcd}\qquad 
\begin{tikzcd}[column sep=3cm, row sep=0.7cm]
Com^{tr}\text{-}\mathcal C\arrow[r,shift left,"I_\Phi^{cs}"]
    &
\mathcal A\text{-}Mod^{cs}\arrow[l,shift left,"R_\Phi^{cs}"]
\end{tikzcd}
\end{equation}
Here, $I_\Phi^{cs}$ and $I_\Phi^{tr}$ are inclusion functors while their respective right adjoints $R_\Phi^{cs}$ and $R_\Phi^{tr}$ are ``rationalization functors.'' Thereafter, using some classical results from \cite{Lin} on coalgebras and density in module categories, we give conditions for the full subcategory $Com^{cs}$-$\mathcal C$ to be a torsion class in $\mathcal A$-$Mod^{tr}$. In particular, it follows that if $\mathcal C:\mathscr X\longrightarrow Coalg$ is a representation taking values in semiperfect coalgebras, then $Com^{cs}$-$\mathcal C$ becomes a hereditary torsion class in $\mathcal C^*$-$Mod^{tr}$. For any coalgebra representation $\mathcal C:\mathscr X\longrightarrow Coalg$, we also construct a canonical functor
$Cont^{tr}$-$\mathcal C\longrightarrow Mod^{cs}$-$\mathcal C^*$. 

\smallskip
We conclude with the following result: if $\mathcal C:\mathscr X\longrightarrow Coalg$ is a quasi-finite representation taking values in cocommutative coalgebras, for each cartesian trans-comodule $\mathcal N\in Com^{tr}_c$-$\mathcal C$, we produce a pair of adjoint functors 
\begin{equation}
\begin{tikzcd}[column sep=3cm, row sep=0.7cm]
Cont^{tr}\text{-}\mathcal C\arrow[r,shift left,"{\boxtimes_{\mathcal C}\mathcal N}"]
    &
Com^{tr}\text{-}\mathcal C\arrow[l,shift left,"{(\mathcal N,\_\_)}"]
\end{tikzcd}
\end{equation} determined by $\mathcal N$ (see Theorem \ref{T7.19fe}). Here, the left adjoint is obtained by using the contratensor product of a contramodule and a comodule constructed by Positselski \cite{semicontra}.

\medskip

{\bf Acknowledgements:} The authors are grateful to L.~Positselski for a useful discussion.

\section{Comodules over coalgebra representations}
Throughout this paper, $K$ will denote a field and $Vect$ the category of $K$-vector spaces. Let $C$ be a $K$-coalgebra. We denote by $\mathbf{M}^C$ (resp. ${^C}\mathbf{M}$) the category of right $C$-comodules (resp. left $C$-comodules). We also denote by $\rho^M:M \longrightarrow M \otimes C$ (resp. ${^M}\rho:M \longrightarrow C \otimes M$) the right coaction (resp. the left coaction) for $M \in \mathbf{M}^C$ (resp. $M \in {{^C}\mathbf{M}}$). 

\smallskip
For $M \in \mathbf{M}^C$ and $N \in {^C}\mathbf{M}$, we recall (see, for instance, \cite[$\S$10.1]{BW}) that their cotensor product $M \square_C N$ is given by the equalizer of the two maps
\begin{equation}
\begin{tikzcd}[column sep=3cm, row sep=0.7cm]
M\otimes N\ar[r,shift left=.75ex,"\rho^M \otimes id_N"]
  \ar[r,shift right=.75ex,swap,"id_M \otimes {^N}{\rho}"]
&
M\otimes C\otimes N
\\
\end{tikzcd}
\end{equation}
If $M\in \mathbf M^C$, we note the isomorphism $M \square_C C \simeq M$ of right $C$-comodules.  
A coalgebra morphism $\alpha:C \longrightarrow D$  induces a pair of adjoint functors
\begin{equation*}
\begin{array}{ll}
\alpha^*: \mathbf{M}^C \longrightarrow \mathbf{M}^D \qquad (M,\rho^M) \mapsto \left(M, (id_M \otimes \alpha) \circ \rho^M \right)\\
\alpha_*:  \mathbf{M}^D \longrightarrow \mathbf{M}^C \qquad  N \mapsto N \square_D C
\end{array}
\end{equation*} The left adjoint $\alpha^*$ is known as the corestriction functor, while the right adjoint $\alpha_*$ is known as the coinduction functor.

\begin{defn}

  Let $\mathscr{X}$ be a small category and let $Coalg$ denote the category of $K$-coalgebras.  By a representation of $\mathscr{X}$ on the category of $K$-coalgebras, we will mean a functor $\mathcal{C}: \mathscr{X}\longrightarrow Coalg$. 
  
  \smallskip In particular, for each object $x\in Ob(\mathscr{X})$, we have a coalgebra $\mathcal{C}_x$ and for any morphism  $\alpha \in \mathcal{X}(x,y)$, we have a morphism $\mathcal{C}_{\alpha}: \mathcal{C}_x\longrightarrow \mathcal{C}_y$ of  coalgebras.

\end{defn}

By abuse of notation, given a coalgebra representation  $\mathcal{C}: \mathscr{X}\longrightarrow Coalg$, we will simply write $\alpha^*$, $\alpha_*$ and so on for the respective functors
$\mathcal C_\alpha^*$, $\mathcal C_{\alpha *}$, etc  between corresponding categories of comodules.
Corresponding to a representation of $\mathscr{X}$ on the category of coalgebras, we will now define two types of comodules, namely, cis-comodules and trans-comodules. We will
work with cis-comodules in the first subsection and with trans-comodules in the next.

\subsection{Cis-comodules over coalgebra representations}

\begin{defn}\label{comod-rep}

  Let $\mathcal{C}: \mathscr{X}\longrightarrow Coalg$ be a coalgebra representation. A (right) cis-comodule $\mathcal{M}$ over $\mathcal{C}$ will consist of the following data:
  \begin{itemize}
  \item[(1)] For each object $x\in Ob(\mathscr{X})$, a right $\mathcal{C}_x$-comodule $\mathcal{M}_x$
  \item[(2)] For each morphism
    $\alpha: x \longrightarrow y$ in $\mathscr{X}$, a morphism
    ${\mathcal{M}}_{\alpha}: \mathcal{M}_x\longrightarrow \mathcal{M}_y\square_{C_y}\mathcal{C}_x
    =\alpha_*\mathcal{M}_y$ of right $\mathcal{C}_x$-comodules (equivalently, a morphism $\mathcal{M}^{\alpha}: \alpha^*\mathcal{M}_x\longrightarrow \mathcal{M}_y$ of right $\mathcal{C}_y$-comodules)
  \end{itemize}

  We further assume that $\mathcal{M}_{id_x} = id_{\mathcal{M}_x}$ and for any pair of composable morphisms
  $x\xrightarrow{\alpha} y\xrightarrow{\beta}z$ in $\mathscr{X}$, we have $\alpha_*(\mathcal{M}_{\beta})\circ \mathcal{M}_{\alpha}= \mathcal{M}_{\beta\alpha}:\mathcal{M}_x\longrightarrow \alpha_*\mathcal{M}_y\longrightarrow \alpha_*\beta_*\mathcal{M}_z=(\beta\alpha)_*\mathcal{M}_z$. The later condition can be equivalently expressed as $\mathcal{M}^{\beta\alpha}=\mathcal{M}^{\beta}\circ \beta^*(\mathcal{M}^\alpha)$.\\
  A morphism $\eta: \mathcal{M} \longrightarrow \mathcal{N}$ of cis-comodules over $\mathcal{C}$ consists of morphisms $\eta_x: \mathcal{M}_x\longrightarrow\mathcal{N}_x$ of right $\mathcal{C}_x$-comodules for   $x\in Ob(\mathscr{X})$ such that for each morphism $x\longrightarrow y$ in $\mathscr{X}$ the following diagram commutes

  \[\begin{tikzcd}
\mathcal{M}_x \arrow{r}{\eta_x} \arrow[swap]{d}{\mathcal{M}_\alpha} & \mathcal{N}_x \arrow{d}{\mathcal{N}_\alpha} \\
\alpha_*\mathcal{M}_y \arrow{r}{\alpha_*\eta_y} & \alpha_*\mathcal{N}_y
\end{tikzcd}
\]
We denote this category of right cis-comodules by $Com^{cs}$-$\mathcal{C}$. Similarly, we may define the category $\mathcal C$-$Com^{cs}$ of left cis-comodules over $\mathcal C$. 
\end{defn}

\begin{thm}\label{Ab}

Let $\mathcal{C}: \mathscr{X}\longrightarrow Coalg$ be a coalgebra representation. Then, $Com^{cs}$-$\mathcal{C}$  is an abelian category.

\end{thm}

\begin{proof}
  Clearly, $Com^{cs}$-$\mathcal{C}$ has a zero object. For any morphism $\eta:\mathcal{M}\longrightarrow \mathcal{N}$ in $Com^{cs}$-$\mathcal{C}$, we define the kernel and cokernel of $\eta$ by setting
  \begin{equation}\label{v2.2y} 
    {Ker(\eta)}_x:=Ker(\eta_x:\mathcal{M}_x\longrightarrow \mathcal{N}_x)\qquad {Coker(\eta)}_x:= Coker(\eta_x: \mathcal{M}_x\longrightarrow \mathcal{N}_x)
    \end{equation}
    for each $x\in \mathscr{X}$. For $\alpha:x\longrightarrow y$ in $\mathscr{X}$, the exactness of the corestriction functor $\alpha^*: \mathcal{M}^{\mathcal{C}_x}\longrightarrow \mathcal{M}^{\mathcal{C}_y}$ induces the morphisms ${Ker(\eta)}^{\alpha}: \alpha^*{Ker(\eta)}_x\longrightarrow {Ker(\eta)}_y$ and ${Coker(\eta)}^{\alpha}: \alpha^*{Coker(\eta)}_x\longrightarrow {Coker(\eta)}_y$. It is also clear from \eqref{v2.2y} that $Coker(Ker(\eta)\hookrightarrow \mathcal{M})= Ker (\mathcal{N}\twoheadrightarrow  Coker(\eta))$.
\end{proof}

  We will now study generators in the category $Com^{cs}$-$\mathcal{C}$. For this, we recall (see, for instance, \cite[Corollary 2.2.9]{DNR}) that any right $C$-comodule is the sum of its finite dimensional subcomodules. If $C$ is a right semiperfect $K$-coalgebra, i.e., the category $\mathbf M^C$  has enough projective objects, we know (see, for instance, \cite[Corollary 2.4.21]{DNR}) that for any finite dimensional right $N\in \mathbf M^C$, there exists a finite dimensional projective $P\in \mathbf M^C$ and an epimorphism $P\longrightarrow N$ in  $\mathbf M^C$. Therefore, it follows that if $C$ is right semiperfect, $\mathbf M^C$ has finitely generated projective generators.\\

For studying generators in $Com^{cs}$-$\mathcal C$, we will now adapt the steps in \cite[$\S$ 4]{AB}, which are motivated by Estrada and Virili \cite{EV}. In fact, we will use a similar argument in several contexts throughout this paper.   Let $\mathcal{C}: \mathscr{X}\longrightarrow Coalg$ be a coalgebra representation. Let $\mathcal{M}$ be a right cis-comodule over $\mathcal{C}$. We consider an object $x\in Ob(\mathscr{X})$ and a morphism
\begin{equation}
  \eta: V\longrightarrow  \mathcal{M}_x
\end{equation}
where $V$ is a finite dimensional projective in $\mathbf M^{\mathcal{C}_x}$. For each object $y\in Ob(\mathscr{X})$, we set $\mathcal{N}_y\subseteq \mathcal{M}_y$ to be the image of the family of maps

\begin{align}\label{defN}
  \mathcal{N}_y&= Im(\bigoplus_{\beta \in \mathscr{X}(x,y)}\beta^*V \xrightarrow{\beta^*\eta}\beta^*\mathcal{M}_x\xrightarrow{\mathcal{M}^{\beta}}\mathcal{M}_y) 
  &= \sum_{\beta \in \mathscr{X}(x,y)}Im(\beta^*V \xrightarrow{\beta^*\eta}\beta^*\mathcal{M}_x\xrightarrow{\mathcal{M}^{\beta}}\mathcal{M}_y)
\end{align} in $\mathbf M^{\mathcal C_y}$.
Let $i_y$ denote the inclusion $\mathcal{N}_y\hookrightarrow \mathcal{M}_y$ and for each $\beta\in \mathscr{X}(x,y)$, we denote by $\eta'_\beta=\mathcal M^\beta\circ \beta^\ast \eta:\beta^*V\longrightarrow \mathcal{N}_y$ the canonical morphism induced from \eqref{defN}.

\begin{lem}\label{lem1}
  Let $\alpha \in \mathscr{X}(y,z)$ and $\beta\in \mathscr{X}(x,y)$. Then, the following composition
  \begin{equation}
    \beta^*V\xrightarrow{\eta'_\beta} \mathcal{N}_y\xrightarrow{i_y}\mathcal{M}_y\xrightarrow{\mathcal{M}_\alpha }\alpha_*\mathcal{M}_z
  \end{equation}
factors through $\alpha_*(i_z):\alpha_*\mathcal{N}_z\longrightarrow \alpha_*\mathcal{M}_z$.
\end{lem}

\begin{proof}
  It is enough to show that the composition
  \begin{equation}
\alpha^*\beta^*V\xrightarrow{\alpha^*(\eta'_\beta)} \alpha^*\mathcal{N}_y\xrightarrow{\alpha^*i_y}\alpha^*\mathcal{M}_y\xrightarrow{\mathcal{M}^\alpha} \mathcal{M}_z
\end{equation}
factors through $i_z: \mathcal{N}_z\longrightarrow \mathcal{M}_z$ since $(\alpha^*,\alpha_*)$ is an adjoint pair. By definition, we know that the composition $
\beta^*V\xrightarrow{\eta'_\beta} \mathcal{N}_y\xrightarrow{i_y}\mathcal{M}_y$ factors through $\beta^*\mathcal{M}_x$ i.e., we have $i_y\circ {\eta'}_\beta= \mathcal{M}^\beta \circ\beta^*\eta$. Applying $\alpha^*$, composing with $\mathcal{M}^{\alpha}$ and using Definition \ref{comod-rep} we obtain
\begin{equation}
\mathcal{M}^\alpha\circ \alpha^*(i_y)\circ \alpha^*({\eta'}_\beta)=\mathcal{M}^\alpha\circ \alpha^*(\mathcal{M}^\beta)\circ\alpha^*(\beta^*\eta)=\mathcal{M}^{\alpha\beta}\circ \alpha^*\beta^*\eta
  \end{equation}
By the definition in \eqref{defN}, we know that $ \mathcal{M}^{\alpha\beta}\circ \alpha^*\beta^*\eta$ factors through $i_z: \mathcal{N}_z\longrightarrow \mathcal{M}_z$ and therefore so does $\mathcal{M}^\alpha\circ \alpha^*(i_y)\circ \alpha^*({\eta'}_\beta)$.
\end{proof}

For the rest of this subsection, we suppose that $\mathcal{C}: \mathscr{X}\longrightarrow Coalg$ is a coalgebra representation such that each $\mathcal C_x$ is right semiperfect, i.e., $\mathbf M^{\mathcal C_x}$ has a set of projective generators.

  \begin{thm}\label{zzcsub}
The objects $\{\mathcal{N}_y\in \mathbf M^{\mathcal C_y}\}_{y\in Ob(\mathscr{X})}$ together determine a subobject $\mathcal N\subseteq \mathcal M$ in $Com^{cs}$-$\mathcal C$.

    \end{thm}

    \begin{proof}
Let $\alpha\in \mathscr{X}(y,z)$. Since $\alpha_*$ is a right adjoint, it preserves monomorphisms and it follows that $\alpha_*(i_z): \alpha_*\mathcal{N}_z\longrightarrow \alpha_*\mathcal{M}_z$ is a monomorphism in $\mathbf M^{\mathcal C_y}$. We will now show that the morphism $\mathcal{M}_\alpha: \mathcal{M}_y\longrightarrow \alpha_*\mathcal{M}_z$ restricts to a morphism $\mathcal{N}_\alpha: \mathcal{N}_y\longrightarrow \alpha_*\mathcal{N}_z$ giving us a commutative diagram 

\[\begin{tikzcd}
\mathcal{N}_y \arrow{r}{i_y} \arrow[swap]{d}{\mathcal{N}_\alpha} & \mathcal{M}_y \arrow{d}{\mathcal{M}_\alpha} \\
\alpha_*\mathcal{N}_z \arrow{r}{\alpha_*(i_z)} & \alpha_*\mathcal{M}_z
\end{tikzcd}
\]

Since $\mathcal C_y$ is right semiperfect, we can fix a set   $\{G_k\}_{k\in K}$ of projective generators for $\mathbf M^{\mathcal C_y}$. Using \cite[Lemma 3.2]{AB}, it suffices to show that for any $k\in K$ and any morphism $\zeta_k:G_k\longrightarrow \mathcal{N}_y$, there exists $\zeta'_k:G_k\longrightarrow \alpha_*\mathcal{N}_z$ such that $\alpha_*(i_z)\circ \zeta'_k= \mathcal{M}_\alpha\circ i_y\circ \zeta_k$. By \ref{defN}, we have an epimorphism
\begin{equation}
\bigoplus_{\beta \in \mathscr{X}(x,y)}\eta'_\beta: \bigoplus_{\beta\in \mathscr{X}(x,y)}\beta^*V\longrightarrow \mathcal{N}_y
\end{equation} in $\mathbf M^{\mathcal C_y}$. 
Since $G_k$ is projective, the morphism $\zeta_k:G_k\longrightarrow \mathcal{N}_y$ can be lifted to a morphism  $ {\zeta}_k'': G_k \longrightarrow \bigoplus_{\beta\in \mathscr{X}(x,y)}\beta^*V$ such that
\begin{equation}\label{eq1}
\zeta_k = \left(\bigoplus_{\beta \in \mathscr{X}(x,y)}{\eta'_\beta}\right)\circ \zeta_k''
  \end{equation}
  We know from Lemma \ref{lem1}, that $\mathcal{M}_\alpha\circ i_y\circ \eta_\beta'$ factors through $\alpha_*(i_z):\alpha_*\mathcal{N}_z\longrightarrow \alpha_*\mathcal{M}_z$ for each $\beta \in \mathscr{X}(x,y)$. It now follows from \eqref{eq1} that $\mathcal{M}_\alpha\circ i_y\circ \zeta_k$ factors through $\alpha_*(i_z)$ as required. 
\end{proof}

\begin{lem}\label{lem2}
  Let $\eta_1': V\longrightarrow \mathcal{N}_x$ be the canonical morphism corresponding to the identity map in $\mathscr{X}(x,x)$. Then, for any $y\in Ob(\mathscr{X})$, we have
  \begin{equation}
    \mathcal{N}_y= Im \left(\bigoplus_{\beta \in \mathscr{X}(x,y)}\beta^*V \xrightarrow{\beta^*\eta_1'}\beta^*\mathcal{N}_x\xrightarrow{\mathcal{N}^{\beta}}\mathcal{N}_y \right)
    \end{equation}

  \end{lem}
  \begin{proof}
Let $\beta \in \mathscr{X}(x,y)$. We consider the following commutative diagram:
\begin{equation}\label{commdiag}
\begin{tikzcd}
\beta^*V \arrow{r}{\beta^*\eta_1'} & \beta^*\mathcal{N}_x \arrow{r}{\mathcal{N}^\beta} \arrow[swap]{d}{\beta^* i_x} & \mathcal{N}_y \arrow{d}{i_y}\\
& \beta^*\mathcal{M}_x \arrow{r}{\mathcal{M}^\beta} & \mathcal{M}_y
\end{tikzcd}
\end{equation}
Since $i_x \circ \eta_1'=\eta$, we have $(\beta^*i_x) \circ (\beta^*\eta_1')=\beta^*\eta$. This gives
\begin{equation*}
Im\left(\mathcal{M}^\beta \circ (\beta^*\eta) \right)=Im\left(\mathcal{M}^\beta \circ  (\beta^*i_x) \circ (\beta^*\eta_1') \right)=Im \left(i_y \circ  \mathcal{N}^\beta \circ (\beta^*\eta_1') \right)=Im \left( \mathcal{N}^\beta \circ (\beta^*\eta_1') \right)
\end{equation*}
where the last equality follows from the fact that $i_y$ is a monomorphism. The result now follows directly from the definition in \eqref{defN}.
\end{proof}

\begin{lem}\label{lem2wsf}
  For any $y\in Ob(\mathscr{X})$, we have
  \begin{equation}
\mathcal{N}_y= \sum_{\beta\in \mathscr{X}(x,y)}Im\left(\beta^*\mathcal{N}_x\xrightarrow{\beta^*(i_x)}\beta^*\mathcal{M}_x\xrightarrow{\mathcal{M}^\beta}\mathcal{M}_y\right)
   \end{equation}
  \end{lem}
  \begin{proof}
   We set $\mathcal{N}_y'= \sum_{\beta\in \mathscr{X}(x,y)}Im\left(\beta^*\mathcal{N}_x\xrightarrow{\beta^*(i_x)}\beta_*\mathcal{M}_x\xrightarrow{\mathcal{M}^\beta}\mathcal{M}_y\right)$. We want to show that $\mathcal{N}'_y= \mathcal{N}_y$. It follows from the commutative diagram \ref{commdiag} that each of the morphisms $\beta^*\mathcal{N}_x\xrightarrow{\beta^*(i_x)}\beta^*\mathcal{M}_x\xrightarrow{\mathcal{M}^\beta}\mathcal{M}_y$ factors through the subcomodule $\mathcal{N}_y\subseteq \mathcal{M}_y$ and hence $\mathcal{N}'_y\subseteq \mathcal{N}_y$. Also, clearly
    \begin{equation}
Im\left(\beta^*V\xrightarrow{\beta^*\eta_1'}\beta^*\mathcal{N}_x\xrightarrow{\beta^*(i_x)}\beta_*\mathcal{M}_x\xrightarrow{\mathcal{M}^\beta}\mathcal{M}_y\right) \subseteq Im\left(\beta^*\mathcal{N}_x\xrightarrow{\beta^*(i_x)}\beta^*\mathcal{M}_x\xrightarrow{\mathcal{M}^\beta}\mathcal{M}_y\right)
\end{equation}
Using Lemma \ref{lem2}, we now  have $\mathcal{N}_y\subseteq \mathcal{N}_y'$ . This proves the result.
\end{proof}

For $\mathcal{M}\in Com^{cs}$-$\mathcal C$, we denote by $el_{\mathscr{X}}(\mathcal{M})$  the union $\cup_{x\in Ob(\mathscr{X})} \mathcal{M}_x$ as sets. The cardinality of $el_{\mathscr{X}}(\mathcal{M})$ will be denoted by $|\mathcal{M}|$. It is easy to see that for any quotient or subobject $\mathcal{N}$ of $\mathcal{M}\in Com^{cs}$-$\mathcal{C}$, we have $|\mathcal{N}|\leq |\mathcal{M}|$. Now, we set
\begin{equation}
\kappa:= sup\{\aleph_0, |Mor(\mathscr X)|, |K| \}
\end{equation}
We note in particular that $|\beta^*V'|\leq \kappa$ for any finite dimensional $\mathcal{C}_x$-comodule $V'$. 

\begin{lem}\label{cardinal}
Let $\mathcal N$ be as constructed in Proposition \ref{zzcsub}. Then, 
	we have $|\mathcal{N}|\leq \kappa$. 
\end{lem}
\begin{proof}
Let $y\in Ob(\mathscr{X})$. From Lemma \ref{lem2}, we have 
\begin{equation}
    \mathcal{N}_y= Im \left(\bigoplus_{\beta \in \mathscr{X}(x,y)}\beta^*V \xrightarrow{\beta^*\eta_1'}\beta^*\mathcal{N}_x\xrightarrow{\mathcal{N}^{\beta}}\mathcal{N}_y \right)
  \end{equation}
  In other words, $\mathcal{N}_y$ is an epimorphic image of $\bigoplus_{\beta \in \mathscr{X}(x,y)}\beta^*V$ and we have
  \begin{equation}
    |\mathcal{N}_y|\leq |\bigoplus_{\beta \in \mathscr{X}(x,y)}\beta^*V|\leq \kappa
  \end{equation}
  Thus, $|\mathcal{N}|=\sum_{y\in Ob(\mathscr{X})}|\mathcal{N}_y|\leq \kappa$.
\end{proof}

\begin{Thm}\label{Groth}
  Let $\mathcal{C}: \mathscr{X}\longrightarrow Coalg$ be a representation taking values in right semiperfect $K$-coalgebras. Then, the   category $Com^{cs}$-$\mathcal{C}$ of right cis-comodules over $\mathcal{C}$ is a Grothendieck category.
 \end{Thm}

 \begin{proof}
   Since finite limits and filtered colimits in $Com^{cs}$-$\mathcal{C}$ are both computed pointwise, it is clear that they commute in $Com^{cs}$-$\mathcal{C}$.
   
   \smallskip
   Now, let $\mathcal{M}$ be an object in $Com^{cs}$-$\mathcal{C}$ and $m \in el_{\mathscr{X}}(\mathcal{M})$. Then, there exists some $x\in Ob(\mathscr{X})$ such that $m \in \mathcal{M}_x$. By \cite[Corollary 2.2.9]{DNR}, there exists a finite dimensional right $\mathcal{C}_x$-comodule $V'$ and a  morphism $\eta': V'\longrightarrow \mathcal{M}_x$ in $\mathbf M^{\mathcal C_x}$ such that $m\in Im(\eta')$. Since $\mathcal{C}_x$ is semiperfect, we can choose a finite dimensional projective $V$ in $\mathbf M^{\mathcal C_x}$ along with an epimorphism  $V\longrightarrow V'$ in $\mathbf M^{\mathcal C_x}$ (see, for instance, \cite[Corollary 2.4.21]{DNR}). Therefore, we have an induced  morphism $\eta:V\longrightarrow \mathcal{M}_x$ in $\mathbf M^{\mathcal C_x}$ such that $m\in Im(\eta)$. We can now define the subobject $\mathcal{N}\subseteq \mathcal{M}$ in $Com^{cs}$-$\mathcal C$ corresponding to $\eta$ as in \eqref{defN}. From the definition of the canonical morphism ${\eta}'_1:V\longrightarrow \mathcal{N}_x$ induced by \eqref{defN}, it follows that $m\in \mathcal{N}_x$. By Lemma \ref{cardinal}, we also have $|\mathcal{N}|\leq \kappa$.
   
   \smallskip
   We now consider the set of isomorphism classes of objects in $Com^{cs}$-$\mathcal{C}$ having cardinality $\leq \kappa$. From the above, we see that any object in $Com^{cs}$-$\mathcal{C}$ may be expressed as a sum of such objects. This proves the result.
     \end{proof}

 \subsection{Trans-comodules over coalgebra representations}
 Let $C$ be a $K$-coalgebra. An object $M\in \mathbf{M}^C$ is said to be quasi-finite (see, for instance, \cite[$\S$ 12.5]{BW}) if the  tensor functor ${-}\otimes M: Vect\longrightarrow \mathbf{M}^C$ has a left adjoint. In that case, this left adjoint is denoted by $H_C(M,{-}): \mathbf{M}^C\longrightarrow Vect$. We note that the coalgebra $C$ is quasi-finite as a right $C$-comodule.
 
 \smallskip
Suppose that $\alpha: C\longrightarrow D$ is a coalgebra morphism such that $C$ is quasi-finite as a right $D$-comodule. Then, for any $N\in \mathbf M^D$, the space
$H_D(C,N)$ can be equipped with the structure of a right $C$-comodule. This determines a functor $H_D(C,{-}): \mathbf{M}^D\longrightarrow \mathbf{M}^C$ which is the left adjoint to the functor ${-}\square_CC:\mathbf M^C\longrightarrow \mathbf M^D$ (see, for instance, \cite[$\S$ 12.7]{BW}). However, ${-}\square_CC:\mathbf M^C\longrightarrow \mathbf M^D$ is merely the corestriction functor, i.e., we have
 $$\mathbf M^C(H_D(C,V),W)\cong \mathbf M^D(V, W\square_CC)\cong \mathbf M^D(V, \alpha^*(W))$$
 for any $V\in \mathbf{M}^D$ and $W\in \mathbf{M}^C$. In other words, if $C$ is quasi-finite as a right $D$-comodule, then the corestriction functor $\alpha^* :\mathbf{M}^C\longrightarrow \mathbf{M}^D$ has a left adjoint $\alpha^!:= H_D(C,{-}):  \mathbf{M}^D\longrightarrow \mathbf{M}^C$.

 \begin{defn}\label{trans-rep}

  Let $\mathcal{C}: \mathscr{X}\longrightarrow Coalg$ be a coalgebra representation. Suppose that $\mathcal C$ is quasi-finite, i.e., for any morphism $\alpha:x\longrightarrow y$
  in $\mathscr X$, the corresponding morphism $\mathcal C_\alpha:\mathcal C_x\longrightarrow \mathcal C_y$ of coalgebras makes $\mathcal C_x$ quasi-finite as a right 
  $\mathcal C_y$-comodule. A (right) trans-comodule $\mathcal{M}$ over $\mathcal{C}$ will consist of the following data:
  \begin{itemize}
  \item[(1)] For each object $x\in Ob(\mathscr{X})$, a right $\mathcal{C}_x$-comodule $\mathcal{M}_x$
  \item[(2)] For each morphism
    $\alpha: x \longrightarrow y$ in $\mathscr{X}$, a morphism
    $^{\alpha}{\mathcal{M}}:  =\alpha^!\mathcal{M}_y=H_{C_y}(C_x,\mathcal{M}_y)\longrightarrow \mathcal{M}_x$ of right $\mathcal{C}_x$-comodules (equivalently, a morphism $_{\alpha}\mathcal{M}: \mathcal{M}_y\longrightarrow \alpha^*\mathcal{M}_x$ of right $\mathcal{C}_y$-comodules)
  \end{itemize}

  We further assume that $^{id_x}\mathcal{M} = id_{\mathcal{M}_x}$ and for any pair of composable morphisms
  $x\xrightarrow{\alpha} y\xrightarrow{\beta}z$ in $\mathscr{X}$, we have $ ^{\alpha}\mathcal{M}\circ\alpha^!(^{\beta}\mathcal{M})={ ^{\beta\alpha}}\mathcal{M}:(\beta\alpha)^!\mathcal{M}_z=\alpha^!\beta^!\mathcal{M}_z\longrightarrow \alpha^!\mathcal{M}_y\longrightarrow \mathcal{M}_x $. The latter condition can be equivalently expressed as $_{\beta\alpha}\mathcal{M}= \beta^*(_{\alpha}\mathcal{M})\circ {_{\beta}\mathcal{M}}$.\\
  A morphism $\eta: \mathcal{M} \longrightarrow \mathcal{N}$ of trans-comodules over $\mathcal{C}$ consists of morphisms $\eta_x: \mathcal{M}_x\longrightarrow\mathcal{N}_x$ of right $\mathcal{C}_x$-comodules for each $x\in Ob(\mathscr{X})$ such that for each morphism $x\longrightarrow y$ in $\mathscr{X}$ the following diagram commutes

  \[\begin{tikzcd}
 \alpha^!\mathcal{M}_y \arrow{r}{\alpha^!\eta_y} \arrow[swap]{d}{^\alpha\mathcal{M}} &  \alpha^!\mathcal{N}_y \arrow{d}{^\alpha\mathcal{N}} \\
\mathcal{M}_x \arrow{r}{\eta_x} & \mathcal{N}_x
\end{tikzcd}
\]
We denote the category of right trans-comodules by $Com^{tr}$-$\mathcal{C}$. Similarly, we may define the category $\mathcal C$-$Com^{tr}$ of left trans-comodules over $\mathcal C$. 
\end{defn}

We remark that if $\mathcal{C'}: \mathscr{X}\longrightarrow Coalg$ is any  coalgebra representation (not necessarily quasi-finite), we can still define the category $Com^{tr}$-$\mathcal{C'}$ of trans-comodules over $\mathcal C'$ by considering only the set of maps $\{_{\alpha}\mathcal{M}\}_{\alpha
\in Mor(\mathscr X)}$ in Definition \ref{trans-rep} satisfying $_{\beta\alpha}\mathcal{M}= \beta^*(_{\alpha}\mathcal{M})\circ {_{\beta}\mathcal{M}}$ for composable morphisms $\alpha$, $\beta$ in $\mathscr X$. However, the assumption of quasi-finiteness will be necessary to establish most of the properties of trans-comodules that we study in this paper.

\smallskip
Let $\mathcal{C}: \mathscr{X}\longrightarrow Coalg$ be a coalgebra representation that is quasi-finite.
As in the proof of Proposition \ref{Ab}, it follows  from the exactness of the corestriction functor $\alpha^*$ that $Com^{tr}$-$\mathcal{C}$ is an abelian category.  We will now study generators for the category $Com^{tr}$-$\mathcal{C}$. For this, we will suppose that each $\mathcal C_x$ is a right semiperfect coalgebra. Accordingly,  let $\mathcal{M}\in Com^{tr}$-$\mathcal C$. We consider an object $x\in Ob(\mathscr{X})$ and a morphism
\begin{equation}
  \eta: V\longrightarrow  \mathcal{M}_x
\end{equation}
in $\mathbf M^{\mathcal C_x}$, where $V$ is a finite dimensional projective right comodule over $\mathcal{C}_x$. For each $y\in Ob(\mathscr{X})$, we set $\mathcal{N}_y\subseteq \mathcal{M}_y$ to be the image of the family of maps

\begin{align}\label{defN1}
  \mathcal{N}_y&= Im(\bigoplus_{\beta \in \mathscr{X}(y,x)}\beta^!V \xrightarrow{\beta^!\eta}\beta^!\mathcal{M}_x\xrightarrow{^{\beta}\mathcal{M}}\mathcal{M}_y) &= \sum_{\beta \in \mathscr{X}(y,x)}Im(\beta^!V \xrightarrow{\beta^!\eta}\beta^!\mathcal{M}_x\xrightarrow{^{\beta}\mathcal{M}}\mathcal{M}_y)
\end{align}
Let $i_y$ denote the inclusion $\mathcal{N}_y\hookrightarrow \mathcal{M}_y$ and for each $\beta\in \mathscr{X}(y,x)$, we denote by $\eta'_\beta:\beta^!V\longrightarrow \mathcal{N}_y$ the canonical morphism induced from \eqref{defN1}. In a manner similar to the proof of Lemma \ref{lem1}, we can show that for any $\alpha \in \mathscr{X}(z,y)$, the composition
 \begin{equation}
    \beta^!V\xrightarrow{\eta'_\beta} \mathcal{N}_y\xrightarrow{i_y}\mathcal{M}_y\xrightarrow{_\alpha \mathcal{M}}\alpha^*\mathcal{M}_z
  \end{equation}
factors through $\alpha^*(i_z):\alpha^*\mathcal{N}_z\longrightarrow \alpha^*\mathcal{M}_z$. Since each $\mathcal C_x$ is right semiperfect, we can prove the following result, the proof of which is similar to that of Proposition \ref{zzcsub}.

 \begin{thm}\label{trans-sub}
The objects $\{\mathcal{N}_y\in \mathbf M^{\mathcal C_y}\}_{y\in Ob(\mathscr{X})}$ together determine a subobject $\mathcal N\subseteq \mathcal M$ in $Com^{tr}$-$\mathcal C$.

    \end{thm}

We also record here the following two equalities, which can be proved in a manner similar to Lemma \ref{lem2} and Lemma \ref{lem2wsf}
\begin{equation}\label{equalz}
 \mathcal{N}_y= Im \left(\bigoplus_{\beta \in \mathscr{X}(y,x)}\beta^!V \xrightarrow{\beta^!\eta_1'}\beta^!\mathcal{N}_x\xrightarrow{^{\beta}\mathcal{N}}\mathcal{N}_y \right)\sum_{\beta\in \mathscr{X}(y,x)}Im\left(\beta^!\mathcal{N}_x\xrightarrow{\beta^!(i_x)}\beta^!\mathcal{M}_x\xrightarrow{^\beta\mathcal{M}}\mathcal{M}_y\right)
\end{equation}
for any $y\in Ob(\mathscr X)$. Here,   $\eta_1': V\longrightarrow \mathcal{N}_x$ is the canonical morphism corresponding to the identity map in $\mathscr{X}(x,x)$.

\begin{lem}\label{shriekp} Let $\alpha:C\longrightarrow D$ be a morphism of coalgebras such that $C$ is quasi-finite as a right $D$-comodule. Let $V\in \mathbf M^D$ be a finite dimensional comodule. Then, $\alpha^!V\in \mathbf M^C$ is finite dimensional. 
\end{lem}

\begin{proof}
Since every comodule is a colimit of its finite dimensional subcomodules, a $C$-comodule $U$ is finitely generated as an object of $\mathbf M^C$ (i.e., $\mathbf M^C(U,\_\_):\mathbf M^C
\longrightarrow Vect$ preserves filtered colimits of systems of monomorphisms) if and only if $U$ is finite dimensional. Suppose that $\{W_i\}_{i\in I}$ is a filtered system of comodules in $\mathbf M^C$ connected by monomorphisms and let $W=\underset{i\in I}{\varinjlim}\textrm{ }W_i$. Since $\alpha^*:\mathbf M^C\longrightarrow \mathbf M^D$ preserves all colimits and all finite limits, we note
\begin{equation}\label{2.22dp}
\underset{i\in I}{\varinjlim}\textrm{ }\mathbf M^C(\alpha^!V,W_i)=\underset{i\in I}{\varinjlim}\textrm{ }\mathbf M^D(V,\alpha^* W_i)=\mathbf M^D(V,\underset{i\in I}{\varinjlim}\textrm{ }\alpha^*W_i)=\mathbf M^D(V,\alpha^*W)=\mathbf M^C(\alpha^!V,W)
\end{equation} It follows from \eqref{2.22dp} that $\alpha^!V\in \mathbf M^C$ is finitely generated as an object in $\mathbf M^C$, i.e., it is finite dimensional as a $K$-vector space.
\end{proof}

For $\mathcal{M}\in Com^{tr}$-$\mathcal{C}$, let $el_{\mathscr{X}}(\mathcal{M})$ denote the union $\cup_{x\in Ob(\mathscr{X})} \mathcal{M}_x$. The cardinality of $el_{\mathscr{X}}(\mathcal{M})$ will be denoted by $|\mathcal{M}|$. It is easy to see that for any quotient or subobject $\mathcal{N}$ of $\mathcal{M}\in Com^{tr}$-$\mathcal{C}$, we have $|\mathcal{N}|\leq |\mathcal{M}|$. Now, we set
\begin{equation}\label{crd222}
\kappa:= sup\{\aleph_0, |K|, |Mor(\mathscr X)|\}
\end{equation}

\begin{lem}\label{cardinal1}
Let $\mathcal N$ be as constructed in Proposition \ref{trans-sub}. Then, 
	we have $|\mathcal{N}|\leq \kappa$. 

\end{lem}
\begin{proof} We consider some $\beta\in \mathscr X(y,x)$. Since $V$ is a finite dimensional projective in $\mathbf M^{\mathcal C_x}$, it follows from Lemma \ref{shriekp} that
$\beta^!V\in \mathbf M^{\mathcal C_y}$ is finite dimensional.  It is now clear from \eqref{crd222} that $|\beta^!V|\leq \kappa$. 

\smallskip
From \eqref{equalz}, we now have
\begin{equation}
    \mathcal{N}_y= Im \left(\bigoplus_{\beta \in \mathscr{X}(y,x)}\beta^!V \xrightarrow{\beta^!\eta_1'}\beta^!\mathcal{N}_x\xrightarrow{^{\beta}\mathcal{N}}\mathcal{N}_y \right)
  \end{equation}
  Since $\mathcal{N}_y$ is an epimorphic image of $\bigoplus_{\beta \in \mathscr{X}(y,x)}\beta^!V$, we see that
  \begin{equation}
    |\mathcal{N}_y|\leq |\bigoplus_{\beta \in \mathscr{X}(y,x)}\beta^!V|\leq \kappa
  \end{equation}
  Thus, $|\mathcal{N}|=\sum_{y\in Ob(\mathscr{X})}|\mathcal{N}_y|\leq \kappa$.
\end{proof}

\begin{Thm}\label{Groth1}
  Let $\mathcal{C}: \mathscr{X}\longrightarrow Coalg$ be a representation taking values in right semiperfect $K$-coalgebras. Suppose that  $\mathcal C$ is quasi-finite, i.e., for each morphism $\alpha:z\longrightarrow y$ in $\mathscr X$, the induced morphism $\mathcal C_\alpha:\mathcal C_z
  \longrightarrow\mathcal C_y$ of coalgebras makes $\mathcal C_z$ a quasi-finite right $\mathcal C_y$-comodule. Then, the   category $Com^{tr}$-$\mathcal{C}$ of right trans-comodules over $\mathcal{C}$ is a Grothendieck category.
 \end{Thm}

 \begin{proof}
   Since filtered colimits and finite limits in $Com^{tr}$-$\mathcal{C}$ are computed pointwise, it is clear that they commute in $Com^{tr}$-$\mathcal{C}$. Now, let $\mathcal{M}$ be an object in $Com^{tr}$-$\mathcal{C}$ and $m \in el_{\mathscr{X}}(\mathcal{M})$. Then, there exists some $x\in Ob(\mathscr{X})$ such that $m \in \mathcal{M}_x$. By \cite[Corollary 2.2.9]{DNR}, there exists a finite dimensional right $\mathcal{C}_x$-comodule $V'$ and a  morphism $\eta': V'\longrightarrow \mathcal{M}_x$ in $\mathbf M^{\mathcal C_x}$ such that $m\in Im(\eta')$. Since $\mathcal{C}_x$ is semiperfect, we can choose a finite dimensional projective $V$ in $\mathbf M^{\mathcal C_x}$ along with an epimorphism  $V\longrightarrow V'$ in $\mathbf M^{\mathcal C_x}$. Therefore, we have an induced morphism $\eta:V\longrightarrow \mathcal{M}_x$ in $\mathbf M^{\mathcal C_x}$ such that $m\in Im(\eta)$. We can now define the subcomodule $\mathcal{N}\subseteq \mathcal{M}$ corresponding to $\eta$ as in \eqref{defN1}. From the definition of the canonical morphism ${\eta_1}':V\longrightarrow \mathcal{N}_x$ in \eqref{equalz}, it follows that $m\in \mathcal{N}_x$. By Lemma \ref{cardinal1}, we also have $|\mathcal{N}|\leq \kappa$. It follows that isomorphism classes of  objects in $Com^{tr}$-$\mathcal{C}$ having cardinality $\leq \kappa$ give a set of generators for $Com^{tr}$-$\mathcal{C}$. \\
 \end{proof}

 \section{Coalgebra representations of a poset and projective generators for comodules}

 Throughout this section, we assume that $\mathscr{X}$ is a partially ordered set. We suppose that $\mathcal{C}: \mathscr{X}\longrightarrow Coalg$ is a coalgebra representation such that $\mathcal C_x$ is right semiperfect for each $x\in Ob(\mathscr X)$. Our objective is to show that under these conditions, both $Com^{cs}$-$\mathcal{C}$ and $Com^{tr}$-$\mathcal{C}$ have projective generators. 

\subsection{Projective generators for cis-comodules}
 \begin{thm}\label{adjoint}
   Let $x\in Ob(\mathscr{X})$. Then, 
   \begin{itemize}
\item[(1)] There is a functor $ex_x^{cs}:\mathbf{M}^{\mathcal{C}_x}\longrightarrow Com^{cs}$-$\mathcal{C}$ defined by setting, for any $y\in Ob(\mathscr X)$:
   \[
     ex_x^{cs}(M)_y=
     \begin{cases}
       \alpha^*(M) &\quad \text{if~} \alpha\in \mathscr{X}(x,y)\\
       0&\quad \text{if~}\mathscr{X}(x,y)=\emptyset\\
     \end{cases}
   \]
 \item[(2)] The evaluation at $x$, i.e., $ev_x^{cs}:Com^{cs}$-$\mathcal{C}\longrightarrow \mathbf{M}^{\mathcal{C}_x}$, $\mathcal{M}\mapsto \mathcal{M}_x$ is an exact functor.
   \item[(3)] $(ex_x^{cs}, ev_x^{cs})$ is a pair of adjoint functors.
     \end{itemize}

\end{thm}

\begin{proof}
(1) Clearly, $ex_x^{cs}(M)_y\in \mathbf{M}^{\mathcal{C}_y}$ for each $y\in Ob(\mathscr{X})$. We consider $\beta:y\longrightarrow y'$. If $x\not\leq y$, then $ex_x^{cs}(M)^\beta=0$. Otherwise, if we have $\alpha:x\longrightarrow y$ and $\gamma:x\longrightarrow y'$, i.e., $\beta\alpha=\gamma$, we note that $id=ex_x^{cs}(M)^\beta:\beta^*\alpha^*(M)\longrightarrow \gamma^*(M)$. It is now clear that $ ex_x^{cs}(M)\in Com^{cs}$-$\mathcal{C}$. 

\smallskip

(2) This follows from the fact that finite limits and finite colimits in $Com^{cs}$-$\mathcal{C}$ are computed pointwise.

\smallskip
(3) Let $\mathcal{M}\in Com^{cs}$-$\mathcal{C}$ and $N\in \mathbf{M}^{\mathcal{C}_x}$. We will show that
 $
Com^{cs}\text{-}\mathcal{C}(ex_x^{cs}(N),\mathcal{M})\cong \mathbf{M}^{\mathcal{C}_x}(N,ev_x^{cs}(\mathcal{M}))
$. We consider  $f:N\longrightarrow \mathcal{M}_x$  in $\mathbf{M}^{\mathcal{C}_x}$. Corresponding to $f$, we now have $\eta^f:ex_x^{cs}(N)\longrightarrow \mathcal M$ defined by setting 
\begin{equation}
\begin{CD}\eta^f_y:ex_x^{cs}(N)_y=\alpha^*N@>\alpha^*f>> \alpha^*\mathcal M_x@>\mathcal M^\alpha>> \mathcal M_y\\
\end{CD}
\end{equation} for $y\in Ob(\mathscr X)$ whenever $\alpha\in \mathscr X(x,y)$ and $\eta^f_y=0$ if $x\not \leq y$. We now take $\beta\in \mathscr X(y,y')$ and claim that $\mathcal M^\beta\circ \beta^*(\eta^f_y)=\eta^f_{y'}\circ ex_x^{cs}(N)^\beta$. If $x\not\leq y$, both sides of this equality vanish. Otherwise, if $\alpha\in \mathscr X(x,y)$, we have the commutative diagram 
\begin{equation}
\begin{CD}
\beta^*ex_x^{cs}(N)_y =\beta^*\alpha^*(N)@>\beta^*(\eta^f_y)>\beta^*(\mathcal M^\alpha)\circ \beta^*(\alpha^*f) > \beta^*\mathcal M_y\\
@Vex_x^{cs}(N)^\beta VidV @V\mathcal M^\beta VV\\
ex_x^{cs}(N)_{y'}= \beta^*\alpha^*(N)@>\eta^f_{y'}>\mathcal M^{\beta\alpha}\circ (\beta^*\alpha^*f)> \mathcal M_{y'}
\end{CD}
\end{equation} which shows that $\eta^f$ is a morphism in  $Com^{cs}$-$\mathcal{C}$. Conversely, given a morphism $\eta:ex_x(N)\longrightarrow \mathcal M$ in  $Com^{cs}$-$\mathcal{C}$, we obtain in particular a morphism $f^\eta:ex_x^{cs}(N)_x=N\longrightarrow \mathcal M_x$ in $\mathbf M^{\mathcal C_x}$. We may  verify directly that these two associations are inverse to each other.

\end{proof}

We observe that the functor $ev_x^{cs}:Com^{cs}\text{-}\mathcal{C} \longrightarrow \mathbf{M}^{\mathcal{C}_x}$  also  has a right adjoint  which we describe below.

\begin{thm}\label{pr3.2py}
 Let $x\in Ob(\mathscr{X})$. Then,

\smallskip
(1) There is a functor $coe_x^{cs}: \mathbf{M}^{\mathcal{C}_x} \longrightarrow Com^{cs}\text{-}\mathcal{C}$ defined by setting, for any $y\in Ob(\mathscr X)$:
\begin{equation*}
coe_x^{cs}(N)_y:=\begin{cases}
\alpha_*N \quad\qquad \text{if}~ \alpha \in \mathscr{X}(y,x)\\
0 \quad \quad \quad \qquad \text{if}~ \mathscr{X}(y,x)=\emptyset
\end{cases}
\end{equation*}

(2) $(ev_x^{cs},coe_x^{cs})$ is a pair of adjoint functors.
\end{thm}
\begin{proof}

 Clearly, $coe_x^{cs}(N)_y\in \mathbf{M}^{\mathcal{C}_y}$ for each $y\in Ob(\mathscr{X})$. We consider $\beta:y'\longrightarrow y$. If $y\not\leq x$, then $ex_x^{cs}(N)_\beta=0$. Otherwise, if we have $\alpha:y\longrightarrow x$ and $\gamma:y'\longrightarrow x$, i.e., $\alpha\beta=\gamma$, we note that $id=coe_x^{cs}(N)_\beta:\gamma_*(N)\longrightarrow \beta_*\alpha_*(N)$. It is now clear that $ coe_x^{cs}(N)\in Com^{cs}$-$\mathcal{C}$. This proves (1). The adjunction in (2) may be verified directly in a manner similar to the proof of Proposition \ref{adjoint}.

\end{proof}

\begin{cor}\label{proj}
Let $\mathscr{X}$ be a poset and $\mathcal{C}: \mathscr{X}\longrightarrow Coalg$ be a representation taking values in right semiperfect $K$-coalgebras. Let $x\in Ob(\mathscr{X})$. Then, the functor $ex_x^{cs}: \mathbf M^{\mathcal{C}_x}\longrightarrow Com^{cs}$-$\mathcal{C}$ preserves projectives.
\end{cor}

\begin{proof}
We know from Proposition \ref{adjoint}(2) that $ev_x^{cs}$ is an exact functor. Since $(ex_x^{cs}, ev_x^{cs})$ is an adjoint pair by Proposition \ref{adjoint}(3), it follows that the left adjoint $ex_x^{cs}$ preserves projective objects.

\end{proof}

\begin{Thm}\label{projgen}
Let $\mathscr{X}$ be a poset and $\mathcal{C}: \mathscr{X}\longrightarrow Coalg$ be a representation taking values in right semiperfect $K$-coalgebras. Then, $Com^{cs}$-$\mathcal{C}$ has a set of projective generators.
  \end{Thm}
  \begin{proof}
    Let $Proj^f(\mathcal{C}_x)$ denote the set of isomorphism classes of finite dimensonal projective $\mathcal{C}_x$-comodules. Since $\mathcal C_x$ is right semiperfect, we know (see \cite[Corollary 2.4.21]{DNR}) that   $Proj^f(\mathcal{C}_x)$ is a generating set for $\mathbf M^{\mathcal C_x}$. For any $V\in Proj^f(\mathcal{C}_x)$, it follows from 
    Corollary \ref{proj} that  $ex_x^{cs}(V)$ is also projective in $Com^{cs}$-$\mathcal{C}$. We will show that the family 
    \begin{equation}
\mathcal{G}= \{ex_x^{cs}(V)~|~x\in Ob(\mathscr{X}), V\in Proj^f(\mathcal{C}_x)\}
\end{equation}
is a set of projective generators for $Com^{cs}$-$\mathcal{C}$. For this, we consider a  non-invertible monomorphism $i:\mathcal{N}\hookrightarrow \mathcal{M}$ in $Com^{cs}$-$\mathcal{C}$.  Since kernels and cokernels in  $Com^{cs}$-$\mathcal{C}$ are constructed pointwise, there exists some $x\in Ob(\mathscr{X})$ such that $i_x:\mathcal{N}_x\hookrightarrow \mathcal{M}_x$ is a non-invertible monomorphism in $\mathbf M^{\mathcal C_x}$. Since  $Proj^f(\mathcal{C}_x)$ is a generating set of $\mathbf M^{\mathcal C_x}$, we can choose a morphism $f:V\longrightarrow \mathcal{M}_x$ with $V\in Proj^f(\mathcal{C}_x)$ such that $f$ does not factor through $i_x:\mathcal{N}_x\hookrightarrow \mathcal{M}_x$. Since $(ex_x^{cs}, ev_x^{cs})$ is  an adjoint pair, this gives us a morphism $\eta^f: ex_x^{cs}(V)\longrightarrow \mathcal{M}$ in $Com^{cs}$-$\mathcal{C}$ corresponding to $f$, which does not factor through $i:\mathcal{N}\longrightarrow \mathcal{M}$. It  follows from \cite[\S 1.9]{Tohuku} that the family $\mathcal{G}$ is a set of generators for $Com^{cs}$-$\mathcal{C}$.
\end{proof}

\subsection{Projective generators for trans-comodules}
In this subsection, we will show that the category $Com^{tr}\text{-}\mathcal{C}$ of trans-comodules over a quasi-finite representation $\mathcal{C}:\mathscr{X} \longrightarrow Coalg$ has projective generators.

\begin{thm}\label{poprojc}
Let $\mathscr{X}$ be a poset and $\mathcal{C}: \mathscr{X}\longrightarrow Coalg$ be a quasi-finite representation taking values in   right semiperfect $K$-coalgebras. Let $x\in Ob(\mathscr{X})$. Then, 
  
\smallskip
(1) There is a functor $ex^{tr}_x: \mathbf{M}^{\mathcal{C}_x} \longrightarrow Com^{tr}\text{-}\mathcal{C}$ given by setting, for any $y\in Ob(\mathscr X)$:
\begin{equation*}
ex^{tr}_x(M)_y:=\begin{cases}
\alpha^!M=H_{\mathcal{C}_x}(\mathcal{C}_y,M) \quad \text{if}~ \alpha \in \mathscr{X}(y,x)\\
0 \quad \quad \quad \quad \quad \quad \quad \quad \quad \text{if}~ \mathscr{X}(y,x)=\emptyset
\end{cases}
\end{equation*}

\smallskip
(2) The evaluation at $x$, i.e., $ev^{tr}_x:Com^{tr}$-$\mathcal{C}\longrightarrow \mathbf{M}^{\mathcal{C}_x}$, $\mathcal{M}\mapsto \mathcal{M}_x$ is an exact functor.

\smallskip
(3) $(ex^{tr}_x,ev^{tr}_x)$ is a pair of adjoint functors.
\end{thm}
\begin{proof}
(1) Since $\mathcal{C}:\mathscr{X} \longrightarrow Coalg$ is a quasi-finite representation, it follows that $ex^{tr}_x(M)_y=
\alpha^!M=H_{\mathcal{C}_x}(\mathcal{C}_y,M)$ is a right $\mathcal{C}_y$-comodule for each $y \in Ob(\mathscr{X})$ and $\alpha \in \mathscr{X}(y,x)$. 
We consider a morphism $\beta:y' \longrightarrow y$ in $\mathscr X$. If $y\not\leq x$, then ${^\beta}ex_x^{tr}(M)=0$. Otherwise, we have  $\alpha:y\longrightarrow x$ and $\gamma:y'\longrightarrow x$, i.e., $\alpha \beta=\gamma$, which gives $id={^\beta}{ex_x^{tr}(M)}:\beta^!ex_x^{tr}(M)_y=\beta^!\alpha^!(M)\longrightarrow \gamma^!(M)=ex_x^{tr}(M)_{y'}$. It follows that $ex^{tr}_x(M) \in Com^{tr}\text{-}\mathcal{C}$.

\smallskip
(2) Since finite limits and finite colimits in $Com^{tr}\text{-}\mathcal{C}$ are computed pointwise, it follows that the functor $ev^{tr}_x$ is exact.

\smallskip
(3) We will show that there is an isomorphism
$
Com^{tr}\text{-}\mathcal{C}\left(ex^{tr}_x(M), \mathcal{N} \right) \simeq \textbf{M}^{\mathcal{C}_x}(M,ev^{tr}_x(\mathcal{N}))
$
for any $M \in \textbf{M}^{\mathcal{C}_x}$ and $\mathcal{N} \in Com^{tr}\text{-}\mathcal{C}$. We start with $f:M \longrightarrow \mathcal{N}_x$   in $\textbf{M}^{\mathcal{C}_x}$. For each $y \in Ob(\mathscr{X})$ and $\alpha \in \mathscr{X}(y,x)$, we set $\eta^f_y: ex^{tr}_x(M)_y=\alpha^!M \longrightarrow \mathcal{N}_y$ to be the composition
\begin{equation*}
\begin{tikzcd}
\alpha^!M \arrow{r}{\alpha^!f} &\alpha^!\mathcal{N}_x \arrow{r}{{^\alpha}\mathcal{N}} &\mathcal{N}_y
\end{tikzcd}
\end{equation*}
Clearly, each $\eta_y^f$ is a morphism of right $\mathcal{C}_y$-comodules. We now take $\beta:y'\longrightarrow y$ and claim that $\eta^f_{y'}\circ {^\beta}ex_x^{tr}(M) ={^\beta}\mathcal N\circ \beta^!\eta^f_y$. If $y\not\leq x$, then both sides vanish. Otherwise, if $\alpha:y\longrightarrow x$, we have a commutative diagram
\begin{equation}
\begin{CD}
\beta^!ex_x^{tr}(M)_y=\beta^!\alpha^!(M) @>\beta^!\eta^f_y>\beta^!({^\alpha}\mathcal{N})\circ \beta^!(\alpha^!f)> \beta^!\mathcal N_y\\
@V{^\beta}ex_x^{tr}(M)VidV @VV{^\beta}\mathcal NV \\
ex_x^{tr}(M)_{y'} =\beta^!\alpha^!(M)@>\eta^f_{y'}>{^{\alpha\beta}}\mathcal N\circ \beta^!\alpha^!(f)> \mathcal N_{y'}\\
\end{CD}
\end{equation} which shows that $\eta^f:ex^{tr}_x(M)\longrightarrow \mathcal N$ is a morphism in $ Com^{tr}\text{-}\mathcal{C}$. On the other hand, if $\eta:ex^{tr}_x(M)\longrightarrow \mathcal N$ is a morphism in $ Com^{tr}\text{-}\mathcal{C}$, we have in particular a morphism $f^\eta:M\longrightarrow \mathcal N_x$ in $\mathbf M^{\mathcal C_x}$. It is easily seen that these associations are inverse to each other. This proves the result. 
\end{proof}

\begin{cor}\label{presproj}
Let $\mathscr{X}$ be a poset and $\mathcal{C}: \mathscr{X}\longrightarrow Coalg$ be a quasi-finite representation   taking values in   right semiperfect $K$-coalgebras. Then, for each $x \in Ob(\mathscr{X})$, the functor $ex^{tr}_x: \mathbf{M}^{\mathcal{C}_x}\longrightarrow Com^{tr}$-$\mathcal{C}$ preserves projectives.
\end{cor}
\begin{proof}
The result follows from the fact that the functor $ex^{tr}_x$ is left adjoint to the exact functor $ev^{tr}_x$.
\end{proof}

In a manner similar to Proposition \ref{pr3.2py}, we can show that the functor $ev_x^{tr}$ also has a right adjoint.

\begin{thm}\label{poprojk}
  Let $\mathscr{X}$ be a poset and $\mathcal{C}: \mathscr{X}\longrightarrow Coalg$ be a quasi-finite representation  taking values in   right semiperfect $K$-coalgebras. For each  $x\in Ob(\mathscr{X})$, the functor $ev_x^{tr}:Com^{tr}\text{-}\mathcal{C}\longrightarrow \mathbf M^{\mathcal C_x}$ has a right adjoint $coe^{tr}_x: \mathbf{M}^{\mathcal{C}_x} \longrightarrow Com^{tr}\text{-}\mathcal{C}$ given by setting, for $y\in Ob(\mathscr X)$:
\begin{equation*}
coe^{tr}_x(M)_y:=\begin{cases}
\alpha^*M \quad \text{if}~ \alpha \in \mathscr{X}(x,y)\\
0 \quad \quad  \text{if}~ \mathscr{X}(x,y)=\emptyset
\end{cases}
\end{equation*}
\end{thm}

\begin{Thm}\label{projgen3.8}
Let $\mathscr{X}$ be a poset and $\mathcal{C}: \mathscr{X}\longrightarrow Coalg$ be a quasi-finite representation taking values in  right semiperfect $K$-coalgebras. Then, the category $Com^{tr}$-$\mathcal{C}$ has a set of projective generators.
\end{Thm}
\begin{proof} 
Since $\mathcal C_x$ is semiperfect for each $x \in Ob(\mathscr{X})$, we know that $Proj^f(\mathcal{C}_x)$ is a generating set for the category $\mathbf{M}^{\mathcal{C}_x}$. By Corollary \ref{presproj}, we also know that $ex^{tr}_x(V)$ is projective in $Com^{tr}$-$\mathcal{C}$ for any projective object $V \in \textbf{M}^{\mathcal{C}_x}$. Since $(ex^{tr}_x, ev^{tr}_x)$ is an adjoint pair, we can now show as in the proof of Theorem \ref{projgen} that 
\begin{equation*}
\mathcal{G}= \{ex^{tr}_x(V)~|~x\in Ob(\mathscr{X}), V\in Proj^f(\mathcal{C}_x)\}
\end{equation*}
is a set of generators for $Com^{tr}$-$\mathcal{C}$. 
\end{proof}

\section{Cartesian comodules over coalgebra representations}

\subsection{Cartesian cis-comodules over coalgebra representations}
We recall that a coalgebra morphism $C\xrightarrow{\alpha} D$ is said to be (left) coflat if the coinduction functor $\alpha_*: \mathcal{M}^D\longrightarrow \mathcal{M}^C$, $M\mapsto M\square_DC$ is exact, i.e., $C$ is coflat as a left $D$-comodule. We will say that a  representation $\mathcal{C}: \mathscr{X}\longrightarrow Coalg$  is left coflat if for each $\alpha: x\longrightarrow y$ in $\mathscr{X}$, we have a left coflat morphism $\mathcal C_\alpha: \mathcal{C}_x\longrightarrow \mathcal{C}_y$ of coalgebras. We will now introduce the category of cartesian cis-comodules over $\mathcal{C}$.

\begin{defn}
  Let $\mathcal{C}: \mathscr{X}\longrightarrow Coalg$ be a left coflat representation taking values in semiperfect $K$-coalgebras. Let $\mathcal{M}\in Com^{cs}$-$\mathcal{C}$. We will say that $\mathcal{M}$ is cartesian if for each $\alpha:x\longrightarrow y$ in $\mathscr{X}$, the morphism $\mathcal{M}_\alpha:\mathcal{M}_x\longrightarrow \alpha_*\mathcal{M}_y= \mathcal{M}_y\square_{\mathcal{C}_y}\mathcal{C}_x$ in $\mathbf{M}^{\mathcal{C}_x}$ is an isomorphism. We let $Com^{cs}_c$-$\mathcal{C}$ denote the full subcategory of $Com^{cs}$-$\mathcal{C}$ consisting of cartesian comodules.
\end{defn}

\begin{lem}\label{cocomplete}
  The category $Com^{cs}_c$-$\mathcal{C}$ is a cocomplete abelian category. Explicitly, the colimit $\mathcal{M}$  of a family $\{\mathcal{M}_i\}_{i\in I}$ of objects in $Com^{cs}_c$-$\mathcal{C}$ is given by setting
  \begin{equation}\label{4.1dpr}
\mathcal{M}_x:= {colim}_{i\in I} {~\mathcal{M}_i}_x\qquad \text{for each } x\in  Ob(\mathscr{X})
    \end{equation}
where the colimit on the right hand side is taken in the category $\mathbf{M}^{\mathcal{C}_x}$.
  \end{lem}

  \begin{proof} Let $\eta: \mathcal{M}\longrightarrow \mathcal{N}$ be a morphism in $Com^{cs}_c$-$\mathcal{C}$. Then, $Ker(\eta), Coker(\eta)\in Com^{cs}$-$\mathcal{C}$ are given by $Ker(\eta)_x= Ker(\eta_x:\mathcal{M}_x\longrightarrow \mathcal{N}_x)$ and $Coker(\eta)_x= Coker(\eta_x:\mathcal{M}_x\longrightarrow \mathcal{N}_x)$ for each $x\in Ob(\mathscr{X})$. Since $\mathcal{C}$ is a coflat representation, for any $\alpha\in \mathscr X(x,y)$, the coinduction functor $\alpha_*$ is exact. Combining with the fact that $\mathcal{M}_\alpha$ and $\mathcal{N}_\alpha$ are isomorphisms, we see that $Ker(\eta)_\alpha$ and $Coker(\eta)_\alpha$ are also isomorphisms. From this, it is also clear that $Coker(Ker(\eta)\longrightarrow \mathcal{M})= Ker(\mathcal{N}\longrightarrow Coker(\eta))$ in $Com^{cs}_c$-$\mathcal{C}$. 
  
  \smallskip
    Let $\{\mathcal{M}_i\}_{i\in I}$ be a directed family of objects in $Com^{cs}_c$-$\mathcal{C}$. Since the cotensor product commutes with direct limits (see, for instance, \cite[\S 10.5]{BW}), it now follows that
    $$(colim_{i\in I}\mathcal{M}_i)_x= colim_{i\in I}{\mathcal{M}_i}_x\xrightarrow {colim_{i\in I}{\mathcal{M}_i}_\alpha}colim_{i\in I}\alpha_*{\mathcal{M}_i}_y= \alpha_*(colim_{i\in I}{\mathcal{M}_i})_y$$ is an isomorphism. This shows that the expression in \eqref{4.1dpr} holds for both directed colimits and cokernels in $Com^{cs}_c$-$\mathcal{C}$. Since any colimit in $Com^{cs}_c$-$\mathcal{C}$ may be expressed in terms of cokernels and directed colimits, this proves the result. \end{proof}

\begin{lem}\label{kappa}
Let $\alpha: C\longrightarrow D$ be a left coflat morphism of semiperfect $K$-coalgebras. Let $\kappa'\geq max\{|K|,\aleph_0\}$. Let $M\in \mathbf{M}^D$ and $A\subseteq \alpha_*M$ be a set of elements such that $|A|\leq \kappa'$. Then, there exists a subcomodule $N\subseteq M$ in $\mathbf{M}^D$ with $|N|\leq \kappa'$ such that $A\subseteq \alpha_*N$.

  \end{lem}
  \begin{proof}
Let $a\in A\subseteq \alpha_*M$. Since $C$ is semiperfect, there exists a finite dimensional projective right $C$-comodule $V^a$ and a
morphism ${\eta}^a: V^a\longrightarrow \alpha_*M$ in $\mathbf{M}^C$ such that $a\in Im({\eta}^a)$. Since $D$ is semiperfect, there exists an epimorphism $\zeta: \bigoplus_{i\in I} V_i\longrightarrow M$ in $\mathbf{M}^D$ where each $V_i$ is a finite dimensional projective right $D$-comodule. Since $\alpha$ is coflat and $\alpha_*$ commutes with direct sums, $\alpha_*\zeta: \bigoplus_{i\in I}\alpha_*V_i\longrightarrow \alpha_*M$ is also an epimorphism. As $V^a$ is a projective right $C$-comodule, ${\eta}^a: V^a\longrightarrow \alpha_*M$ can be lifted to a morphism ${\eta'}^a: V^a\longrightarrow \bigoplus_{i\in I} \alpha_*V_i$. Moreover, since $V^a$ finitely generated, ${\eta'}^a$ factors through a finite direct sum of objects from the family $\{\alpha_*V_i\}_{i\in I}$, which we denote by $\{\alpha_*V_1^a,\hdots, \alpha_*V^a_{n^a}\}$. Thus, we obtain a morphism ${\zeta'}^a: \bigoplus_{r=1}^{n^a}V_r^a\longrightarrow M$ such that ${\eta}^a$ factors through $\alpha_*{\zeta'}^a$. Now, we set
\begin{equation}
N:= Im\left({\zeta'}:= \bigoplus_{a\in A}{\zeta'}^a: \bigoplus_{a\in A}\bigoplus_{r=1}^{n^a}V_r^a\longrightarrow M\right)
\end{equation}
Since $\alpha_*$ is exact and  commutes with direct sums, we now obtain 
\begin{equation}
\alpha_*(N):= Im\left(\alpha_*\zeta'= \bigoplus_{a\in A}\alpha_*{\zeta'}^a: \bigoplus_{a\in A}\bigoplus_{r=1}^{n^a}\alpha_* V_r^a\longrightarrow \alpha_*M\right)
  \end{equation}
  Since $a\in Im(\eta^a)$ and $\eta^a$ factors through $\alpha_*{\zeta'}^a$, we have $A\subseteq \alpha_*(N)$. Finally, since each $V_r^a$ is finite dimensional, we note that  $|V_r^a|\leq\kappa'$. Since $N$ is a quotient of $\bigoplus_{a\in A}\bigoplus_{r=1}^{n^a}V_r^a$ and $|A|\leq \kappa'$, it follows that $|N|\leq \kappa'$.
  \end{proof}

\begin{lem}\label{double}
  Let $\alpha: C\longrightarrow D$ be a left coflat morphism of semiperfect $K$-coalgebras and let $M\in \mathbf{M}^D$. Let $\kappa'\geq max\{|K|, \aleph_0, |C|\}$.  Let $A\subseteq M$ and $B\subseteq \alpha_*M$ be sets of elements such that $|A|,|B|\leq \kappa'$. Then, there exists a subcomodule $N\subseteq M$ in $ \mathbf{M}^D$ such that
  \begin{itemize}
  \item[(1)] $|N|\leq \kappa'$, $|\alpha_*N|\leq \kappa'$
  \item[(2)] $A\subseteq N$ and $B\subseteq \alpha_*N$
  \end{itemize}
  \end{lem}  

  \begin{proof}
    By Lemma \ref{kappa}, we can obtain a right $D$-comodule $N_1$ such that $|N_1|\leq \kappa'$ and $B\subseteq \alpha_*N_1$. Again, taking $\alpha=id_D: D\longrightarrow D$ in Lemma \ref{kappa}, we can also obtain a right $D$-comodule $N_2$ such that $|N_2|\leq \kappa'$ and $A\subseteq N_2$. We set, $N:= N_1+N_2\subseteq M$. Since $N$ is a quotient of $N_1\oplus N_2$, we have $|N|\leq \kappa'$. Clearly, we also have $ A\subseteq N_2\subseteq N$. Also, since $\alpha_*$ is a right adjoint, it preserves inclusions and hence $B\subseteq \alpha_*N_1\subseteq \alpha_*N$. Finally, since $\alpha_*N= N\square_DC$ is a linear subspace of $N\otimes C$, it follows from the definition of $\kappa'$ that $|\alpha_*N|\leq \kappa'$.
    \end{proof}

    Let  $\mathcal{C}: \mathscr{X}\longrightarrow Coalg$ be a left coflat representation taking values in semiperfect $K$-coalgebras and let $\mathscr{X}$ be a poset. We will now show that $Com_c^{cs}$-$\mathcal C$ has a generator. We will do this in a manner similar to \cite{EV} using transfinite induction (see also \cite{AB}). We begin by setting
    \begin{equation}
\kappa':=sup\{|K|,\aleph_0,|Mor(\mathscr{X})|, |\mathcal{C}_x|,~ x\in Ob(\mathscr{X})\}
\end{equation}
Let $\mathcal{M}\in Com_c^{cs}$-$\mathcal C$ and let $m_0\in el_{\mathscr{X}}(\mathcal{M})$. This means that there exists some $x\in Ob(\mathscr{X})$ such that $m_0\in \mathcal{M}_x$. By \cite[Corollary 2.2.9]{DNR}, there exists a finite dimensional right $\mathcal{C}_x$-comodule $V'$ and a morphism $\eta':V'\longrightarrow \mathcal{M}_x$ in $\mathbf{M}^{\mathcal{C}_x}$ such that $m_0\in im(\eta')$.  As in \eqref{defN}, corresponding to $\eta'$, we can define a subcomodule $\mathcal{N}\subseteq \mathcal{M}$ in $Com^{cs}$-$\mathcal{C}$ such that $m_0\in \mathcal{N}_x$ and  $|\mathcal{N}|\leq \kappa'$ (by Lemma \ref{cardinal}).
We now fix a well ordering of the set $Mor(\mathscr{X})$ and consider the induced lexicographic ordering on $\mathbb{N}\times Mor(\mathscr{X})$.  To each $(n,\alpha: y\longrightarrow z)\in \mathbb{N}\times Mor(\mathscr{X})$, we will associate a subcomodule $\mathcal{P}(n,\alpha)$ of $\mathcal{M}$ in $Com^{cs}$-$\mathcal{C}$ which satisfies the following conditions:
\begin{itemize}
\item[(1)] $m_0\in el_{\mathscr{X}}(\mathcal{P}(1,\alpha_0))$, where $\alpha_0$ is the least element of $Mor(\mathscr{X})$.
\item[(2)] $\mathcal{P}(m,\beta)\subseteq \mathcal{P}(n,\alpha)$ for $(m,\beta)\leq (n,\alpha)$ in $\mathbb{N}\times Mor(\mathscr{X})$.
\item[(3)] For each pair $(n, \alpha:y\longrightarrow z)\in \mathbb{N}\times Mor(\mathscr{X})$, the morphism $ \mathcal{P}(n,\alpha)_\alpha: \mathcal{P}(n,\alpha)_y\longrightarrow \alpha_*\mathcal{P}(n,\alpha)_z$ is an isomorphism in $\mathbf{M}^{\mathcal{C}_y}$.
  \item[(4)] $|\mathcal{P}(n,\alpha)|\leq \kappa'$
  \end{itemize}

First, we fix $(n, \alpha:y\longrightarrow z)\in \mathbb{N}\times Mor(\mathscr{X})$. To define $\mathcal P(n,\alpha)$, we begin by setting 

\begin{equation}\label{4.4rg}
A_0^0(w) = 
     \begin{cases}
      \mathcal{N}_w&\quad\text{if $n=1$ and $\alpha=\alpha_0$}\\
       \bigcup_{(m,\beta)<(n,\alpha)}\mathcal{P}(m,\beta)_w&\quad\text{otherwise}
     \end{cases}
\end{equation}

for each $w\in Ob(\mathscr{X})$. Clearly, each $A_0^0(w)\subseteq \mathcal{M}_w$ and $|A_0^0(w)|\leq \kappa'$. 
For $\alpha:y\longrightarrow z$, we have $A_0^0(z)\subseteq \mathcal{M}_z$ and $A_0^0(y)\subseteq \mathcal{M}_y\cong \alpha_*\mathcal{M}_z$ since $\mathcal{M}$ is cartesian. By  Lemma \ref{double} we can obtain a subcomodule $A_1^0(z)\subseteq \mathcal{M}_z$ in $\mathbf{M}^{\mathcal{C}_z}$ such that
\begin{equation}\label{A10}
  |A_1^0(z)|\leq \kappa'\quad |\alpha_*A_1^0(z)|\leq \kappa'\quad A_0^0(z)\subseteq A_1^0(z)\quad A_0^0(y)\subseteq \alpha_*A_1^0(z)
\end{equation}
Now, we set $A_1^0(y):=\alpha_*A_1^0(z)$ and $A_1^0(w):=A_0^0(w)$ for any $w\neq y,z\in Ob(\mathscr{X})$. We observe that since $\mathscr{X}$ is a poset, $y=z$ implies $\alpha: y\longrightarrow z$ is the identity and hence $A_1^0(y)= A_1^0(z)$. It now follows from \eqref{A10} that for every $w\in Ob(\mathscr{X})$, we have $A_0^0(w)\subseteq A_1^0(w)$  and $|A_1^0(w)|\leq \kappa'$.

 \begin{lem}\label{subcomod}
Let $B\subseteq el_{\mathscr{X}}(\mathcal{M})$ with $|B|\leq \kappa'$. Then, there is a subcomodule $\mathcal{Q}\subseteq\mathcal{M}$ in $Com^{cs}$-$\mathcal{C}$ such that $B\subseteq el_{\mathscr{X}}(\mathcal{Q})$ and $|\mathcal{Q}|\leq \kappa'$. In particular, if $m_0\in B$ is such that $m_0\in \mathcal{M}_x$ for some $x\in Ob(\mathscr{X})$, then $m_0\in \mathcal{Q}_x$.

  \end{lem}
  \begin{proof}
    As in the proof of Theorem \ref{Groth}, for any $m_0\in B\subseteq el_\mathscr{X}(\mathcal{M})$ we can choose a subcomodule $\mathcal{Q}(m_0)\subseteq \mathcal{M}$ in $Com^{cs}$-$\mathcal{C}$ such that $m_0\in el_{\mathscr{X}}(\mathcal{Q}(m_0))$ and $|\mathcal{Q}(m_0)|\leq \kappa'$. Now, we set $\mathcal{Q}:= \sum_{m_0\in B}\mathcal{Q}(m_0)$. Since $\mathcal{Q}$ is a quotient of $\oplus_{m_0\in B}\mathcal{Q}(m_0)$ and $|B|\leq \kappa'$, we have $|\mathcal{Q}|\leq \kappa'$. The last statement is clear.
  \end{proof}

We note that we have $A_1^0(w)\subseteq\mathcal{M}_w $ and $|A_1^0(w)|\leq \kappa'$ for all $w\in Ob(\mathscr{X})$. Thus, by the definition of $\kappa'$, we have $\bigcup_{w\in Ob(\mathscr{X})}A_1^0(w)\subseteq el_\mathscr{X}(\mathcal{M})$ and $|\bigcup_{w\in Ob(\mathscr{X})}A_1^0(w)|\leq \kappa'$.  Therefore, by Lemma \ref{subcomod},we can choose a subcomodule $\mathcal{Q}^0(n,\alpha)\subseteq\mathcal{M}$ in $Com^{cs}$-$\mathcal{C}$ such that
 $\bigcup_{w\in Ob(\mathscr{X})}A_1^0(w)\subseteq el_{\mathscr{X}}(\mathcal{Q}^0(n,\alpha))$ and $|\mathcal{Q}^0(n,\alpha)|\leq \kappa'$. We also have $A_1^0(w)\subseteq \mathcal{Q}^0(n,\alpha)_w\subseteq \mathcal{M}_w$ for each $w\in Ob(\mathscr{X})$ by Lemma \ref{subcomod}.  \vspace{2mm}
 
  We will now iterate this construction. Suppose that for each $l\leq r$, we have constructed a subcomodule $\mathcal{Q}^l(n,\alpha)\subseteq\mathcal{M}$ in $Com^{cs}$-$\mathcal{C}$ satisfying  $ \bigcup_{w\in Ob(\mathscr{X})}A_1^l(w)\subseteq el_{\mathscr{X}}(\mathcal{Q}^l(n,\alpha))$ and $|\mathcal{Q}^l(n,\alpha)|\leq \kappa'$. Let $A_0^{r+1}(w):= \mathcal{Q}^r(n,\alpha)_w$ for each $w\in Ob(\mathscr{X})$. Since $A_0^{r+1}(y)\subseteq \mathcal{M}_y\cong \alpha_*(\mathcal{M}_z)$ and $A_0^{r+1}(z)\subseteq \mathcal{M}_z$, using Lemma \ref{double}, we can obtain a subcomodule $A_1^{r+1}(z)\subseteq \mathcal{M}_z$ in $\mathbf{M}^{\mathcal{C}_z}$ such that
 
 \begin{equation}\label{m+1}
  |A_1^{r+1}(z)|\leq \kappa'\quad |\alpha_*A_1^{r+1}(z)|\leq \kappa'\quad A_0^{r+1}(z)\subseteq A_1^{r+1}(z)\quad A_0^{r+1}(y)\subseteq \alpha_*A_1^{r+1}(z)
\end{equation}

Now, we set $A_1^{r+1}(y):= \alpha_*A_1^{r+1}(z)$ and $A_1^{r+1}(w):= A_0^{r+1}(w)$ for any $w\neq y,z\in Ob(\mathscr{X})$. Combining this with \eqref{m+1} it follows that $A_0^{r+1}(w)\subseteq A_1^{r+1}(w)$ for all $w\in Ob(\mathscr{X})$ and each $|A_1^{r+1}(w)|\leq \kappa'$. Therefore, by Lemma \ref{subcomod} we can choose a subcomodule $\mathcal{Q}^{r+1}(n,\alpha)\subseteq \mathcal{M}$ in $Com^{cs}$-$\mathcal{C}$ such that $\bigcup_{w\in Ob(\mathscr{X})}A_1^{r+1}(w)\subseteq el_{\mathscr{X}}(\mathcal{Q}^{r+1}(n,\alpha))$ and $|\mathcal{Q}^{r+1}(n,\alpha)|\leq \kappa'$. In particular, $A_1^{r+1}(w)\subseteq \mathcal{Q}^{r+1}(n,\alpha)_w$ for each $w\in Ob(\mathscr{X})$. Finally, we set

\begin{equation}\label{Definition4.7}
  \mathcal{P}(n,\alpha):=\bigcup_{r\geq 0}\mathcal{Q}^r(n,\alpha)
\end{equation}

\begin{lem}\label{Pnalpha}

The family $\{\mathcal{P}(n,\alpha)~|~(n,\alpha)\in \mathbb{N}\times Mor(\mathscr{X})\}$ satisfies conditions (1)-(4). 

 \end{lem}
 \begin{proof} The conditions (1) and (2) follow immediately from the definitions in \eqref{4.4rg} and \eqref{Definition4.7}.   Since each $|\mathcal{Q}^r(n,\alpha)|\leq \kappa'$, the result of (4) follows from \eqref{Definition4.7}. It remains to show (3). For this, we observe that $\mathcal P(n,\alpha)_z$  can be expressed as the filtered union
   \begin{equation}
A_1^0(z)\hookrightarrow \mathcal{Q}^0(n,\alpha)_z\hookrightarrow A_1^1(z)\hookrightarrow \mathcal{Q}^1(n,\alpha)_z\hookrightarrow\hdots \hookrightarrow A_1^{r+1}(z)\hookrightarrow \mathcal{Q}^{r+1}(n,\alpha)_z\hookrightarrow\hdots
\end{equation}
of objects in $\mathbf{M}^{\mathcal{C}_z}$. Since $\alpha_*$ preserves monomorphisms as well as direct limits, we can also express $\alpha_*\mathcal{P}(n,\alpha)_z$ as the filtered union
\begin{equation}\label{2}
\alpha_*A_1^0(z)\hookrightarrow \alpha_*\mathcal{Q}^0(n,\alpha)_z\hookrightarrow \alpha_*A_1^1(z)\hookrightarrow \alpha_*\mathcal{Q}^1(n,\alpha)_z\hookrightarrow\hdots \hookrightarrow \alpha_*A_1^{r+1}(z)\hookrightarrow \alpha_*\mathcal{Q}^{r+1}(n,\alpha)_z\hookrightarrow\hdots
  \end{equation}
   of objects in $\mathbf{M}^{\mathcal{C}_y}$. Likewise, $\mathcal{P}(n,\alpha)_y$ can be expressed as the filtered union
    \begin{equation}\label{1}
   	A_1^0(y)\hookrightarrow \mathcal{Q}^0(n,\alpha)_y\hookrightarrow A_1^1(y)\hookrightarrow \mathcal{Q}^1(n,\alpha)_y\hookrightarrow\hdots \hookrightarrow A_1^{r+1}(y)\hookrightarrow \mathcal{Q}^{r+1}(n,\alpha)_y\hookrightarrow\hdots
   \end{equation}
of objects in $\mathbf{M}^{\mathcal{C}_y}$.
   Also, by definition, we have isomorphisms $A_1^r(y)=\alpha_*A_1^r(z)$ in $\mathbf M^{\mathcal C_y}$ for each $r\geq 0$. Together,   these  induce an isomorphism $\mathcal{P}(n,\alpha)_\alpha: \mathcal{P}(n,\alpha)_y\longrightarrow \alpha_*\mathcal{P}(n,\alpha)_z$.
 \end{proof}

 \begin{lem}\label{gen}
Let $\mathcal M\in Com^{cs}_c$-$\mathcal{C}$. For any $m_0\in el_\mathscr{X}(\mathcal{M})$ there exists a subcomodule $\mathcal{P}\subseteq \mathcal{M}$ in $Com_c^{cs}$-$\mathcal{C}$ such that $m_0\in el_\mathscr{X}(\mathcal{P})$ and $|\mathcal{P}|\leq \kappa'$.
\end{lem}
\begin{proof}
  Clearly, the set $\mathbb{N}\times Mor(\mathscr{X})$ with the lexicographic ordering is filtered. We set $
    \mathcal{P}:=\bigcup_{(n,\alpha)\in \mathbb{N}\times Mor(\mathscr{X})}\mathcal{P}(n,\alpha)\subseteq \mathcal{M}
 $
  in $Com^{cs}$-$\mathcal{C}$. By Lemma \ref{Pnalpha}, $m_0\in \mathcal{P}(1,\alpha_0)$ which implies $m_0\in el_{\mathscr{X}}(\mathcal{P})$. Also, since each $|\mathcal{P}(n,\alpha)|\leq \kappa'$, we have $|\mathcal{P}|\leq \kappa'$. Now, we fix any morphism $\gamma: u\longrightarrow v$ in $\mathscr{X}$. We note that the family $\{(m,\gamma)\}_{m\geq 1}$ is cofinal in $\mathbb{N}\times Mor(\mathscr{X})$ and hence we may write 
 $
    \mathcal{P}=\underset{m\geq1}{ \underset{\longrightarrow}{\text{lim}}}\mathcal{P}(m,\gamma)$. 
 Since cotensor product commutes with direct limits (see, for instance, \cite[\S 10.5]{BW}), we have $ \gamma_*(\mathcal{P}_v)=\underset{m\geq1}{ \underset{\longrightarrow}{\text{lim}}} \gamma_*(\mathcal{P}(m,\gamma)_v)$. Since each $\mathcal{P}(m,\gamma)_\gamma: \mathcal{P}(m,\gamma)_u\longrightarrow \gamma_*\mathcal{P}(m,\gamma)_v$ is an isomorphism, it follows that the filtered colimit $\mathcal{P}_\gamma: \mathcal{P}_u\longrightarrow \gamma_*(\mathcal{P}_v)$ is an isomorphism.
\end{proof}

\begin{Thm}\label{T4.8b}
Let $\mathscr{X}$ be a poset and let $\mathcal{C}: \mathscr{X}\longrightarrow Coalg$ be a left coflat representation taking values in semiperfect $K$-coalgebras. Then, $Com^{cs}_c$-$\mathcal{C}$ is a Grothendieck category.
  \end{Thm}

  \begin{proof}
From the description of direct limits in $Com^{cs}_c$-$\mathcal{C}$ given in Lemma \ref{cocomplete},   it follows that they commute with finite limits. Also, given any $\mathcal{M}\in Com_c^{cs}$-$\mathcal{C}$, it follows from Lemma \ref{gen}, that $\mathcal{M}$ can be expressed as a sum of a family $\{\mathcal{P}_{m_0}\}_{m_0\in el_\mathscr{X}(\mathcal{M})}$ of subcomodules of $\mathcal{M}$  such that $\mathcal{P}_{m_0}\in Com^{cs}_c$-$\mathcal{C}$ and $|\mathcal{P}_{m_0}|\leq \kappa'$. Therefore, isomorphism classes of cartesian right cis-comodules $\mathcal{P}$ with $|\mathcal{P}|\leq \kappa'$ form a set of generators for $Com^{cs}_c$-$\mathcal{C}$. 
\end{proof}

\begin{Thm}\label{T4.9b}  Let $\mathscr{X}$ be a poset and let $\mathcal{C}: \mathscr{X}\longrightarrow Coalg$ be a left coflat representation taking values in semiperfect $K$-coalgebras. Then, the inclusion functor $i: Com^{cs}_c$-$\mathcal{C}\hookrightarrow Com^{cs}$-$\mathcal{C}$ has a right adjoint. 
\end{Thm}

\begin{proof}
It follows from Lemma \ref{cocomplete} that the inclusion functor $i: Com^{cs}_c$-$\mathcal{C}\hookrightarrow Com^{cs}$-$\mathcal{C}$ preserves all colimits. By Theorem \ref{T4.8b}, $Com^{cs}_c$-$\mathcal{C}$ is a Grothendieck category and it follows (see, for instance, \cite[Proposition 8.3.27]{KS}) that $i$ admits a right adjoint. 
\end{proof}

\begin{rem}
\emph{We note that the right adjoint in Theorem \ref{T4.9b} may be seen as a coalgebraic counterpart of the classical quasi-coherator construction (see \cite[Lemme 3.2]{Ill}), which is right adjoint to the inclusion of quasi-coherent sheaves into the category of modules over a scheme.}
\end{rem}

\subsection{Cartesian trans-comodules over coalgebra representations}
Let $\mathscr{X}$ be a small category and let $\mathcal{C}: \mathscr{X} \longrightarrow Coalg$ be a quasi-finite representation taking values in  right semi-perfect $K$-coalgebras. In this section, we will introduce the category of cartesian trans-comodules over the representation $\mathcal{C}
$. 

\smallskip

Let $\alpha:C \longrightarrow D$ be a coalgebra morphism such that $C$ is quasi-finite as a right $D$-comodule. Then, the functor $\alpha^!=H_D(C,-): \mathbf{M}^D \longrightarrow \mathbf{M}^C$ is exact if and only if the direct sum $C^{(\Lambda)}$ is injective as a right $D$-comodule for any indexing set $\Lambda$ (see  \cite[Corollary 3.10]{takh}). In such a situation, we will say that the morphism  $\alpha:C \longrightarrow D$ of coalgebras is
(right) $\Sigma$-injective. We note in particular that any identity morphism of coalgebras is $\Sigma$-injective. 

\begin{defn}
Let $\mathcal{C}: \mathscr{X} \longrightarrow Coalg$ be a quasi-finite representation   taking values in right semi-perfect coalgebras. Suppose that $\mathcal C$ is  $\Sigma$-injective, i.e., for each 
each $\alpha \in \mathscr{X}(x,y)$, the morphism $\mathcal C_\alpha:\mathcal C_x\longrightarrow \mathcal C_y$ of coalgebras is $\Sigma$-injective. 
 Let $\mathcal{M} \in Com^{tr}\text{-}\mathcal{C}$. Then, we say that $\mathcal{M}$ is cartesian if for each $\alpha \in \mathscr{X}(x,y)$, the morphism $^\alpha \mathcal{M}:\alpha^!\mathcal{M}_y \longrightarrow  \mathcal{M}_x$ is an isomorphism in $\mathbf{M}^{\mathcal{C}_x}$. We will denote the full subcategory of cartesian trans-comodules  by $Com^{tr}_c\text{-}\mathcal{C}$.
\end{defn}

\begin{lem}\label{trans-ab}
The category $Com^{tr}_c$-$\mathcal{C}$ is cocomplete and abelian. Further, if $\mathcal{M}$ denotes the colimit of a family $\{\mathcal{M}_i\}_{i\in I}$ of objects in $Com^{tr}_c$-$\mathcal{C}$, then for each $x\in Ob(\mathscr{X})$, we have $
\mathcal{M}_x= {colim}_{i\in I} {\mathcal{M}_i}_x
$,
where the colimit on the right is taken in the category $\mathbf{M}^{\mathcal{C}_x}$. 
\end{lem}
\begin{proof}
Let $\eta:\mathcal{M} \longrightarrow \mathcal{N}$ be a morphism in $Com^{tr}_c$-$\mathcal{C}$. We know that $Ker(\eta)$ and $Coker(\eta)$ in   $Com^{tr}\text{-}\mathcal{C}$ are constructed pointwise. Since $\alpha^!$ is exact, it follows as in the proof of Lemma \ref{cocomplete} that $Ker(\eta)$ and $Coker(\eta)$ lie in  $Com^{tr}_c\text{-}\mathcal{C}$ and hence  $Com^{tr}_c\text{-}\mathcal{C}$ is abelian. Further, since $\alpha^!$ is left adjoint, it preserves colimits and we may verify as in the proof of Lemma \ref{cocomplete} that  $Com^{tr}_c\text{-}\mathcal{C}$  is cocomplete, with all colimits computed pointwise. \end{proof}

We will now look at generators in $Com^{tr}_c$-$\mathcal{C}$.

\begin{lem}\label{transkappa}
Let $\alpha: C\longrightarrow D$ be a quasi-finite and $\Sigma$-injective morphism between right-semiperfect coalgebras. Let $\kappa'\geq max\{|K|,\aleph_0, |D|\}$. Let $M \in \mathbf M^D$ and $A\subseteq \alpha^!M$ be a set of elements such that $|A|\leq \kappa'$. Then, there exists a subcomodule $N\subseteq M$ in $\mathbf M^D$ with $|N|\leq \kappa'$ such that $A\subseteq \alpha^!N$.
\end{lem}
\begin{proof}
This follows in a manner similar to Lemma \ref{kappa}, using the fact that $\alpha^!$ is both a left adjoint and assumed to be exact. \end{proof}

\begin{lem}\label{transdouble}
Let $\alpha: C\longrightarrow D$ be a quasi-finite and $\Sigma$-injective morphism between right semiperfect coalgebras. Let $\kappa' \geq max\{|K|, \aleph_0, |C^*|,2^{|D|}\}$. Let $M\in \mathbf M^D$  and let $A\subseteq M$ and $B\subseteq \alpha^! M$ be sets of elements such that $|A|,|B|\leq \kappa'$. Then, there exists a subcomodule $N\subseteq M$ in $\mathbf{M}^D$ such that

\smallskip
(1) $|N|\leq \kappa'$, $|\alpha^! N|\leq \kappa'$

\smallskip
(2) $A\subseteq N$ and $B\subseteq \alpha^!N$
\end{lem}  
\begin{proof}
By Lemma \ref{transkappa}, we know that there exists a $D$-subcomodule $N_1 \subseteq M$ such that $|N_1|\leq \kappa'$ and $B\subseteq \alpha^! N_1$. Similarly, taking $\alpha=id_D: D\longrightarrow D$ in Lemma \ref{transkappa}, we can also obtain a $D$-subcomodule $N_2 \subseteq M$ such that $|N_2|\leq \kappa'$ and $A\subseteq N_2$. We set $N:= N_1+N_2 \subseteq M$. Since $N$ is a quotient of $N_1\oplus N_2$, we have $|N|\leq \kappa'$. Also,  $ A\subseteq N_2\subseteq N$. Since $\alpha^!$ is exact, it preserves monomorphisms and thus $B\subseteq \alpha^! N_1\subseteq \alpha^! N$. Finally, we know that the functor given by the composition $\textbf{M}^D \xrightarrow{\alpha^!} \textbf{M}^C \xrightarrow{forget} Vect$ is exact. Therefore, using \cite[Corollary 3.12]{takh}, we have that $\alpha^! N= H_D(C,N)=N \square_D H_D(C,D)$ as a vector space. Moreover, $H_D(C,D)=\underset{i \in I}{\varinjlim}~ Hom(D_i,C)^*$, where $\{D_i\}_{i \in I}$ is a directed family of finite dimensional subcoalgebras of $D$ (see, for instance \cite[Section 1]{takeu}). It now follows that $|\alpha^!N| \leq \kappa'$.\end{proof}

Let $\mathscr X$ be a poset. We continue with  $\mathcal{C}: \mathscr{X} \longrightarrow Coalg$ being a quasi-finite and $\Sigma$-injective representation taking values in right semiperfect coalgebras.   We fix
\begin{equation}
\kappa':=sup\{\aleph_0,|K|,|Mor(\mathscr{X})|, \{2^{|\mathcal{C}_x|}\}_{x\in Ob(\mathscr X)},  \{ |\mathcal{C}_x^*|\}_{x\in Ob(\mathscr X)}\} 
\end{equation}

We now choose a well ordering of the set $Mor(\mathscr{X})$ and consider the induced lexicographic ordering of $\mathbb{N}\times Mor(\mathscr{X})$.  Let $\mathcal{M}\in Com^{tr}_c$-$\mathcal{C}$ and let $m_0 \in el_{\mathscr{X}}(\mathcal{M})$, i.e. $m_0 \in \mathcal{M}_x$ for some $x\in Ob(\mathscr{X})$. We will now define a family of  subcomodules $\{\mathcal{P}(n,\alpha) | (n,\alpha: y\longrightarrow z)\in \mathbb{N}\times Mor(\mathscr{X})\}$ of $\mathcal{M}$ which satisfies the following conditions:

\begin{enumerate}
\item[(1')] $m_0 \in el_{\mathscr{X}}(\mathcal{P}(1,\alpha_0))$, where $\alpha_0$ is the least element of $Mor(\mathscr{X})$.

\item[(2')] $\mathcal{P}(m,\beta)\subseteq \mathcal{P}(n,\alpha)$ whenever $(m,\beta)\leq (n, \alpha)$ in $\mathbb{N}\times Mor(\mathscr{X})$.

\item[(3')]  For each $(n, \alpha:y\longrightarrow z)\in \mathbb{N}\times Mor(\mathscr{X})$, the morphism $^\alpha \mathcal{P}(n,\alpha): \alpha^! \mathcal{P}(n,\alpha)_z\longrightarrow \mathcal{P}(n,\alpha)_y$ is an isomorphism in $\mathbf{M}^{\mathcal{C}_y}$.

\item[(4')] $|\mathcal{P}(n,\alpha)| \leq \kappa'$.
\end{enumerate}

We know that there exists a finite dimensional $\mathcal{C}_x$-comodule $V$ and a morphism $\eta:V \longrightarrow \mathcal{M}_x$ such that $m_0 \in im(\eta)$. Then, we can define the subcomodule $\mathcal{N}\subseteq \mathcal{M}$ corresponding to $\eta$ as in \eqref{defN1} such that $m_0 \in \mathcal{N}_x$. By Lemma \ref{cardinal1}, we also know that $|\mathcal{N}|\leq \kappa'$. 

\smallskip
For each pair $(n,\alpha: y\longrightarrow z)\in \mathbb{N}\times Mor(\mathscr{X})$, we now start constructing the comodule $\mathcal{P}(n,\alpha)$ inductively.  We set
\begin{equation}\label{transA00}
A_0^0(w) = 
\begin{cases}
\mathcal{N}_w&\quad\text{if $n=1$ and $\alpha=\alpha_0$}\\
\bigcup\limits_{(m,\beta)<(n,\alpha)}\mathcal{P}(m,\beta)_w&\quad\text{otherwise}
\end{cases}
\end{equation}
for each $w\in Ob(\mathscr{X})$, where for $(n,\alpha) \neq (1,\alpha_0)$, we assume that we have already constructed the subcomodules $\mathcal{P}(m,\beta)$ satisfying all the properties (1')-(4') for any $(m, \beta) < (n,\alpha) \in \mathbb{N}\times Mor(\mathscr{X})$.

\smallskip
 Clearly, each $A_0^0(w)\subseteq \mathcal{M}_w$ and $|A_0^0(w)|\leq \kappa'$. 
Since $\mathcal{M}$ is cartesian, we know that $\alpha^! \mathcal{M}_z\cong \mathcal{M}_y$. Using Lemma  \ref{transdouble}, we can obtain a comodule $A_1^0(z)\subseteq \mathcal{M}_z$ in $\mathbf{M}^{\mathcal{C}_z}$ such that
\begin{equation}\label{A10trans}
|A_1^0(z)|\leq \kappa' \quad |\alpha^! A_1^0(z)|\leq \kappa' \quad A_0^0(z)\subseteq A_1^0(z)\quad A_0^0(y)\subseteq \alpha^!A_1^0(z)
\end{equation}

We now set $A_1^0(y):=\alpha^!A_1^0(z)$ and $A_1^0(w):=A_0^0(w)$ for any $w\neq y,z \in Ob(\mathscr{X})$. 
 It  follows from \eqref{A10trans} that $A_0^0(w)\subseteq A_1^0(w)$ for every $w\in Ob(\mathscr{X})$ and each $|A_1^0(w)|\leq \kappa'$.

 \begin{lem}\label{subtranscomod}
	Let $B\subseteq el_{\mathscr{X}}(\mathcal{M})$ with $|B|\leq \kappa'$. Then, there is a subcomodule $\mathcal{Q}\subseteq\mathcal{M}$ in $Com^{tr}$-$\mathcal{C}$ such that $B\subseteq el_{\mathscr{X}}(\mathcal{Q})$ and $|\mathcal{Q}|\leq \kappa'$. In particular, if $m_0\in B$ is such that $m_0\in \mathcal{M}_x$ for some $x\in Ob(\mathscr{X})$, then $m_0\in \mathcal{Q}_x$.
\end{lem}
\begin{proof}
This is similar to the proof of Lemma \ref{subcomod}.
\end{proof}

Applying Lemma \ref{subtranscomod} to $B=\bigcup\limits_{w\in Ob(\mathscr{X})}A_1^0(w)$, we see that there exists a subcomodule $\mathcal{Q}^0(n,\alpha)\hookrightarrow \mathcal{M}$ such that $A_1^0(w)\subseteq \mathcal{Q}^0(n,\alpha)_w \subseteq \mathcal{M}_w$ for each $w\in Ob(\mathscr{X})$ and 
$|\mathcal{Q}^0(n,\alpha)|\leq \kappa'$.
Suppose now that we have constructed a subcomodule $\mathcal{Q}^l(n,\alpha)\hookrightarrow \mathcal{M}$ for every $l\leq r$ such that  $ \bigcup\limits_{w\in Ob(\mathscr{X})}A_1^l(w)\subseteq el_{\mathscr{X}}(\mathcal{Q}^l(n,\alpha))$ and $|\mathcal{Q}^l(n,\alpha)|\leq \kappa'$.
We now set $A_0^{r+1}(w):= \mathcal{Q}^r(n,\alpha)_w$ for each $w\in Ob(\mathscr{X})$. 
Since $A_0^{r+1}(y)\subseteq \mathcal{M}_y\cong \alpha^!\mathcal{M}_z$ and $A_0^{r+1}(z)\subseteq \mathcal{M}_z$, using Lemma \ref{transdouble}, we can obtain a subcomodule $A_1^{r+1}(z)\hookrightarrow \mathcal{M}_z$ in $\mathbf{M}^{\mathcal{C}_z}$ such that
\begin{equation}\label{transm+1}
|A_1^{r+1}(z)|\leq \kappa' \qquad |\alpha^!A_1^{r+1}(z)|\leq \kappa' \qquad A_0^{r+1}(z)\subseteq A_1^{r+1}(z)\qquad A_0^{r+1}(y)\subseteq \alpha^!A_1^{r+1}(z)
\end{equation}
Then, we set $A_1^{r+1}(y):= \alpha^!A_1^{r+1}(z)$ and $A_1^{r+1}(w):= A_0^{r+1}(w)$ for any $w\neq y,z\in Ob(\mathscr{X})$.  It now follows from  \eqref{transm+1} that $A_0^{r+1}(w)\subseteq A_1^{r+1}(w)$ for all $w\in Ob(\mathscr{X})$ and each $|A_1^{r+1}(w)|\leq \kappa'$.  Using Lemma \ref{subtranscomod}, we can obtain a subcomodule $\mathcal{Q}^{r+1}(n,\alpha)\hookrightarrow \mathcal{M}$ such that
$\bigcup_{w\in Ob(\mathscr{X})}A_1^{r+1}(w)\subseteq el_{\mathscr{X}}(\mathcal{Q}^{r+1}(n,\alpha))$ and $|\mathcal{Q}^{r+1}(n,\alpha)|\leq \kappa'$.  
We finally set
\begin{equation}\label{transDefinition}
\mathcal{P}(n,\alpha):=\bigcup_{r\geq 0}\mathcal{Q}^r(n,\alpha)
\end{equation}

\begin{lem}
The family $\{\mathcal{P}(n,\alpha) | (n, \alpha) \in \mathbb{N} \times Mor(\mathscr{X})\}$ satisfies conditions (1')-(4'). 
\end{lem}

\begin{proof}
The idea of the proof is similar to that of Lemma \ref{Pnalpha}, using here the fact that $\alpha^!$ preserves colimits (being a left adjoint) and the assumption that $\alpha^!$ is exact. 
\end{proof}

\begin{lem}\label{transgen}
Let $\mathcal M\in Com^{tr}_c$-$\mathcal{C}$. For any $m_0\in el_\mathscr{X}(\mathcal{M})$ there exists a subcomodule $\mathcal{P}\subseteq \mathcal{M}$ in $Com_c^{tr}$-$\mathcal{C}$ such that $m_0\in el_\mathscr{X}(\mathcal{P})$ and $|\mathcal{P}|\leq \kappa'$.
\end{lem}

\begin{proof}
Since $\mathbb{N}\times Mor(\mathscr{X})$ with the lexicographic ordering is filtered, we set $\mathcal{P}:=\bigcup_{(n,\alpha)\in \mathbb{N}\times Mor(\mathscr{X})}\mathcal{P}(n,\alpha)\subseteq \mathcal{M}$. The proof now follows as in Lemma \ref{gen}, using the fact that $\alpha^!$ is a left adjoint and that  $\{(m,\beta)\}_{m\geq 1}$ is cofinal in $\mathbb{N}\times Mor(\mathscr{X})$ for any morphism $\beta$ in $\mathscr X$.
\end{proof} 

\begin{Thm}
Let $\mathscr{X}$ be a poset. Let $\mathcal{C}: \mathscr{X}\longrightarrow Coalg$ be a quasi-finite and $\Sigma$-injective representation taking values in right semiperfect coalgebras.  Then, $Com^{tr}_c$-$\mathcal{C}$ is a Grothendieck category.
\end{Thm}
\begin{proof} Since filtered colimits and finite limits in $Com^{tr}_c$-$\mathcal{C}$ are both computed pointwise, it is clear that they commute in $Com^{tr}_c$-$\mathcal{C}$.
  It also follows from Lemma \ref{transgen} that any $\mathcal{M}\in Com^{tr}_c$-$\mathcal{C}$ can be expressed as a sum of a family $\{\mathcal{P}_{m_0}\}_{m_0\in el_\mathscr{X}(\mathcal{M})}$ of cartesian subcomodules such that each $|\mathcal{P}_{m_0}|\leq \kappa'$. Therefore, the isomorphism classes of cartesian comodules $\mathcal{P}$ with $|\mathcal{P}|\leq \kappa'$ form a set of generators for $Com^{tr}_c$-$\mathcal{C}$. 
\end{proof}

\begin{Thm}
Let $\mathscr{X}$ be a poset. Let $\mathcal{C}: \mathscr{X}\longrightarrow Coalg$ be a quasi-finite and $\Sigma$-injective representation taking values in  right semiperfect coalgebras. Then, the inclusion functor $i:Com^{tr}_c$-$\mathcal{C} \longrightarrow Com^{tr}$-$\mathcal{C}$ has a right adjoint.
\end{Thm}
\begin{proof}
Using Lemma \ref{trans-ab}, we know that the inclusion functor $i:Com^{tr}_c$-$\mathcal{C} \longrightarrow Com^{tr}$-$\mathcal{C}$ preserves colimits. Since $Com^{tr}_c$-$\mathcal{C}$ is a Grothendieck category, it follows that $i$ has a right adjoint.
\end{proof}

\section{Contramodules over coalgebra representations}
Let $C$ be a $K$-coalgebra having coproduct $\Delta_C$ and counit $\epsilon_C$. A contramodule over $C$  (see, for instance, \cite[\S 0.2.4]{semicontra}) consists of a $K$-space $M$ along with a $K$-linear  ``contraction map"  $\pi_M:Hom_K(C,M)\longrightarrow M$ such that the following diagrams

\begin{equation}\label{Feq5.0}
\xymatrix{
	& M\ar[d]_{Hom({\epsilon_C,M})} \ar[dr]^{id}&  \\ 
	& Hom_K(C,M)\ar[r]_{~~~~~~~\pi_M} & M}
\xymatrix{
	& Hom_K(C, Hom_K(C,M))\cong Hom_K(C\otimes C, M) \ar[d]_{\pi_M\circ{-}} \ar[r]^{~~~~~~~~~~~~~~~~~~~~~{-}\circ \Delta_C}
	& Hom_K(C,M) \ar[d]_{\pi_M} \\ 
	& Hom_K(C,M)\ar[r]_{\pi_M} 
	& M}
\end{equation}
commute.  If the isomorphism $Hom_K(C, Hom_K(C,M))\cong Hom_K(C\otimes C, M)$ in the diagram above is obtained from the adjointness of ${-}\otimes C$ and $Hom_K(C,{-}) $ (resp. the adjointness of  $C\otimes {-}$ and $Hom_K(C,{-}) $), then $M$ is a right (resp. left) $C$-contramodule. 

\smallskip
Equivalently, right $C$-contramodules  may be identified with objects in the Eilenberg-Moore category of modules over the following monad: let $T_C$ denote  the endofunctor  
\begin{equation}T_C:Vect\longrightarrow Vect~~~~~~~~~~ M \mapsto Hom_K(C,M)
\end{equation}
on the category of $K$-vector spaces. Then, there is a natural transformation $\eta: 1_{Vect}\longrightarrow T_C$ given by 
\begin{equation}\eta(M): M\longrightarrow Hom(C, M)\ \ \ \ \ \ m\mapsto \epsilon_C\cdot m\end{equation}
for each vector space $M$, where $(\epsilon_C\cdot m)(c):= \epsilon_C(c)m$ for each  $c\in C$.
There is also a natural transformation $\mu: T_C\circ T_C\longrightarrow T_C$ given by 
\begin{equation}\label{Feq5.1} Hom_K(C,Hom_K(C,M))\cong Hom_K(C\otimes C,M)\longrightarrow Hom(C,M)~~~~~~f\mapsto f\circ\Delta_C
\end{equation}
where  the isomorphism $Hom_K(C,Hom_K(C,M))\cong Hom_K(C\otimes C,M)$ in \eqref{Feq5.1} comes from the adjointness of ${-}\otimes C$ and $Hom_K(C,{-}) $. 
It may be verified that $(T_C, \mu, \eta)$ is a monad on the category of $K$-vector spaces. Accordingly, a right $C$-contramodule may be described (see, for instance, \cite[$\S$ 4.4]{BBW}) as a pair
$(M,\pi_M)$ satisfying the conditions in \eqref{Feq5.0}.  We will denote the category of right $C$-contramodules by $\mathbf M_{[C,\_\_]}$.

\smallskip
The free contradmodules are of the form $T_C(V)$ for some vector space $V$. Further, the free functor $T_C:Vect\longrightarrow \mathbf M_{[C,\_\_]}$ is left adjoint to the forgetful functor, which gives natural isomorphisms
 (see \cite[\S 0.2.4]{semicontra})
\begin{equation}\label{5adjcontra}
	\mathbf M_{[C,\_\_]}(T_C(V), M)\cong Hom_K(V,M)
\end{equation} for any vector space $V$ and any $C$-contramodule $M$. We mention that the forgetful functor $\mathbf M_{[C,\_\_]}\longrightarrow Vect$
is exact (see \cite[$\S$ 3.1.2]{semicontra}). From \eqref{5adjcontra}, we also note that direct sums of free contramodules can be expressed as
\begin{equation}\label{dirsumfreecm}
\underset{i\in I}{\bigoplus} \textrm{ }T_C(V_i)=T_C\left(\underset{i\in I}{\bigoplus}V_i\right)
\end{equation} for any family of vector spaces $\{V_i\}_{i\in I}$. 
We know  (see  \cite[\S 0.2.4]{semicontra}) that the category $\mathbf M_{[C,\_\_]}$ is abelian, has enough projectives and that the projective objects
in $\mathbf M_{[C,\_\_]}$ correspond to direct summands of free contramodules.  We note in particular that $C^*=Hom_K(C,K)$ is  a projective (free) $C$-contramodule.

\smallskip
Let $\alpha:C\longrightarrow D$ be a morphism of $K$-coalgebras. Then, we have a functor 
\begin{equation}
	\alpha_\bullet: \mathbf M_{[C,\_\_]}\longrightarrow \mathbf M_{[D,\_\_]}~~~~~~(\pi_M^C:Hom_K(C,M)\longrightarrow M)\mapsto (\pi_M^D:Hom_K(D,M)\xrightarrow{\_\_ \circ \alpha}Hom_K(C,M)\xrightarrow{\pi_M^C}M)
\end{equation}
This is called {\it contrarestriction of scalars}. We note that $\alpha_\bullet$ is exact.
The contrarestriction functor $\alpha_\bullet$ has a left adjoint $\alpha^\bullet$ called the {\it contraextension} of scalars \cite[$\S$ 4.8]{Pmem}. One can first define it on the free $D$-contramodules by setting
\begin{equation}\label{free}
	\alpha^\bullet(T_D(V)):=T_C(V)
\end{equation} for any vector space $V$. Then, $\alpha^\bullet$ may be extended to $\mathbf M_{[D,\_\_]}$ by using the fact that $\alpha^\bullet$ is right exact and that every $D$-contramodule may be expressed as a cokernel of free $D$-contramodules.

\smallskip
There is an equivalent definition of the contraextension of scalars which we describe now (see, \cite[\S 5.3]{co-contra}): Let $(P,\rho_P)$ be a right $C$-comodule and $(M,\pi_M)$ be a right $C$-contramodule. The $K$-space of cohomomorphisms $Cohom_C(P,M)$ is defined to be the coequalizer of the maps
$$\begin{tikzcd}[column sep=3cm, row sep=0.7cm]
Hom_K(P \otimes C,M) \cong Hom_K(P,Hom_K(C,M))\ar[r,shift left=.75ex,"{Hom_K(\rho_P,M)}"]
  \ar[r,shift right=.75ex,swap,"{Hom_K(P,\pi_M)}"]
&
Hom_K(P,M)
\\
\end{tikzcd}
$$
The contraextension $\alpha^\bullet:\mathbf M_{[D,\_\_]} \longrightarrow \mathbf M_{[C,\_\_]}$ may be defined as $\alpha^\bullet M=Cohom_D(C,M)$, where $C$ is considered as a right $D$-comodule. The right $C$-contramodule structure on $Cohom_D(C,M)$ is induced by the left $C$-comodule structure of $C$.

\begin{defn}\label{cont-rep}
	Let $\mathcal{C}: \mathscr{X}\longrightarrow Coalg$ be a coalgebra representation. A right trans-contramodule $\mathcal{M}$ over $\mathcal{C}$ will consist of the following data:
	\begin{itemize}
		\item[(1)] For each object $x\in Ob(\mathscr{X})$, a right $\mathcal{C}_x$-contramodule $\mathcal{M}_x$
		\item[(2)] For each morphism
		$\alpha: x \longrightarrow y$ in $\mathscr{X}$, a morphism
		$_{\alpha}{\mathcal{M}}: \mathcal{M}_y\longrightarrow {\alpha}_\bullet\mathcal{M}_x$ of right $\mathcal{C}_y$-contramodules (equivalently, a morphism $^{\alpha}\mathcal{M}: {\alpha}^\bullet\mathcal{M}_y\longrightarrow \mathcal{M}_x$ of right $\mathcal{C}_x$-contramodules)
	\end{itemize}
We further assume that $_{id_x}\mathcal{M} = id_{\mathcal{M}_x}$ and for any pair of composable morphisms
$x\xrightarrow{\alpha} y\xrightarrow{\beta}z$ in $\mathscr{X}$, we have $\beta_\bullet(_{\alpha}\mathcal{M})\circ {_{\beta}}\mathcal{M}= {_{\beta\alpha}}\mathcal{M}:\mathcal{M}_z\longrightarrow \beta_\bullet\mathcal{M}_y\longrightarrow \beta_\bullet\alpha_\bullet\mathcal{M}_x=(\beta\alpha)_\bullet\mathcal{M}_x$. The latter condition can be equivalently expressed as $^{\beta\alpha}\mathcal{M}={^{\alpha}}\mathcal{M}\circ \alpha^\bullet({^\beta}\mathcal{M})$.\\
A morphism $\eta: \mathcal{M} \longrightarrow \mathcal{N}$ of trans-contramodules over $\mathcal{C}$ consists of morphisms $\eta_x: \mathcal{M}_x\longrightarrow\mathcal{N}_x$ of right $\mathcal{C}_x$-contramodules for each $x\in Ob(\mathscr{X})$ such that for each morphism $x\longrightarrow y$ in $\mathscr{X}$ the following diagram commutes

\[\begin{tikzcd}
	\mathcal{M}_y \arrow{r}{\eta_y} \arrow[swap]{d}{_\alpha\mathcal{M}} & \mathcal{N}_y \arrow{d}{_\alpha\mathcal{N}} \\
	\alpha_\bullet\mathcal{M}_x\arrow{r}{\alpha_\bullet\eta_x} & \alpha_\bullet\mathcal{N}_x
\end{tikzcd}
\]
We denote this category of right trans-contramodules by $Cont^{tr}$-$\mathcal{C}$.
\end{defn}

\begin{rem} \emph{Unlike in the case of comodules, we only define trans-contramodules over coalgebra representations and not cis-contramodules. This is because, to the knowledge of the authors, there do not appear to be standard conditions in the literature that would make the contrarestriction functor $\alpha_\bullet$ a left adjoint. We note that if $\alpha=\epsilon_C:C\longrightarrow K$ is the counit, the contrarestriction $\alpha_\bullet$ reduces to the forgetful functor $\mathbf M_{[C,\_\_]}\longrightarrow Vect$ which does not preserve colimits in general.}

\end{rem}

\begin{thm}\label{Abcont}
	
	Let $\mathcal{C}: \mathscr{X}\longrightarrow Coalg$ be a  coalgebra representation. Then, $Cont^{tr}$-$\mathcal{C}$  is an abelian category.
	
\end{thm}

\begin{proof} For any morphism $\eta:\mathcal{M}\longrightarrow \mathcal{N}$ in $Cont^{tr}$-$\mathcal{C}$, we define the kernel and cokernel of $\eta$ by setting
	\begin{equation}\label{Feq5.8}
		{Ker(\eta)}_x:=Ker(\eta_x:\mathcal{M}_x\longrightarrow \mathcal{N}_x)\qquad {Coker(\eta)}_x:= Coker(\eta_x: \mathcal{M}_x\longrightarrow \mathcal{N}_x)
	\end{equation}
	for each $x\in \mathscr{X}$. For $\alpha:x\longrightarrow y$ in $\mathscr{X}$, the exactness of the contrarestriction functor $\alpha_\bullet: \mathbf M_{[\mathcal C_x,\_\_]}\longrightarrow \mathbf M_{[\mathcal C_y,\_\_]}$ induces the morphisms ${Ker(\eta)}_{\alpha}: {Ker(\eta)}_y\longrightarrow \alpha_\bullet{Ker(\eta)}_x$ and ${Coker(\eta)}_{\alpha}: {Coker(\eta)}_x\longrightarrow \alpha_\bullet{Coker(\eta)}_x$. Since each
	$\mathbf M_{[\mathcal C_x,\_\_]}$ is abelian, it also follows from \eqref{Feq5.8} that $Coker(Ker(\eta)\hookrightarrow \mathcal{M})= Ker (\mathcal{N}\twoheadrightarrow  Coker(\eta))$.

\end{proof}

We recall (see \cite[$\S$ 4.4]{BBW}) that for any coalgebra $C$ over a field $K$, the category of right $C$-contramodules is generated by the free (projective) right $C$-contramodule $C^*= Hom_K(C,K)$. Let $\mathcal{C}: \mathscr{X}\longrightarrow Coalg$ be a coalgebra representation. Let $\mathcal{M}\in Cont^{tr}$-$\mathcal{C}$. We consider an object $x\in Ob(\mathscr{X})$ and a morphism
\begin{equation}
	\eta: {\mathcal{C}_x^*}\longrightarrow  \mathcal{M}_x
\end{equation} in $\mathbf M_{[\mathcal C_x,\_\_]}$. 
For each object $y\in Ob(\mathscr{X})$, we set $\mathcal{N}_y\subseteq \mathcal{M}_y$ to be the image in $\mathbf M_{[\mathcal C_y,\_\_]}$ of the family of maps

\begin{align}\label{defN-cont}
	\mathcal{N}_y= Im\left(\bigoplus_{\beta \in \mathscr{X}(y,x)}\beta^\bullet({\mathcal{C}_x^*})=\bigoplus_{\beta \in \mathscr{X}(y,x)}\mathcal{C}_y^* \xrightarrow{\beta^\bullet\eta}\beta^\bullet\mathcal{M}_x\xrightarrow{^{\beta}\mathcal{M}}\mathcal{M}_y\right)
\end{align}
where the equality in \eqref{defN-cont} follows from \eqref{free}.
Let $i_y$ denote the inclusion $\mathcal{N}_y\hookrightarrow \mathcal{M}_y$  in $\mathbf M_{[\mathcal C_y,\_\_]}$  and for each $\beta\in \mathscr{X}(y,x)$, we denote by $\eta'_\beta:\beta^\bullet ({\mathcal{C}_x^*})={\mathcal{C}_y^*}\longrightarrow \mathcal{N}_y$ the canonical morphism induced from \eqref{defN-cont}.

\smallskip
For $\alpha \in \mathscr{X}(z,y)$ and $\beta\in \mathscr{X}(y,x)$, we may now verify as in the proof of Lemma \ref{lem1} that the following composition  in $\mathbf M_{[\mathcal C_y,\_\_]}$ 
	\begin{equation}\label{5.11Feq}
		\beta^\bullet({\mathcal{C}_x^*})={\mathcal{C}_y^*}\xrightarrow{\eta'_\beta} \mathcal{N}_y\xrightarrow{i_y}\mathcal{M}_y\xrightarrow{_\alpha\mathcal{M} }\alpha_\bullet\mathcal{M}_z
	\end{equation}
	factors through $\alpha_\bullet(i_z):\alpha_\bullet\mathcal{N}_z\longrightarrow \alpha_\bullet\mathcal{M}_z$.

	  \begin{thm}\label{Thm5.4R}
		Let $\mathcal{C}: \mathscr{X}\longrightarrow Coalg$ be a coalgebra representation and  let  $\mathcal{M}\in Cont^{tr}$-$\mathcal{C}$.  Then, the objects $\{\mathcal{N}_y\in \mathbf M_{[\mathcal C_y,\_\_]}\}_{y\in Ob(\mathscr{X})}$ together determine a subobject $\mathcal{N}\subseteq \mathcal M$ in $Cont^{tr}$-$\mathcal{C}$.		
	\end{thm}
	
	\begin{proof}
		Let $\alpha\in \mathscr{X}(z,y)$. Since $\alpha_\bullet$ is a right adjoint, it preserves monomorphisms and hence $\alpha_\bullet(i_z): \alpha_\bullet\mathcal{N}_z\longrightarrow \alpha_\bullet\mathcal{M}_z$ is a monomorphism in $\mathbf M_{[\mathcal C_y,\_\_]}$. We will now show that the morphism $_\alpha\mathcal{M}: \mathcal{M}_y\longrightarrow \alpha_\bullet\mathcal{M}_z$ restricts to a morphism $_\alpha\mathcal{N}: \mathcal{N}_y\longrightarrow \alpha_\bullet\mathcal{N}_z$ giving us a commutative diagram 
		
		\[\begin{tikzcd}
			\mathcal{N}_y \arrow{r}{i_y} \arrow[swap]{d}{_\alpha\mathcal{N}} & \mathcal{M}_y \arrow{d}{_\alpha\mathcal{M}} \\
			\alpha_\bullet\mathcal{N}_z \arrow{r}{\alpha_\bullet(i_z)} & \alpha_\bullet\mathcal{M}_z
		\end{tikzcd}
		\]
		
		We recall that the abelian category $\mathbf M_{[\mathcal C_y,\_\_]}$ has a projective generator $\mathcal{C}_y^*$. It is clear from the proof of \cite[Lemma 3.2]{AB} that  it suffices to show that for any morphism $\zeta:\mathcal{C}_y^*\longrightarrow \mathcal{N}_y$ in $\mathbf M_{[\mathcal C_y,\_\_]}$ there exists $\zeta':\mathcal{C}_y^*\longrightarrow \alpha_\bullet\mathcal{N}_z$ such that $\alpha_\bullet(i_z)\circ \zeta'= {_\alpha}\mathcal{M}\circ i_y\circ \zeta$. By \ref{defN-cont}, we have an epimorphism
		\begin{equation}
			\bigoplus_{\beta \in \mathscr{X}(y,x)}\eta'_\beta: \bigoplus_{\beta\in \mathscr{X}(y,x)}{\mathcal{C}_y^*}\longrightarrow \mathcal{N}_y
		\end{equation}
		 in $\mathbf M_{[\mathcal C_y,\_\_]}$. Since $\mathcal{C}_y^*$ is projective, the morphism $\zeta:\mathcal{C}_y^*\longrightarrow \mathcal{N}_y$ can be lifted to a morphism  $ {\zeta}'': {\mathcal{C}_y^*} \longrightarrow \bigoplus_{\beta\in \mathscr{X}(y,x)}{\mathcal{C}_y^*}$ such that
		\begin{equation}\label{eq1-cont}
			\zeta = \left(\bigoplus_{\beta \in \mathscr{X}(y,x)}{\eta_\beta}'\right)\circ \zeta''
		\end{equation}
		We know from \eqref{5.11Feq} that $_\alpha\mathcal{M}\circ i_y\circ \eta_\beta'$ factors through $\alpha_\bullet(i_z):\alpha_\bullet\mathcal{N}_z\longrightarrow \alpha_\bullet\mathcal{M}_z$ for each $\beta \in \mathscr{X}(y,x)$. Therefore, it follows from \eqref{eq1-cont} that $_\alpha\mathcal{M}\circ i_y\circ \zeta$ factors through $\alpha_\bullet(i_z)$ as required. This proves the result.
		
	\end{proof}

\begin{thm}\label{lem2-cont}
	Let $\eta_1': {\mathcal{C}_x^*}\longrightarrow \mathcal{N}_x$ be the canonical morphism corresponding to the identity map in $\mathscr{X}(x,x)$. Then, for any $y\in Ob(\mathscr{X})$, we have
	\begin{equation}
		\mathcal{N}_y= Im \left(	\bigoplus_{\beta \in \mathscr{X}(y,x)}\ {\mathcal{C}_y^*}=\bigoplus_{\beta \in \mathscr{X}(y,x)}\beta^\bullet({\mathcal{C}_x^*}) 
	\xrightarrow{\beta^\bullet\eta_1'}\beta^\bullet\mathcal{N}_x\xrightarrow{^{\beta}\mathcal{N}}\mathcal{N}_y \right)
	\end{equation}
	
\end{thm}
\begin{proof}
	Let $\beta \in \mathscr{X}(y,x)$. We consider the following commutative diagram:
	\begin{equation}\label{commdiag-cont}
		\begin{tikzcd}
			\beta^\bullet({\mathcal{C}_x^*}) \arrow{r}{\beta^\bullet\eta_1'} & \beta^\bullet\mathcal{N}_x \arrow{r}{^\beta\mathcal{N}} \arrow[swap]{d}{\beta^\bullet  i_x} & \mathcal{N}_y \arrow{d}{i_y}\\
			& \beta^\bullet\mathcal{M}_x \arrow{r}{^\beta\mathcal{M}} & \mathcal{M}_y
		\end{tikzcd}
	\end{equation}
	Since $i_x \circ \eta_1'=\eta$, we have $(\beta^\bullet i_x) \circ (\beta^\bullet\eta_1')=\beta^\bullet\eta$. This gives
	\begin{equation*}
	\begin{array}{ll}
		Im\left(\bigoplus_{\beta \in \mathscr{X}(y,x)}{^\beta}\mathcal{M} \circ (\beta^\bullet\eta) \right)&=Im\left(\bigoplus_{\beta \in \mathscr{X}(y,x)}{^\beta}\mathcal{M} \circ  (\beta^\bullet i_x) \circ (\beta^\bullet\eta_1') \right)\\&=Im \left(\bigoplus_{\beta \in \mathscr{X}(y,x)}i_y \circ  {^\beta}\mathcal{N} \circ (\beta^\bullet\eta_1') \right)=Im \left(\bigoplus_{\beta \in \mathscr{X}(y,x)} {^\beta}\mathcal{N} \circ (\beta^\bullet\eta_1') \right)\\
		\end{array}
	\end{equation*}
	where the last equality follows from the fact that $i_y$ is a monomorphism of contramodules. The result now follows using the definition in \eqref{defN-cont}.
\end{proof}

For $\mathcal{M}\in Cont^{tr}$-$\mathcal{C}$,  let  $el_{\mathscr{X}}(\mathcal{M})$ denote the union $\cup_{x\in Ob(\mathscr{X})} \mathcal{M}_x$ as sets. The cardinality of $el_{\mathscr{X}}(\mathcal{M})$ will be denoted by $|\mathcal{M}|$. For any coalgebra $C$, we recall that the forgetful functor $\mathbf M_{[C,\_\_]}\longrightarrow Vect$ is exact. Hence, epimorphisms and monomorphisms in $\mathbf M_{[C,\_\_]}$ correspond respectively to epimorphisms and monomorphisms in $Vect$. It follows  that for any quotient or subobject $\mathcal{N}$ of $\mathcal{M}\in Cont^{tr}$-$\mathcal{C}$, we have $|\mathcal{N}|\leq |\mathcal{M}|$. Now, we set
\begin{equation}
	\kappa:=sup\{\aleph_0, (|K|^{|Mor(\mathscr{X})|})^{|\mathcal{C}_y|}~|~ y\in Ob(\mathscr{X})\}
\end{equation}

\begin{lem}\label{cardinal-cont} Let $\mathcal N$ be as constructed in Proposition \ref{Thm5.4R}. Then, 
	we have $|\mathcal{N}|\leq \kappa$. 
	
\end{lem}
\begin{proof}
	Let $y\in Ob(\mathscr{X})$. From Proposition \ref{lem2-cont}, we have 
	\begin{equation}
		\mathcal{N}_y= Im \left(\bigoplus_{\beta \in \mathscr{X}(y,x)}\mathcal{C}_y^* \xrightarrow{\beta^\bullet\eta_1'}\beta^\bullet\mathcal{N}_x\xrightarrow{{^\beta}\mathcal{N}}\mathcal{N}_y \right)
	\end{equation}
	From the adjunction in \eqref{5adjcontra}, we know that the left adjoint $T_{\mathcal C_y}: Vect\longrightarrow \mathbf M_{[\mathcal C_y,\_\_]}$ preserves direct sums. Since $\mathcal C_y^*=Hom_K(\mathcal C_y,K)$ is a free contramodule, we therefore have
	\begin{equation} \bigoplus_{\beta \in \mathscr{X}(y,x)}\mathcal{C}_y^*=Hom_K\left(\mathcal{C}_y, \bigoplus_{\beta \in \mathscr{X}(y,x)}K\right)
	\end{equation}  Since $\mathcal{N}_y$ is an epimorphic image of $\bigoplus_{\beta \in \mathscr{X}(y,x)}\mathcal{C}_y^*$, we now have
	\begin{equation}\label{directsumoffreecont}
		|\mathcal{N}_y|\leq |\bigoplus_{\beta \in \mathscr{X}(y,x)}\mathcal{C}_y^*|= |Hom_K(\mathcal{C}_y, \bigoplus_{\beta \in \mathscr{X}(y,x)}K)|\leq \kappa
	\end{equation}
Thus, $|\mathcal{N}|=\sum_{y\in Ob(\mathscr{X})}|\mathcal{N}_y|\leq \kappa$.
\end{proof}

\begin{Thm}\label{Grothcont}
	Let $\mathcal{C}: \mathscr{X}\longrightarrow Coalg$ be a coalgebra representation. Then, $Cont^{tr}$-$\mathcal{C}$ is cocomplete and has a set of generators.
\end{Thm}
\begin{proof} We note that colimits in $Cont^{tr}$-$\mathcal{C}$  exist and can be computed pointwise. 
	  Let $\mathcal{M}$ be an object in $Cont^{tr}$-$\mathcal{C}$ and $m \in el_{\mathscr{X}}(\mathcal{M})$. Then, there exists some $x\in Ob(\mathscr{X})$ such that $m \in \mathcal{M}_x$. From \eqref{5adjcontra}, we note that $\mathbf M_{[\mathcal C_x,\_\_]}(\mathcal C_x^*,M)\cong Vect(K,\mathcal M_x)=\mathcal M_x$ and hence  there exists a morphism $\eta: \mathcal{C}_x^*\longrightarrow \mathcal{M}_x$ in $\mathbf M_{[\mathcal C_x,\_\_]}$ such that $m\in Im(\eta)$. Thus, it follows from \eqref{defN-cont} that we can define 
	  $\mathcal N\subseteq \mathcal M$ in $Cont^{tr}$-$\mathcal{C}$ such that
	\begin{equation}
		m\in Im\left(\eta: \mathcal{C}_x^*\longrightarrow \mathcal{M}_x\right)\subseteq Im\left(\bigoplus_{\beta \in \mathscr{X}(x,x)}\beta^\bullet ({\mathcal{C}_x^*}) \xrightarrow{^\beta\mathcal{M}\circ \beta^\bullet\eta}\mathcal{M}_x\right)= \mathcal{N}_x
	\end{equation}
	
    By Lemma \ref{cardinal-cont}, we also have  
    $|\mathcal{N}|\leq \kappa$. For convenience, let us denote this subobject  $\mathcal{N}\subseteq \mathcal{M}$ by $\mathcal{N}_m$. We claim that $\bigoplus_{m'\in el_{\mathscr{X}}(\mathcal M)}{\mathcal{N}_{m'}}\longrightarrow \mathcal{M}$ is an epimorphism  in $Cont^{tr}$-$\mathcal{C}$. For this,  it is enough to show $\bigoplus_{m'\in el_{\mathscr{X}}(\mathcal M)}{\mathcal{N}_{m'x}}\longrightarrow \mathcal{M}_x$ is an epimorphism in $\mathbf M_{[\mathcal C_x,\_\_]}$ for each $x\in Ob(\mathscr{X})$. If $m\in \mathcal{M}_x$ for some $x\in Ob(\mathscr{X})$, clearly we have
    \begin{equation}
    	m\in Im\left({\mathcal{N}_{mx}}\hookrightarrow \mathcal{M}_x\right)\subseteq Im\left(\bigoplus_{m'\in el_{\mathscr X}(\mathcal M)}\mathcal{N}_{m'x}\longrightarrow \mathcal{M}_x\right)
    \end{equation} 
Thus, by choosing one object from each isomorphism class of objects in $Cont^{tr}$-$\mathcal{C}$ having cardinality $\leq \kappa$, we obtain a set of generators for $Cont^{tr}$-$\mathcal{C}$. This proves the result.
	
\end{proof}

We continue with $\mathcal{C}: \mathscr{X}\longrightarrow Coalg$ being a coalgebra representation. For the rest of this section, we assume that $\mathscr X$ is a poset. We will now show that $Cont^{tr}$-$\mathcal{C}$ is a locally presentable category with projective generators.

\begin{thm}\label{adjoint-cont}
	Let $\mathscr{X}$ be a poset and $\mathcal{C}: \mathscr{X}\longrightarrow Coalg$ be a coalgebra representation. Let $x\in Ob(\mathscr{X})$. Then,
	
	\smallskip
(1) There is a functor $ex_x:\mathbf M_{[\mathcal C_x,\_\_]}\longrightarrow Cont^{tr}$-$\mathcal{C}$ defined by setting, for any $y\in Ob(\mathscr X)$, 
		\[
		(ex_x(N))_y=
		\begin{cases}
			\alpha^\bullet(N) &\quad \text{if~} \alpha\in \mathscr{X}(y,x)\\
			0&\quad \text{if~}\mathscr{X}(y,x)=\emptyset\\
		\end{cases}
		\]
		(2) The evaluation at $x$, $ev_x:Cont^{tr}$-$\mathcal{C}\longrightarrow \mathbf M_{[\mathcal C_x,\_\_]}$, $\mathcal{M}\mapsto \mathcal{M}_x$ is an exact functor.
		
		\smallskip
		(3) $(ex_x, ev_x)$ is a pair of adjoint functors.
		
		\smallskip

		\smallskip
		(4) The functor $ev_x$ also has a right adjoint, $coe_x:  \mathbf M_{[\mathcal C_x,\_\_]}\longrightarrow Cont^{tr}\text{-}\mathcal{C}$ given by setting
		
		\begin{equation*}
		(coe_x(N))_y:=\begin{cases}
			\alpha_\bullet N \quad \quad\text{if}~ \alpha \in \mathscr{X}(x,y)\\
			0 \quad \quad \quad \quad \text{if}~ \mathscr{X}(x,y)=\emptyset
		\end{cases}
	\end{equation*}

	\end{thm}

\begin{proof} (1), (2) and (3) follow as in the proof of Proposition \ref{poprojc}. The result of (4) is proved in a manner similar to Proposition \ref{poprojk}.
\end{proof}

\begin{cor}\label{proj-cont}
Let $\mathscr{X}$ be a poset and $\mathcal{C}: \mathscr{X}\longrightarrow Coalg$ be a coalgebra representation. For $x\in Ob(\mathscr{X})$, the functor $ex_x: \mathbf M_{[\mathcal C_x,\_\_]}\longrightarrow Cont^{tr}$-$\mathcal{C}$ preserves projectives.
\end{cor}

\begin{proof}
We know from Proposition \ref{adjoint-cont}(2) that $ev_x:Cont^{tr}$-$\mathcal{C}\longrightarrow \mathbf M_{[\mathcal C_x,\_\_]}$ is an exact functor. By Proposition \ref{adjoint-cont}(3), we know that $(ex_x, ev_x)$ is an adjoint pair and hence the left adjoint $ex_x$ preserves projective objects.
\end{proof}	

We recall that an object $X$ in a category $\mathscr{D}$  is said to be $\lambda$-presentable for some regular cardinal $\lambda$ if the representable functor $\mathscr{D}(X,{-})$ preserves $\lambda$-directed colimits (see \cite[$\S$ 1.13]{AR}). We note that for regular cardinals $\lambda\leq \lambda'$, any $\lambda$-presentable object is also $\lambda'$-presentable. A category $\mathscr{D}$ is called locally $\lambda$-presentable if it is cocomplete and it has a set of $\lambda$-presentable generators. A category $\mathscr{D}$ is said to be  locally presentable if $\mathscr{D}$ is   locally $\lambda$-presentable for some regular cardinal $\lambda$. For further details on locally presentable categories, we refer the reader to \cite{AR}.

\begin{Thm}\label{projgen-cont}
Let $\mathscr{X}$ be a poset and $\mathcal{C}: \mathscr{X}\longrightarrow Coalg$ be a coalgebra representation. Then, $Cont^{tr}$-$\mathcal{C}$ is a locally presentable abelian category with projective generators.
  \end{Thm}
  \begin{proof}
  If $C$ is any $K$-coalgebra, we know (see \cite[$\S$ 5]{P2}) that the category $\mathbf M_{[C,\_\_]}$ of contramodules is locally presentable. Indeed, $C^*=Hom_K(C,K)$ is itself a $\lambda$-presentable generator
  in $\mathbf M_{[C,\_\_]}$ for some $\lambda$. We suppose therefore that for each $x\in Ob(\mathscr X)$, the object $\mathcal C_x^*$ is a $\lambda_x$-presentable generator for $\mathbf M_{[\mathcal C_x,\_\_]}$.  We set
  \begin{equation*}
  \lambda_0:= sup\{\lambda_x~|~ x\in \mathscr{X}\}
  \end{equation*} Then, it is clear that each $\mathcal C_x^*$ is a $\lambda_0$-presentable generator in $\mathbf M_{[\mathcal C_x,\_\_]}$.  Since $\mathcal C_x^*$ is projective in $\mathbf M_{[\mathcal C_x,\_\_]}$, it follows by Corollary \ref{proj-cont} that $ex_x(\mathcal{C}_x^*)$ is projective in $Cont^{tr}$-$\mathcal{C}$. From the definitions and by \eqref{free}, we know that
 \[
		(ex_x(\mathcal{C}_x^*))_y =
		\begin{cases}
			\alpha^\bullet(\mathcal{C}_x^*)=\mathcal{C}_y^* &\quad \text{if~} \alpha\in \mathscr{X}(y,x)\\
			0&\quad \text{if~}\mathscr{X}(y,x)=\emptyset\\
		\end{cases}
		\]
for $y\in Ob(\mathscr X)$. We claim that the family 
    \begin{equation}
\mathcal{G}= \{ex_x(\mathcal{C}_x^*)~|~x\in Ob(\mathscr{X})\}
\end{equation}
is a set of $\lambda_0$-presentable generators for $Cont^{tr}$-$\mathcal{C}$. For this, we consider a non-invertible monomorphism $i:\mathcal{N}\hookrightarrow \mathcal{M}$ in $Cont^{tr}$-$\mathcal{C}$. Since kernels and cokernels in  $Cont^{tr}$-$\mathcal{C}$ are constructed pointwise, there exists some $x\in Ob(\mathscr{X})$ such that $i_x:\mathcal{N}_x\hookrightarrow \mathcal{M}_x$ is a non-invertible monomorphism in $\mathbf M_{[\mathcal C_x,\_\_]}$. Since
$\mathcal{C}_x^*$ is a generator of $\mathbf M_{[\mathcal C_x,\_\_]}$, we can choose a morphism $f:\mathcal{C}_x^*\longrightarrow \mathcal{M}_x=ev_x(\mathcal M)$ such that $f$ does not factor through $i_x:\mathcal{N}_x\hookrightarrow \mathcal{M}_x$. Since $(ex_x, ev_x)$ is an adjoint pair, this induces a morphism $\eta_f: ex_x(\mathcal{C}_x^*)\longrightarrow \mathcal{M}$ in $Cont^{tr}$-$\mathcal{C}$ corresponding to $f$, which does not factor through $i:\mathcal{N}\longrightarrow \mathcal{M}$. Therefore, it follows from \cite[\S 1.9]{Tohuku} that the family $\mathcal{G}$ is a set of generators for $Cont^{tr}$-$\mathcal{C}$.

\smallskip
Finally, let $\{\mathcal M_i\}_{i\in I}$ be a $\lambda_0$-directed system of objects in $Cont^{tr}$-$\mathcal{C}$. Since each $\mathcal C_x^*$ is $\lambda_0$-presentable in $\mathbf M_{[\mathcal C_x,\_\_]}$ and colimits in $Cont^{tr}$-$\mathcal{C}$ are computed pointwise, we note that
\begin{equation*}
\begin{array}{ll}
Cont^{tr}\text{-}\mathcal C(ex_x(\mathcal{C}_x^*),\underset{i\in I}{\varinjlim}\textrm{ }\mathcal M_i)&=\mathbf M_{[\mathcal C_x,\_\_]}(\mathcal C_x^*,ev_x(\underset{i\in I}{\varinjlim}\textrm{ }\mathcal M_i))=\mathbf M_{[\mathcal C_x,\_\_]}(\mathcal C_x^*,\underset{i\in I}{\varinjlim}\textrm{ }ev_x(\mathcal M_i))\\
& =\underset{i\in I}{\varinjlim}\textrm{ }\mathbf M_{[\mathcal C_x,\_\_]}(\mathcal C_x^*,ev_x(\mathcal M_i))=\underset{i\in I}{\varinjlim}\textrm{ }Cont^{tr}\text{-}\mathcal C(ex_x(\mathcal{C}_x^*),\mathcal M_i)\\
\end{array}
\end{equation*} This shows that each $ex_x(\mathcal{C}_x^*)$ is $\lambda_0$-presentable in  $Cont^{tr}$-$\mathcal{C}$. This proves the result.
\end{proof}

\section{Cartesian trans-contramodules over coalgebra representations}	
Let  $\mathcal{C}:\mathscr{X} \longrightarrow Coalg$ be a coalgebra representation. We will now introduce the category  $Cont^{tr}_c\text{-}\mathcal{C}$  of cartesian trans-contramodules over $\mathcal{C}$ and study its generators.  While the outline is roughly similar to that in Section 4.2, we rely extensively on  presentable objects in contramodule categories to set up the transfinite induction argument. This is because of the fact that colimits in the category of contramodules do not correspond in general to colimits of underlying vector spaces.

\begin{lem}\label{coextex}
Let $\alpha:C \longrightarrow D$ be a right coflat morphism of coalgebras, i.e., $C$ is coflat as a right $D$-comodule. Then, the contraextension functor $\alpha^\bullet:\mathbf M_{[D,\_\_]} \longrightarrow \mathbf M_{[ C,\_\_]}$ is exact.
\end{lem}
\begin{proof}
This follows from \cite[\S 3.1, Lemma ]{P}.
\end{proof}

\begin{defn}
Let $\mathcal{C}:\mathscr{X} \longrightarrow Coalg$ be a right coflat representation. Let $\mathcal{M}\in Cont^{tr}$-$\mathcal{C}$. We will say that $\mathcal{M}$ is cartesian if for each $\alpha:x \longrightarrow y$ in $\mathscr{X}$, the morphism $^\alpha\mathcal{M}:\alpha^\bullet\mathcal{M}_y=Cohom_{\mathcal{C}_y}(\mathcal{C}_x,\mathcal{M}_y) \longrightarrow \mathcal{M}_x$ in $\mathbf M_{[\mathcal C_x,\_\_]}$ is an isomorphism.
We will denote by $Cont^{tr}_c\text{-}\mathcal{C}$ the full subcategory of $Cont^{tr}\text{-}\mathcal{C}$ whose objects are cartesian trans-contramodules over $\mathcal{C}$.
\end{defn}

\begin{lem}
Let $\mathcal{C}:\mathscr{X} \longrightarrow Coalg$ be a right coflat representation. Then, $Cont^{tr}_c\text{-}\mathcal{C}$ is a cocomplete abelian category.
\end{lem}
\begin{proof}
Let $\eta: \mathcal{M}\longrightarrow \mathcal{N}$ be a morphism in $Cont^{tr}_c\text{-}\mathcal{C}$. Then, its kernel and cokernel in $Cont^{tr}\text{-}\mathcal{C}$ are respectively given by 
\begin{equation*}
Ker(\eta)_x= Ker(\eta_x:\mathcal{M}_x\longrightarrow \mathcal{N}_x)\qquad Coker(\eta)_x= Coker(\eta_x:\mathcal{M}_x\longrightarrow \mathcal{N}_x)
\end{equation*}
 Since $\mathcal{C}$ is a coflat representation,   it follows by Lemma \ref{coextex}  that the contraextension functor $\alpha^\bullet:\mathbf M_{[\mathcal C_y,\_\_]} \longrightarrow \mathbf M_{[\mathcal C_x,\_\_]}$ is exact for any $\alpha:x\longrightarrow y$ in $\mathscr X$. From this, it is clear that $^\alpha Ker(\eta): \alpha^\bullet Ker(\eta)_y \longrightarrow Ker(\eta)_x$ and $^\alpha Coker(\eta): \alpha^\bullet Coker(\eta)_y \longrightarrow Coker(\eta)_x$ are isomorphisms.   It follows that  $Cont^{tr}_c\text{-}\mathcal{C}$  is abelian.
 
 \smallskip
Now let $\{\mathcal{M}_i\}_{i\in I}$ be a  family of objects in $Cont^{tr}_c\text{-}\mathcal{C}$ and take $\alpha \in \mathscr{X}(x,y)$.   Since $\alpha^\bullet$ is a left adjoint, it preserves direct sums. Using the fact that $ ^\alpha\mathcal{M}_i:\alpha^\bullet {\mathcal{M}_i}_y \longrightarrow {\mathcal{M}_i}_x$ is an isomorphism for each $i \in I$, we see that
\begin{equation*}
\begin{tikzcd}[row sep=large, column sep=8ex]
\alpha^\bullet \left(\bigoplus\limits_{i \in I} \mathcal{M}_i\right)_y=\bigoplus\limits_{i \in I} \alpha^\bullet {\mathcal{M}_i}_y \arrow{r}{\bigoplus\limits_{i \in I} {^\alpha}\mathcal{M}_i}&  \bigoplus\limits_{i \in I} {\mathcal{M}_i}_x= \left(\bigoplus\limits_{i \in I}  {\mathcal{M}_i} \right)_x
\end{tikzcd}
\end{equation*}
is an isomorphism in $\mathbf M_{[\mathcal C_x,\_\_]}$. Combining with the fact that $Cont^{tr}_c\text{-}\mathcal{C}$  has cokernels, it follows that $Cont^{tr}_c\text{-}\mathcal{C}$  is cocomplete.
\end{proof}

\begin{lem}\label{sumcontrcard}
(a) Let $C$ be a $K$-coalgebra and let $\{M_i\}_{i\in I}$ be a system of objects in $\mathbf M_{[C,\_\_]}$. Let $\kappa'$ be a cardinal such that $\kappa'\geq sup\{\aleph_0,|K|,|I|,|M_i|,i\in I\}$ Then, if $M=\underset{i\in I}{colim}\textrm{ }M_i$, we have $|M|\leq \kappa'^{dim(C)}$.

\smallskip
(b) Let  $\mathcal{C}:\mathscr{X} \longrightarrow Coalg$ be a coalgebra representation and $\{\mathcal M_i\}_{i\in I}$ be a system of objects in $Cont^{tr}\text{-}\mathcal{C}$.  Let $\kappa'\geq sup\{\aleph_0,|K|,|I|,|Mor(\mathscr X)|,|\mathcal M_i|,i\in I\}$ and $\lambda\geq sup\{\aleph_0, dim(\mathcal C_x), x\in Ob(\mathscr X)\}$. Then, if $\mathcal M=\underset{i\in I}{colim}\textrm{ }\mathcal M_i$, we have $|\mathcal M|\leq \kappa'^{\lambda}$.
\end{lem}
\begin{proof} We set $N:=\bigoplus_{i\in I}M_i$. Then, there is an epimorphism $N\longrightarrow M$ in $\mathbf M_{[C,\_\_]}$ and it therefore suffices to show that $|N|\leq \kappa'^{dim(C)}$. For each $i\in I$, any element $m\in M_i=Vect(K,M_i)=\mathbf M_{[C,\_\_]}(C^*,M_i)$ corresponds to a morphism $C^*\longrightarrow M_i$ whose image contains $m$. Together, these induce an epimorphism $C^{*(M_i)}\longrightarrow M_i$ in $\mathbf M_{[C,\_\_]}$. Since direct sums preserve epimorphisms, we see that $N$ becomes a quotient
of $\bigoplus_{i\in I}C^{*(M_i)}$ in $\mathbf M_{[C,\_\_]}$. Applying \eqref{dirsumfreecm}, we see that the direct sum of free contramodules $\bigoplus_{i\in I}C^{*(M_i)}=T_C(V)$, where $V$ is a vector space of dimension 
$\sum_{i\in I}|M_i|\leq \kappa'$.  It follows that $|M|\leq |N|\leq |T_C(V)|\leq \kappa'^{dim(C)}$. This proves (a). The result of (b) follows directly from (a).
\end{proof}

\begin{lem}\label{transkappanew}
Let $\alpha: C\longrightarrow D$ be a right coflat morphism of $K$-coalgebras. Let $\lambda_C$ be a regular cardinal such that $C^*$ is $\lambda_C$-presentable in $\mathbf M_{[C,\_\_]}$. Let $\kappa'\geq max\{\aleph_0, \lambda_C , |K|\}$. Let $M \in \mathbf M_{[D,\_\_]}$ and $A\subseteq \alpha^\bullet M$ be a set of elements such that $|A|\leq \kappa'$. Then, there exists a subobject $N\subseteq M$ in $\mathbf M_{[D,\_\_]}$ with $|N|\leq \kappa'^{dim(D)}$ such that $A\subseteq \alpha^\bullet N$.
\end{lem}
\begin{proof}
Let $a \in A \subseteq \alpha^\bullet M$.  From \eqref{5adjcontra}, we note that $\mathbf M_{[C,\_\_]}(C^*,\alpha^\bullet M)\cong Vect(K,\alpha^\bullet M)=\alpha^\bullet M$ and hence there exists a morphism $\eta^a\in \mathbf M_{[C,\_\_]}(C^*,\alpha^\bullet M)$ such that $a \in Im(\eta^a)$. Since $D^*$ is a generator in $\mathbf M_{[D,\_\_]}$, we can choose an epimorphism $\zeta: D^{*(I)} \longrightarrow M$ in $\mathbf M_{[D,\_\_]}$ for some indexing set $I$. We may suppose that $|I|>\lambda_C$. Since $\alpha^\bullet$ is a left adjoint and therefore preserves colimits, we see that $\alpha^\bullet\zeta: (\alpha^\bullet D^*)^{(I)} \longrightarrow \alpha^\bullet M$ is an epimorphism in $\mathbf M_{[C,\_\_]}$.  Since $C^*$ is projective, the morphism $\eta^a$ can be lifted to a morphism $\eta'^{a}:C^* \longrightarrow (\alpha^\bullet D^*)^{(I)}=C^{*(I)} $.

\smallskip
Because $\lambda_C$ is a regular cardinal, we note that the collection of subsets of $I$ with cardinality $<\lambda_C$ is a $\lambda_C$-directed system. 
Because $C^*$ is $\lambda_C$-presentable in $\mathbf M_{[C,\_\_]}$, there exists a subset $J_a\subseteq I$ with $|J_a| < \lambda_C$ such that $\eta'^{a}$ factors through ${C^*}^{(J_a)}$ and we obtain the following commutative diagram:
\begin{equation*}
\begin{tikzcd}[row sep=large, column sep=14ex]
C^* \ar{r}{} \ar[swap]{d}{\eta^a}  \ar{rd}{\eta'^a}&(\alpha^\bullet D^*)^{(J_a)} =C^{*(J_a)} \ar{d}\\
\alpha^\bullet M &(\alpha^\bullet D^*)^{(I)} =C^{*(I)}\ar{l}{\alpha^\bullet \zeta}
\end{tikzcd}
\end{equation*}
Thus, we obtain a morphism $\zeta'^a: D^{*(J_a)}\longrightarrow M$ in $\mathbf M_{[D,\_\_]}$ such that $\eta^a$ factors through $\alpha^\bullet\zeta'^a$. In  $\mathbf M_{[D,\_\_]}$, we set
\begin{equation*}
N:=Im~\left(\zeta'=\bigoplus_{a \in A} \zeta'^a: \bigoplus_{a \in A} D^{*(J_a)} \longrightarrow M   \right)
\end{equation*}

Since $\alpha^\bullet$ is exact and preserves colimits, we also have 
\begin{equation*}
\alpha^\bullet N:=Im~\left(\alpha^\bullet \zeta'=\bigoplus_{a \in A} \alpha^\bullet\zeta'^a: \bigoplus_{a \in A} \alpha^\bullet D^{*(J_a)} \longrightarrow \alpha^\bullet M   \right)
\end{equation*}
 Since $a \in Im(\eta^a)$ and $\eta^a$ factors through $\alpha^\bullet\zeta'^a$, we have $a \in \alpha^\bullet N$. Thus, $A \subseteq \alpha^\bullet N$. By \eqref{dirsumfreecm}, we know that the direct sum $\bigoplus_{a \in A}D^{*(J_a)}$ of free contramodules is given by $T_D(V)$, where $V$ is the vector space
 $\bigoplus_{a \in A}K^{|J_a|}$. Since $N$ is a quotient of $T_D(V)$ and $|V|\leq \kappa'$, it follows  that $|N| \leq \kappa'^{dim (D)}$.
\end{proof}

\begin{lem}\label{transdoublecont}
  Let  $\alpha: C\longrightarrow D$ be a right coflat morphism of $K$-coalgebras and let $M\in \mathbf M_{[D,\_\_]}$. Let $\lambda$ be a regular cardinal such that $C^*$ and $D^*$ are $\lambda$-presentable in $\mathbf M_{[C,\_\_]}$ and $\mathbf M_{[D,\_\_]}$ respectively. We also choose cardinals 
  \begin{equation*} \kappa'\geq max\{\aleph_0, \lambda,|K|\}\qquad \mu\geq max\{dim(C),dim(D),\aleph_0\}
  \end{equation*} Let $A\subseteq M$ and $B\subseteq \alpha^\bullet M$ be such that $|A|,|B|\leq \kappa'^\mu$. Then, there exists a subobject $N\subseteq M$ in $\mathbf M_{[D,\_\_]}$ such that
  \begin{itemize}
  \item[(1)] $A\subseteq N$ and $B\subseteq \alpha^\bullet N$
  \item[(2)] $|N|\leq \kappa'^\mu$, $|\alpha^\bullet N|\leq \kappa'^\mu$ 
  \end{itemize}
  \end{lem}  
\begin{proof}
By Lemma \ref{transkappanew}, we know that there exists a $D$-subcontramodule $N_1 \subseteq M$ such that $|N_1| \leq (\kappa'^\mu)^{dim(D)}=\kappa'^\mu$ and $B \subseteq \alpha^\bullet N_1$.  Taking $\alpha=id_D: D\longrightarrow D$ in Lemma \ref{transkappanew}, we can also obtain a $D$-subcontramodule $N_2 \subseteq M$ such that $|N_2|\leq  (\kappa'^\mu)^{dim(D)}=\kappa'^\mu$ and $A\subseteq N_2$. We set $N:= N_1+N_2\subseteq M$. Since $N$ is a quotient of $N_1\oplus N_2$, it follows from Lemma \ref{sumcontrcard}(a) that $|N|\leq |N_1\oplus N_2|\leq (\kappa'^\mu)^{dim(D)}=\kappa'^\mu$. We also have $ A\subseteq N_2\subseteq N$. Since  $\alpha^\bullet$ is exact, it preseves monomorphims and hence $B\subseteq \alpha^\bullet N_1\subseteq \alpha^\bullet N$. By definition, we know that $\alpha^\bullet N=Cohom_D(C,N)$ is the cokernel of the difference of the maps
\begin{equation*}
\begin{tikzcd}
 Hom_K(C \otimes C,N) \simeq Hom_K(C,Hom_K(C,N)) \ar[r,shift left=.75ex,""]
  \ar[r,shift right=.75ex,swap,""]
&
Hom_K(C,N)
\end{tikzcd}
\end{equation*}
Therefore, $|\alpha^\bullet N|=|Cohom_D(C,N)| \leq |Hom_K(C,N)| = |N|^{dim(C)}\leq (\kappa'^{\mu})^{dim(C)}=\kappa'^\mu$. 
\end{proof}

Let  $\mathscr X$ be a poset and let $\mathcal{C}:\mathscr{X} \longrightarrow Coalg$ be a right coflat representation. We choose $\lambda$ such that each $\mathcal C_x^*$ is $\lambda$-presentable in 
$\mathbf M_{[\mathcal C_x,\_\_]}$ for  $x\in Ob(\mathscr X)$. We now set
\begin{eqnarray*}
&	\kappa:= sup\{\aleph_0, |K|,  \lambda, |Mor(\mathscr{X})|, (|K|^{|Mor(\mathscr{X})|})^{|\mathcal{C}_x|}, x\in Ob(\mathscr{X})\} \\
&\mu:=sup\{{\aleph_0,dim(\mathcal{C}_x)}, x \in Ob(\mathscr{X})\}
\end{eqnarray*}

We now choose a well ordering of the set $Mor(\mathscr{X})$ and consider the induced lexicographic ordering of $\mathbb{N}\times Mor(\mathscr{X})$.  Let $\mathcal{M}\in Cont^{tr}_c$-$\mathcal{C}$ and let $m_0 \in el_{\mathscr{X}}(\mathcal{M})$, i.e. $m_0 \in \mathcal{M}_x$ for some $x\in Ob(\mathscr{X})$.
We will now define a family of subobjects  $\{\mathcal{P}(n,\alpha) | (n,\alpha: y\longrightarrow z)\in \mathbb{N}\times Mor(\mathscr{X})\}$ of $\mathcal{M}$ in $Cont^{tr}\text{-}\mathcal{C}$ which satisfies the following conditions:

\smallskip
(1'') $m_0 \in el_{\mathscr{X}}(\mathcal{P}(1,\alpha_0))$, where $\alpha_0$ is the least element of $Mor(\mathscr{X})$.

\smallskip
(2'') $\mathcal{P}(n,\alpha)\subseteq \mathcal{P}(m,\beta)$ whenever $(n,\alpha)\leq (m,\beta)$ in $\mathbb{N}\times Mor(\mathscr{X})$.

\smallskip 
(3'') For each $(n, \alpha:y\longrightarrow z)\in \mathbb{N}\times Mor(\mathscr{X})$, the morphism $ ^\alpha\mathcal{P}(n,\alpha): \alpha^\bullet \mathcal{P}(n,\alpha)_z\longrightarrow \mathcal{P}(n,\alpha)_y$ is an isomorphism in $\mathbf M_{[\mathcal C_y,\_\_]}$.

\smallskip
(4'') $|\mathcal{P}(n,\alpha)| \leq \kappa^\mu$.

\smallskip
For each pair $(n,\alpha: y\longrightarrow z)\in \mathbb{N}\times Mor(\mathscr{X})$, we now start constructing  $\mathcal{P}(n,\alpha)$.
We know that there exists a morphism $\eta:\mathcal{C}_x^* \longrightarrow \mathcal{M}_x$   in $\mathbf M_{[\mathcal C_x,\_\_]}$ such that $m_0 \in Im(\eta)$. Then, we can define the subobject $\mathcal{N} \subseteq \mathcal{M}$ in $Cont^{tr}\text{-}\mathcal{C}$ as in \eqref{defN-cont}
such that $m_0 \in \mathcal{N}_x$. We also know by Lemma \ref{cardinal-cont} that $|\mathcal{N}| \leq \kappa\leq \kappa^\mu$.

\smallskip
For $(n, \alpha:y\longrightarrow z)\in \mathbb{N}\times Mor(\mathscr{X})$, we set
\begin{equation}\label{transA00cont}
A_0^0(w) = 
\begin{cases}
\mathcal{N}_w&\quad\text{if $n=1$ and $\alpha=\alpha_0$}\\
\bigcup\limits_{(m,\beta)<(n,\alpha)}\mathcal{P}(m,\beta)_w&\quad\text{otherwise}
\end{cases}
\end{equation}
for each $w\in Ob(\mathscr{X})$, where for $(m, \beta) <(n,\alpha)$, we assume that   $\mathcal{P}(m,\beta)$ satisfies all the properties (1'')-(4''). Clearly, $A_0^0(w)\subseteq \mathcal{M}_w$ and $|A_0^0(w)|\leq \kappa^\mu $ for each $w\in Ob(\mathscr{X})$.

\smallskip
Since $\mathcal{M}\in Cont^{tr}_c\text{-}\mathcal{C}$, we have $\alpha^\bullet \mathcal{M}_z\cong \mathcal{M}_y$. Using Lemma \ref{transdoublecont}, we can obtain a contramodule $A_1^0(z)\subseteq \mathcal{M}_z$ in $\mathbf M_{[\mathcal C_z,\_\_]}$ such that
\begin{equation}\label{A10transcont}
|A_1^0(z)|\leq \kappa^\mu \qquad |\alpha^\bullet A_1^0(z)|\leq \kappa^\mu   \qquad A_0^0(z)\subseteq A_1^0(z)\qquad A_0^0(y)\subseteq \alpha^\bullet A_1^0(z)
\end{equation}
We now set $A_1^0(y):=\alpha^\bullet A_1^0(z)$ and $A_1^0(w):=A_0^0(w)$ for any $w\neq y,z \in Ob(\mathscr{X})$. 
 It now follows from \eqref{A10transcont} that $A_0^0(w)\subseteq A_1^0(w)$ for every $w\in Ob(\mathscr{X})$ and each $|A_1^0(w)|\leq \kappa^\mu$.

 \begin{lem}\label{subcontra}
Let $B\subseteq el_{\mathscr{X}}(\mathcal{M})$ with $|B|\leq \kappa^\mu$. Then, there is a subobject $\mathcal{Q}\hookrightarrow \mathcal{M}$ in $Cont^{tr}$-$\mathcal{C}$ such that $B\subseteq el_{\mathscr{X}}(\mathcal{Q})$ and $|\mathcal{Q}|\leq \kappa^\mu$.
\end{lem}
\begin{proof}
As in the proof of Theorem \ref{Grothcont}, for any $m_0\in B\subseteq el_\mathscr{X}(\mathcal{M})$ we can choose a subobject $\mathcal{Q}(m_0)\subseteq \mathcal{M}$ in $Cont^{tr}$-$\mathcal{C}$ such that $m_0\in el_{\mathscr{X}}(\mathcal{Q}(m_0))$ and $|\mathcal{Q}(m_0)|\leq \kappa \leq \kappa^\mu$. Now, we set $\mathcal{Q}:= \sum_{m_0\in B}\mathcal{Q}(m_0)\subseteq 
\mathcal M$. Since $\mathcal{Q}$ is a quotient of $\oplus_{m_0\in B}\mathcal{Q}(m_0)$ and $|B|\leq \kappa^\mu$, it follows from Lemma \ref{sumcontrcard}(b) that
$|\mathcal Q|\leq |\oplus_{m_0\in B}\mathcal{Q}(m_0)|\leq (\kappa^\mu)^\mu=\kappa^\mu$.
\end{proof}

Using Lemma \ref{subcontra} with $B=\bigcup_{w\in Ob(\mathscr{X})}A_1^0(w)$, we know that there exists  $\mathcal{Q}^0(n,\alpha)\hookrightarrow \mathcal{M}$  in $Cont^{tr}$-$\mathcal{C}$ such that
$\bigcup_{w\in Ob(\mathscr{X})}A_1^0(w)\subseteq el_{\mathscr{X}}(\mathcal{Q}^0(n,\alpha))$ and $|\mathcal{Q}^0(n,\alpha)|\leq \kappa^\mu$. In particular, we have $A_1^0(w)\subseteq \mathcal{Q}^0(n,\alpha)_w \subseteq \mathcal{M}_w$ for each $w\in Ob(\mathscr{X})$.

\smallskip
We will now iterate this construction. Suppose that we have constructed $\mathcal{Q}^l(n,\alpha)\hookrightarrow \mathcal{M}$  in $Cont^{tr}$-$\mathcal{C}$ for every $l\leq r$ such that  $ \bigcup_{w\in Ob(\mathscr{X})}A_1^l(w)\subseteq el_{\mathscr{X}}(\mathcal{Q}^l(n,\alpha))$ and $|\mathcal{Q}^l(n,\alpha)|\leq \kappa^\mu$.
We now set $A_0^{r+1}(w):= \mathcal{Q}^r(n,\alpha)_w$ for each $w\in Ob(\mathscr{X})$.  Since $A_0^{r+1}(y)\subseteq \mathcal{M}_y\cong \alpha^\bullet\mathcal{M}_z$ and $A_0^{r+1}(z)\subseteq \mathcal{M}_z$, using Lemma \ref{transdoublecont}, we can obtain $A_1^{r+1}(z)\hookrightarrow \mathcal{M}_z$ in $\mathbf M_{[\mathcal{C}_z,\_\_]}$ such that
\begin{equation}\label{transm+1cont}
|A_1^{r+1}(z)|\leq \kappa^\mu  \qquad |\alpha^\bullet A_1^{r+1}(z)|\leq \kappa^\mu  \qquad A_0^{r+1}(z)\subseteq A_1^{r+1}(z)\qquad A_0^{r+1}(y)\subseteq \alpha^\bullet A_1^{r+1}(z)
\end{equation}
Then, we set $A_1^{r+1}(y):= \alpha^\bullet A_1^{r+1}(z)$ and $A_1^{r+1}(w):= A_0^{r+1}(w)$ for any $w\neq y,z\in Ob(\mathscr{X})$.  It now follows from  \eqref{transm+1cont} that $A_0^{r+1}(w)\subseteq A_1^{r+1}(w)$ for all $w\in Ob(\mathscr{X})$ and each $|A_1^{r+1}(w)|\leq \kappa^\mu$.  

\smallskip
Again using Lemma \ref{subcontra} with $B=\bigcup_{w\in Ob(\mathscr{X})}A_1^{r+1}(w)$, we can obtain   $\mathcal{Q}^{r+1}(n,\alpha)\hookrightarrow \mathcal{M}$ in $Cont^{tr}$-$\mathcal{C}$ such that
$\bigcup_{w\in Ob(\mathscr{X})}A_1^{r+1}(w)\subseteq el_{\mathscr{X}}(\mathcal{Q}^{r+1}(n,\alpha))$ and $|\mathcal{Q}^{r+1}(n,\alpha)|\leq \kappa^\mu$.  In particular, $A_1^{r+1}(w)\subseteq \mathcal{Q}^{r+1}(n,\alpha)_w$ for each $w\in Ob(\mathscr{X})$. We now define
\begin{equation}\label{transDefinitioncont}
\mathcal{T}(n,\alpha):=\varinjlim_{r\geq 0}\textrm{ }\mathcal{Q}^r(n,\alpha)\qquad \mathcal{P}(n,\alpha):=Im\left(\varinjlim_{r\geq 0}\textrm{ }\mathcal{Q}^r(n,\alpha)\longrightarrow \mathcal M\right)=Im\left(\mathcal T(n,\alpha)\longrightarrow \mathcal M\right)
\end{equation}
in $Cont^{tr}$-$\mathcal{C}$. 

\begin{lem}\label{constcont-1}
The  morphism $ ^\alpha\mathcal{T}(n,\alpha): \alpha^\bullet \mathcal{T}(n,\alpha)_z\longrightarrow \mathcal{T}(n,\alpha)_y$ is an isomorphism in $\mathbf M_{[\mathcal{C}_y,\_\_]}$. Also, 
$|\mathcal{T}(n,\alpha)| \leq \kappa^\mu$.
\end{lem}

\begin{proof}
By definition, $\mathcal T(n,\alpha)_z$ is the following colimit in $\mathbf M_{[\mathcal{C}_z,\_\_]}$.
\begin{equation*}
\mathcal T(n,\alpha)_z=colim\left(A_1^0(z)\longrightarrow \mathcal{Q}^0(n,\alpha)_z\longrightarrow A_1^1(z)\longrightarrow \mathcal{Q}^1(n,\alpha)_z\longrightarrow\hdots \longrightarrow A_1^{r+1}(z)\longrightarrow \mathcal{Q}^{r+1}(n,\alpha)_z\longrightarrow \hdots \right)
\end{equation*} Since $\alpha^\bullet$ is a left adjoint, we have therefore
\begin{equation*}
\alpha^\bullet\mathcal T(n,\alpha)_z=colim\left(\alpha^\bullet A_1^0(z)\longrightarrow \alpha^\bullet\mathcal{Q}^0(n,\alpha)_z\longrightarrow  \hdots \longrightarrow \alpha^\bullet A_1^{r+1}(z)\longrightarrow \alpha^\bullet \mathcal{Q}^{r+1}(n,\alpha)_z\longrightarrow \hdots\right)
\end{equation*} in $\mathbf M_{[\mathcal{C}_y,\_\_]}$. We also know that $\alpha^\bullet A_1^r(z)= A_1^r(y)$ for each $r\geq 0$. Writing $\mathcal T(n,\alpha)_y$ as the colimit
of $\{A_1^r(y)\}_{r\geq 0}$  in $\mathbf M_{[\mathcal{C}_y,\_\_]}$, we see that   $^\alpha\mathcal T(n,\alpha)$ is an isomorphism. Finally, we note that
each $|\mathcal Q^r(n,\alpha)|\leq \kappa^\mu$. Applying Lemma \ref{sumcontrcard}(b), we now obtain $|\mathcal{T}(n,\alpha)|\leq (\kappa^\mu)^\mu \leq \kappa^\mu$.
\end{proof}

\begin{lem}\label{constcont}
The family $\{\mathcal{P}(n,\alpha) | (n, \alpha) \in \mathbb{N} \times Mor(\mathscr{X})\}$ satisfies  conditions (1'')--(4'').
\end{lem}
\begin{proof}
We know that $m_0 \in \mathcal{N}_x$. For $n=1$ and $\alpha=\alpha_0$, we have $\mathcal{N}_x= A_0^0(x)\subseteq A_1^0(x)\subseteq  \mathcal{Q}^0(1,\alpha_0)_x$. Since $\mathcal Q^0(1,\alpha_0)\hookrightarrow\mathcal M$ factors through $\mathcal T(1,\alpha_0)$, it follows that $m_0\in  \mathcal{Q}^0(1,\alpha_0)_x\subseteq Im(\mathcal T(1,\alpha_0)
\longrightarrow \mathcal M)_x=\mathcal P(1,\alpha_0)_x$.   Thus, condition {(1'')} is satisfied. 

\smallskip 
Now, for any $(m,\beta)<(n,\alpha)$, it again follows by \eqref{transA00cont} and \eqref{transDefinitioncont} that we have $\mathcal{P}(m,\beta)_w\subseteq A_0^0(w)\subseteq A_1^0(w)\subseteq \mathcal{Q}^0(n,\alpha)_w\subseteq Im(\mathcal T(n,\alpha)\longrightarrow\mathcal M)_w=\mathcal{P}(n,\alpha)_w$ for every $w\in Ob(\mathscr{X})$. This shows that condition {(2'')} is satisfied. The condition {{(4'')}} follows from Lemma \ref{constcont-1}
 and the fact that $\mathcal P(n,\alpha)$ is a quotient of $\mathcal T(n,\alpha)$. 

\smallskip
 For $(n, \alpha:y\longrightarrow z)\in \mathbb{N}\times Mor(\mathscr{X})$, we know from Lemma \ref{constcont-1} that   $ ^\alpha\mathcal{T}(n,\alpha): \alpha^\bullet \mathcal{T}(n,\alpha)_z\longrightarrow \mathcal{T}(n,\alpha)_y$ is an isomorphism in $\mathbf M_{[\mathcal{C}_y,\_\_]}$. Since $\mathcal M$ is cartesian, we also know that 
 $^\alpha\mathcal M:\alpha^\bullet\mathcal M_z\longrightarrow \mathcal M_y$ is an isomorphism. Since $\alpha^\bullet$ is exact and $\mathcal P(n,\alpha)=Im(\mathcal T(n,\alpha) \longrightarrow \mathcal M)$, it follows that $\alpha^\bullet\mathcal P(n,\alpha)_z\cong \mathcal P(n,\alpha)_y$.  This proves condition {(3'')}.
\end{proof}

\begin{lem}\label{gencont}
Let $\mathscr{X}$ be a poset and let $\mathcal{C}:\mathscr{X} \longrightarrow Coalg$ be a right coflat representation. Let $\mathcal{M} \in Cont^{tr}_c\text{-}\mathcal{C}$ and $m_0 \in el_{\mathscr{X}}(\mathcal{M})$. 
Then, there exists   $\mathcal{P} \subseteq \mathcal{M}$ in $ Cont^{tr}_c\text{-}\mathcal{C}$ with $m_0 \in el_{\mathscr{X}}(\mathcal{P})$ such that $|\mathcal{P}| \leq \kappa^\mu$.
\end{lem}
\begin{proof}
We note that the set $\mathbb{N}\times Mor(\mathscr{X})$ with the lexicographic ordering is filtered. We  set
\begin{equation}\label{6.5tre}
\mathcal T:=\underset{(n,\alpha)\in \mathbb{N}\times Mor(\mathscr{X})}{\varinjlim}\mathcal P(n,\alpha)\qquad \mathcal{P}:=Im(\mathcal T\longrightarrow \mathcal M)
\end{equation} in $Cont^{tr}$-$\mathcal{C}$. 
By Lemma \ref{constcont}, we know that $m_0\in \mathcal{P}(1,\alpha_0)_x\subseteq Im(\mathcal T_x\longrightarrow \mathcal M_x)=\mathcal P_x$. Thus, $m_0\in el_{\mathscr{X}}(\mathcal{P})$. By Lemma \ref{constcont}, we also know that each $|\mathcal P(n,\alpha)|\leq \kappa^\mu$. Applying Lemma \ref{sumcontrcard}(b) to the colimit in \eqref{6.5tre}, we see that
$|\mathcal P|\leq |\mathcal T|\leq (\kappa^\mu)^\mu=\kappa^\mu$. 

\smallskip
It remains to show that $\mathcal P$ is cartesian. For this, we consider a morphism $\beta: z\longrightarrow w$ in $\mathscr{X}$. Then, the family $\{(s,\beta)\}_{s\geq 1}$ is cofinal in $\mathbb{N}\times Mor(\mathscr{X})$ and therefore we may express
\begin{equation}
\mathcal T:=\underset{s\geq 1}{\varinjlim}\textrm{ }\mathcal P(s,\beta)
\end{equation}
Since $\beta^\bullet$ is a left adjoint, we have $\beta^\bullet\mathcal T_w=\beta^\bullet\left(\underset{s\geq 1}{\varinjlim}\textrm{ }\mathcal P(s,\beta)_w\right) = \underset{s\geq 1}{\varinjlim}\textrm{ }\beta^\bullet\mathcal P(s,\beta)_w$. By Lemma \ref{constcont}, each 
$^\beta\mathcal{P}(s,\beta): \beta^\bullet\mathcal{P}(s,\beta)_w\longrightarrow \mathcal{P}(s,\beta)_z$ is an isomorphism and it follows that $^\beta \mathcal T:\beta^\bullet \mathcal T_w\longrightarrow \mathcal T_z$ is an isomorphism. In other words, $\mathcal T\in Cont^{tr}_c\text{-}\mathcal{C}$. It follows that $\mathcal P=Im(\mathcal T\longrightarrow \mathcal M)$
also lies in $Cont^{tr}_c\text{-}\mathcal{C}$.
\end{proof}

\begin{Thm}\label{T6.11exes}
Let $\mathscr{X}$ be a poset and let $\mathcal{C}:\mathscr{X} \longrightarrow Coalg$ be a right coflat representation.
 Then, the category $Cont^{tr}_c\text{-}\mathcal{C}$ of cartesian trans-contramodules over $\mathcal{C}$ has a set of generators.
\end{Thm}
\begin{proof}
It follows from Lemma \ref{gencont} that any $\mathcal{M}\in Cont^{tr}_c\text{-}\mathcal{C}$ can be expressed as a quotient of a direct sum $\bigoplus_{m_0 \in el_\mathscr{X}(\mathcal{M})}\mathcal{P}_{m_0}$ of cartesian subcontramodules with each $|\mathcal{P}_{m_0}| \leq \kappa^\mu$. Therefore, the isomorphism classes of cartesian contramodules $\mathcal{P}$ with $|\mathcal{P}| \leq \kappa^\mu$ form a set  of generators for $Cont^{tr}_c\text{-}\mathcal{C}$.
\end{proof}

\section{Rational pairings, torsion classes, functors between comodules and contramodules}

In this final section, we relate modules over algebra representations to comodules and contramodules over coalgebra representations. An algebra representation will be a  functor
$\mathcal{A}: \mathscr{X}\longrightarrow Alg$, where $Alg$ is the category of $K$-algebras. In \cite{EV}, Estrada and Virili considered modules over a representation of a small category $\mathscr X$ taking values in small preadditive categories, i.e., in algebras with several objects. Our ``cis-modules'' over $\mathcal A$ will be identical to the modules of Estrada and Virili \cite{EV} in the   case where the representation in \cite{EV} takes values in $K$-algebras. We will observe that cis-modules over $\mathcal A$ are related to trans-comodules over its finite dual representation and vice versa. 

\subsection{Modules over algebra representations}

\smallskip
For any algebra $A$, let $\mathbf M_A$ denote the category of right  $A$-modules. 
Corresponding to any algebra morphism $\alpha:A \longrightarrow B$, there is a restriction of scalars $\alpha_\circ: \mathbf M_B\longrightarrow \mathbf M_A$ and an extension of scalars
$\alpha^\circ:\mathbf M_A\longrightarrow \mathbf M_B$. The functor $\alpha_\circ$ also has a right adjoint
		\begin{align*}
& \alpha^\dagger:\mathbf M_A \longrightarrow\mathbf M_B \qquad \qquad  N \mapsto Hom_A(B,N)
\end{align*}
Let $\mathcal A:\mathscr X\longrightarrow Alg$ be an algebra representation. In particular, for each object $x\in Ob(\mathscr{X})$, we have an algebra $\mathcal{A}_x$ and for any   $\alpha \in \mathscr{X}(x,y)$, we have a morphism $\mathcal{A}_{\alpha}: \mathcal{A}_x\longrightarrow \mathcal{A}_y$ of $K$-algebras. As with coalgebra representations, we will abuse notation to write $\alpha^\circ$, $\alpha_\circ$ and $\alpha^\dagger$ respectively for functors $\mathcal A_\alpha^\circ$, $\mathcal A_{\alpha\circ}$, $\mathcal A_\alpha^\dagger$ for any morphism $\alpha$ in $\mathscr X$.

\begin{defn}
Let $\mathcal{A}: \mathscr{X}\longrightarrow Alg$ be an algebra representation. A (right) cis-module $\mathcal{M}$ over $\mathcal{A}$ will consist of the following data:
  \begin{itemize}
  \item[(1)] For each object $x\in Ob(\mathscr{X})$, a right $\mathcal{A}_x$-module $\mathcal{M}_x$
  \item[(2)] For each morphism
    $\alpha: x \longrightarrow y$ in $\mathscr{X}$, a morphism
    ${\mathcal{M}}_{\alpha}: \mathcal{M}_x\longrightarrow \alpha_\circ\mathcal{M}_y$ of right $\mathcal{A}_x$-modules (equivalently, a morphism $\mathcal{M}^{\alpha}: \alpha^\circ\mathcal{M}_x\longrightarrow \mathcal{M}_y$ of right $\mathcal{A}_y$-modules).
  \end{itemize}

  We further assume that $\mathcal{M}_{id_x} = id_{\mathcal{M}_x}$ for any $x\in Ob(\mathscr X)$ and $\mathcal{M}^{\beta\alpha}=\mathcal{M}^{\beta}\circ \beta^\circ (\mathcal{M}^\alpha)$ for composable morphisms $\alpha$, $\beta$ in $\mathscr X$ (equivalently, $\alpha_\circ(\mathcal{M}_{\beta})\circ \mathcal{M}_{\alpha}= \mathcal{M}_{\beta\alpha}$). 
 A morphism $\eta: \mathcal{M} \longrightarrow \mathcal{N}$ of cis-modules over $\mathcal{A}$ consists of morphisms $\eta_x: \mathcal{M}_x\longrightarrow\mathcal{N}_x$ of right $\mathcal{A}_x$-modules for each $x\in Ob(\mathscr{X})$ compatible with the morphisms in (2). 
We denote this category of right cis-modules by $Mod^{cs}$-$\mathcal{A}$. Similarly, we may define the category $\mathcal A$-$Mod^{cs}$ of left cis-modules over $\mathcal A$. 

\smallskip
We will say that $\mathcal M\in Mod^{cs}$-$\mathcal{A}$ is cartesian if each $\mathcal M^\alpha$ is an isomorphism. The full subcategory of cartesian cis-modules will be denoted by $Mod^{cs}_c$-$\mathcal{A}$. 
\end{defn}

 The following results now follow from \cite[Theorem 3.18, Theorem 3.23]{EV}.
\begin{Thm}
  Let $\mathcal{A}: \mathscr{X}\longrightarrow Alg $ be an algebra representation. Then,
\begin{itemize}
\item[(1)] $Mod^{cs}$-$\mathcal{A}$ is a Grothendieck category. If  $\mathscr{X}$ is a poset, then $Mod^{cs}$-$\mathcal{A}$ has a projective generator.
\item[(2)] Suppose that $\mathcal{A}: \mathscr{X}\longrightarrow Alg $ is left  flat, i.e., each $\alpha^\circ=\_\_\otimes_{\mathcal A_x}\mathcal A_y:\mathbf M_{\mathcal A_x}\longrightarrow \mathbf M_{\mathcal A_y}$ is exact for $\alpha\in \mathscr X(x,y)$.  Then,  $Mod^{cs}_c$-$\mathcal{A}$ is a Grothendieck category.
  \end{itemize}
\end{Thm}

We now introduce the notion of trans-module over an algebra representation.
\begin{defn}
Let $\mathcal{A}: \mathscr{X}\longrightarrow Alg$ be an algebra representation. A (right) trans-module $\mathcal{M}$ over $\mathcal{A}$ will consist of the following data:
  \begin{itemize}
  \item[(1)] For each object $x\in Ob(\mathscr{X})$, a right $\mathcal{A}_x$-module $\mathcal{M}_x$
  \item[(2)] For each morphism
    $\alpha: x \longrightarrow y$ in $\mathscr{X}$, a morphism
    $_\alpha{\mathcal{M}}: \mathcal{M}_y\longrightarrow \alpha^\dagger\mathcal{M}_x=Hom_{\mathcal{A}_x}(\mathcal{A}_y, \mathcal{M}_x)$ of right $\mathcal{A}_y$-modules (equivalently, a morphism $^{\alpha}\mathcal{M}: \alpha_\circ\mathcal{M}_y\longrightarrow \mathcal{M}_x$ of right $\mathcal{A}_x$-modules) 
  \end{itemize}

  We further assume that $_{id_x}\mathcal{M} = id_{\mathcal{M}_x}$ and $_{\beta\alpha}\mathcal M=\beta^\dagger({_\alpha}\mathcal M)\circ {_\beta}\mathcal M$ for   composable morphisms $\alpha$, $\beta$ in $\mathscr X$ (equivalently, $^{\beta\alpha}\mathcal M={^\alpha}\mathcal M\circ \alpha_\circ({_\beta}\mathcal M)$).  A morphism $\eta: \mathcal{M} \longrightarrow \mathcal{N}$ of trans-modules over $\mathcal{A}$ consists of morphisms $\eta_x: \mathcal{M}_x\longrightarrow\mathcal{N}_x$ of right $\mathcal{A}_x$-modules for each $x\in Ob(\mathscr{X})$ compatible with the morphisms in (2). 
We denote this category of right trans-modules by $Mod^{tr}$-$\mathcal{A}$. Similarly, we may define the category $\mathcal A$-$Mod^{tr}$ of left trans-modules over $\mathcal A$. 
 
 \smallskip
We will say that $\mathcal M\in Mod^{tr}$-$\mathcal{A}$ is cartesian if each ${_\alpha}\mathcal M$ is an isomorphism. The full subcategory of cartesian trans-modules will be denoted by $Mod^{tr}_c$-$\mathcal{A}$. 
\end{defn}

We may easily verify that $Mod^{tr}$-$\mathcal{A}$ is a cocomplete abelian category. To study 
generators in $Mod^{tr}$-$\mathcal{A}$, we consider $\mathcal{M} \in Mod^{tr}\text{-}\mathcal{A}$ and $m_0\in \mathcal M_x$ for some $x \in Ob(\mathscr{X})$. If we take a morphism 
$
\eta:\mathcal{A}_x \longrightarrow \mathcal{M}_x
$
in $Mod\text{-}\mathcal{A}_x$ such that  $m_0\in Im(\eta)$, we can set for each $y \in Ob(\mathscr{X})$ 
\begin{equation}\label{modtrx}
\mathcal{N}_y:=Im\left(\bigoplus\limits_{\beta \in \mathscr{X}(y,x)} \beta_\circ\mathcal{A}_x \xrightarrow{\beta_\circ\eta} \beta_\circ\mathcal{M}_x   \xrightarrow{^\beta{\mathcal{M}}} \mathcal{M}_y   \right)
\end{equation} As in previous sections, we can show that the  family $\{\mathcal{N}_y \in \mathbf M_{\mathcal A_y}\}_{y \in Ob(\mathscr{X})}$ determines an object $\mathcal N$ in $Mod^{tr}\text{-}\mathcal{A}$. Moreover, $m_0\in \mathcal N_x$.  Now, we set
\begin{equation*}
\kappa:= sup\{\aleph_0, |Mor(\mathscr X)|, |K|, |\mathcal{A}_x|, {x \in Ob(\mathscr{X})}  \}
\end{equation*} Then, if $|\mathcal N|$ denotes the cardinality of the union $\bigcup_{y\in Ob(\mathscr X)}\mathcal N_y$, we have $|\mathcal N|\leq \kappa$.  This proves the following result.

\begin{Thm}\label{Grothmodg}
Let $\mathcal{A}:\mathscr{X} \longrightarrow Alg$ be an algebra representation. Then, $Mod^{tr}\text{-}\mathcal{A}$ is a Grothendieck category.
\end{Thm}

As in Section 3.2 and Section 5, we may show that when $\mathscr X$ is a poset,  the evaluation functor $ev_x^{tr}:Mod^{tr}\text{-}\mathcal{A}\longrightarrow \mathbf M_{\mathcal A_x}$ for each $x\in Ob(\mathscr X)$ has both a left adjoint $ex_x^{tr}:\mathbf M_{\mathcal A_x}\longrightarrow Mod^{tr}\text{-}\mathcal{A}$ and a right adjoint $coe^{tr}_x:\mathbf M_{\mathcal A_x}\longrightarrow Mod^{tr}\text{-}\mathcal{A}$. Since each $\mathcal A_x$ is a projective generator in $\mathbf M_{\mathcal A_x}$, we may prove the following result in a manner similar to Theorem \ref{projgen3.8}.

\begin{Thm}\label{modtrproj}
Let $\mathscr{X}$ be a poset and $\mathcal{A}: \mathscr{X}\longrightarrow Alg$ be an algebra representation. Then,  $\{ex^{tr}_x(\mathcal A_x)\}_{x\in Ob(\mathscr X)}$ is a set of projective generators for  $Mod^{tr}\text{-}\mathcal{A}$.
\end{Thm}

Let $\mathscr{X}$ be a poset and $\mathcal{A}:\mathscr{X} \longrightarrow Alg$ be an algebra representation such that for each $\alpha \in \mathscr{X}(x,y)$, the $K$-algebra $\mathcal{A}_y$ is finitely generated and projective as an $\mathcal{A}_x$-module. Then, each $\alpha^\dagger$ is exact and commutes with colimits. Accordingly, we may verify that the category $Mod^{tr}_c\text{-}\mathcal{A}$ of cartesian trans-modules over $\mathcal A$ is abelian and cocomplete. 

\smallskip
We take $\mathcal M\in Mod^{tr}_c\text{-}\mathcal{A}$ and consider  $m_0\in \mathcal M_x$ for some $x\in Ob(\mathscr X)$. Using a transfinite induction argument similar to Section 4.2, we can obtain a subobject $\mathcal P\hookrightarrow \mathcal M$ in $Mod^{tr}_c\text{-}\mathcal{A}$ such that $m_0\in 
\mathcal P_x$. Accordingly, we can prove the following result.

\begin{Thm}\label{T6.11xs}
Let $\mathcal{A}:\mathscr{X} \longrightarrow Alg$ be an algebra representation such that for each $\alpha \in \mathscr{X}(x,y)$, the $K$-algebra $\mathcal{A}_y$ is finitely generated and projective as an $\mathcal{A}_x$-module. Then, $Mod^{tr}_c$-$\mathcal{A}$ is a Grothendieck category.
\end{Thm}

\subsection{Rational modules, torsion classes, functors between comodules and contramodules}

Let $C$ be a $K$-coalgebra and let $A$ be a $K$-algebra. We recall (see, for instance, \cite[$\S$ 4.18]{BW}) that a rational pairing of $C$ and $A$ consists of a morphism $\varphi: C\otimes A\longrightarrow K$ such that:

\smallskip
(i) For any vector space $V$, the following linear map is injective
\begin{equation*} \delta_V:V \otimes C \longrightarrow Hom_K(A,V)\qquad  v \otimes c \longmapsto [a \mapsto \varphi(c\otimes a)v]
\end{equation*} 

\smallskip
(ii) The morphism $A\longrightarrow C^*$ induced by $\varphi:C\otimes A\longrightarrow K$ is a morphism of $K$-algebras. 

\smallskip
We note in particular that the induced map $C\longrightarrow A^*$ is injective. For more on rational pairings, we refer the reader, for instance, to 
\cite{AGL}. 

\smallskip
Now let  $\varphi: C\otimes A\longrightarrow K$ be a rational pairing, let $\mathbf M^C$ denote the category of right $C$-comodules and $_A\mathbf M$ denote the category of left $A$-modules. We note that any right $C$-comodule with structure map $\rho^M:M\longrightarrow M\otimes C$  may be treated as a left $A$-module by setting
$am:=\sum m_0\varphi(m_1\otimes a)$ for each $a\in A$ and $m\in A$, where $\rho^M(m)=\sum m_0\otimes m_1\in M\otimes C$. In fact, this embeds
$\mathbf M^C$ as the smallest full Grothendieck subcategory of $_A\mathbf M$ containing $C$ (see \cite[$\S$ 4.19]{BW}). Accordingly, the inclusion
$I_\varphi: \mathbf M^C\longrightarrow {_A}\mathbf M$ has a right adjoint $R_\varphi: {_A}\mathbf M\longrightarrow \mathbf M^C$ given by setting (see 
\cite[$\S$ 41.1]{BW})
\begin{equation}
R_\varphi(N):=\sum \{\mbox{$Im(f)$ $\vert$  $f\in {_A}\mathbf M(I_\varphi(M), N)$, $M\in \mathbf M^C$}\}
\end{equation}  A left $A$-module $N$ is said to be rational if it satisfies $R_\varphi(N)=N$. Since the collection of rational modules is closed under quotients and subobjects (see \cite[Corollary 2.4]{Rad}), it follows that for any $N\in {_A}\mathbf M$, $R_\varphi(N)$ is the sum of all rational submodules of $N$.  

\begin{defn}\label{funcratpr}
	Let $\mathcal{C}:\mathscr{X} \longrightarrow Coalg$ be a coalgebra representation and let $\mathcal{A}:\mathscr{X}^{op} \longrightarrow Alg$ be an algebra representation.  Then, a rational pairing $(\mathcal{C},\mathcal{A},\Phi=\{\varphi_x\}_{x\in Ob(\mathscr X)})$  is a triple such that
	
	\smallskip
	(1) for each $x \in Ob(\mathscr{X})$, $\varphi_x:\mathcal C_x\otimes\mathcal A_x\longrightarrow K$ is a rational pairing of a coalgebra with an algebra
	
	\smallskip
	(2) for each $\alpha \in \mathscr{X}(x,y)$, we have $\varphi_y(\mathcal C_\alpha(c)\otimes a)=\varphi_x(c\otimes\mathcal A_\alpha(a))$ for any
	$c\in \mathcal C_x$ and $a\in \mathcal A_y$. 
\end{defn}

\begin{defn}
		Let $(\mathcal{C},\mathcal{A},\Phi)$ be a rational pairing of  $\mathcal{C}:\mathscr{X} \longrightarrow Coalg$ with $\mathcal{A}:\mathscr{X}^{op} \longrightarrow Alg$.
	A (left) trans-module $\mathcal{M}$ over the algebra representation $\mathcal{A}$ is said to be rational if $R_{\varphi_x}(\mathcal{M}_x)=
	\mathcal M_x$  for each $x \in Ob(\mathscr{X})$.
	
	\smallskip
	Similarly, we can define rational cis-modules.
We will denote by $\mathcal{A}\text{-}Rat^{tr}$ (resp.  $\mathcal{A}\text{-}Rat^{cs}$) the full subcategory of $\mathcal{A}\text{-}Mod^{tr}$ (resp.  $\mathcal{A}\text{-}Mod^{cs}$) whose objects are rational trans-modules (resp. rational cis-modules) over $\mathcal{A}$. 
\end{defn}

\begin{Thm}
Let $(\mathcal{C},\mathcal{A},\Phi)$ be a rational pairing. Then, 

\smallskip	
(1)	the inclusion functor $I_\Phi^{tr}:\mathcal{A}\text{-}Rat^{tr} \longrightarrow \mathcal{A}\text{-}Mod^{tr}$ admits a right adjoint $R_\Phi^{tr}:\mathcal{A}\text{-}Mod^{tr} \longrightarrow \mathcal{A}\text{-}Rat^{tr}$.

\smallskip
(2) $\mathcal{N} \in \mathcal{A}\text{-}Mod^{tr}$ is rational if and only if $R_\Phi^{tr}(\mathcal{N})=\mathcal{N}$.
	\end{Thm}
	\begin{proof}
(1) Let $\mathcal{N} \in \mathcal{A}\text{-}Mod^{tr}$. For each $y\in Ob(\mathscr X)$, we define
	\begin{equation}\label{7.3gt}
	R_\Phi^{tr}(\mathcal{N})_y:=\{n_y \in \mathcal{N}_y~ |~ {^\alpha}{\mathcal{N}}(n_y)  \in R_{\varphi_x}(\mathcal{N}_x)~ \forall~ \alpha \in \mathscr{X}^{op}(x,y)\}
	\end{equation}
	In particular, we note that $R_\Phi^{tr}(\mathcal N)_y\subseteq R_{\varphi_y}(\mathcal N_y)$. Since the class of rational modules in ${_{\mathcal A_y}}\mathbf M$ is closed under subobjects, it follows that each $R_\Phi^{tr}(\mathcal{N})_y$ is a rational $\mathcal A_y$-module.  From \eqref{7.3gt}, it is also clear that for any morphism $\beta\in \mathscr X^{op}(y,z)$, the morphism $^\beta\mathcal N:\beta_\circ\mathcal N_z\longrightarrow \mathcal N_y$
	restricts to a morphism $^\beta R_\Phi^{tr}(\mathcal N):\beta_\circ R_\Phi^{tr}(\mathcal N)_z\longrightarrow R_\Phi^{tr}(\mathcal N)_y$. 
It follows that $R_\Phi^{tr}(\mathcal{N}) \in \mathcal{A}\text{-}Rat^{tr}$. 

\smallskip
We consider now some $\mathcal M\in  \mathcal{A}\text{-}Rat^{tr}$ and a morphism $\eta: I_\Phi(\mathcal M)\longrightarrow \mathcal N$ in $ \mathcal{A}\text{-}Mod^{tr}$. We claim that $\eta_y(\mathcal M_y)\subseteq R_\Phi^{tr}(\mathcal N)_y$ for each $y\in Ob(\mathscr X)$. Accordingly, for each $\alpha \in \mathscr{X}^{op}(x,y)$, we consider the following commutative diagram in $_{\mathcal A_x}\mathbf M$:
\begin{equation}\label{7.4fgt}
\begin{CD}
\alpha_\circ\mathcal M_y @>\alpha_\circ\eta_y>>\alpha_\circ\mathcal N_y\\
@V^\alpha\mathcal MVV @VV^\alpha\mathcal NV \\ 
\mathcal M_x @>\eta_x>> \mathcal N_x\\
\end{CD}
\end{equation} We consider $m_y\in\mathcal M_y$. By \eqref{7.4fgt}, we have $(^\alpha\mathcal N)(\alpha_\circ\eta_y(m_y))=\eta_x(^\alpha\mathcal M(m_y))$. Since $\mathcal M_x$ is a rational $\mathcal A_x$-module, we know that $Im(\eta_x)\subseteq R_{\varphi_x}(\mathcal N_x)$. It follows that 
$(^\alpha\mathcal N)(\alpha_\circ\eta_y(m_y))\in R_{\varphi_x}(\mathcal N_x)$ for each $\alpha \in \mathscr{X}^{op}(x,y)$, i.e., $\eta_y(m_y)
\in R_\Phi^{tr}(\mathcal N)_y$. The result is now clear.

\smallskip
(2) Let $\mathcal{N} \in \mathcal{A}\text{-}Mod^{tr}$ be rational. Then, by definition, each $\mathcal{N}_x=R_{\varphi_x}(\mathcal N_x)$ for $x\in Ob(\mathscr X)$. From \eqref{7.3gt}, it is immediate that $R_\Phi^{tr}(\mathcal N)=\mathcal N$. Conversely, suppose that $\mathcal{N}=R_\Phi^{tr}(\mathcal{N})$. Then, each $\mathcal{N}_x=R_\Phi^{tr}(\mathcal{N})_x \subseteq R_{\varphi_x}(\mathcal{N}_x)$ is rational.
	\end{proof}

\begin{Thm}\label{Theorem7.10yc}
Let  $(\mathcal{C},\mathcal{A},\Phi)$ be a rational pairing. Then, the categories $Com^{cs}\text{-}\mathcal{C}$ and $\mathcal{A}\text{-}Rat^{tr}$ are isomorphic.
\end{Thm}
\begin{proof}
	Let $\mathcal{M} \in Com^{cs}\text{-}\mathcal{C}$. Then, for each $x\in Ob(\mathscr X)$, $\mathcal M_x$ carries the structure of a rational $\mathcal A_x$-module. Further, given any $\alpha\in \mathscr X(x,y)=\mathscr X^{op}(y,x)$, the morphism $\alpha^*\mathcal M_x\longrightarrow \mathcal M_y$ of right $\mathcal C_y$-comodules induces a morphism $\alpha_\circ\mathcal M_x\longrightarrow \mathcal M_y$ of left $\mathcal A_y$-modules. As such, $\mathcal M$ may be treated as an object of $\mathcal{A}\text{-}Rat^{tr}$.

	\smallskip 
	Conversely, suppose we have $\mathcal{N} \in \mathcal{A}\text{-}Rat^{tr}$. Then, for each $x\in Ob(\mathscr X)$, the rational $\mathcal A_x$-module $\mathcal N_x$ can be equipped with the structure of a right $\mathcal C_x$-comodule. For $\alpha\in \mathscr X(x,y)=\mathscr X^{op}(y,x)$, we also have the morphism $^\alpha\mathcal N:\alpha_\circ\mathcal N_x\longrightarrow \mathcal N_y$ of left $\mathcal A_y$-modules. From condition (2) in Definition 
	\ref{funcratpr}, it follows that the left $\mathcal A_y$-module structure on $\alpha_\circ\mathcal N_x$ can be obtained from a right $\mathcal C_y$-comodule structure, i.e., $\alpha_\circ\mathcal N_x$ is a rational $\mathcal A_y$-module. Since  $\mathbf M^{\mathcal C_y}$ is a full subcategory of ${_{\mathcal A_y}}\mathbf M$, we see that $^\alpha\mathcal N:\alpha_\circ\mathcal N_x\longrightarrow \mathcal N_y$ may now be treated as a morphism
	of right $\mathcal C_y$-comodules. Hence, $\mathcal N$ may be treated as an object of $Com^{cs}\text{-}\mathcal{C}$.
\end{proof}

Similarly, given a rational pairing  $(\mathcal{C},\mathcal{A},\Phi)$, we can show that the inclusion $I_\Phi^{cs}: \mathcal{A}\text{-}Rat^{cs}\hookrightarrow  \mathcal{A}\text{-}Mod^{cs}$ admits a right adjoint $R_\Phi^{cs}$ and that the categories $Com^{tr}\text{-}\mathcal{C}$ and $ \mathcal{A}\text{-}Rat^{cs}$ are isomorphic.

\smallskip
We know in particular that if $C$ is any $K$-coalgebra, then $(C,C^*)$ is a rational pairing. Similarly, for any $K$-algebra $A$, it is well known that we have a rational pairing
$(A^\circ,A)$, where $A^\circ$ is the finite dual coalgebra of $A$ given by setting (see, for instance, \cite[Lemma 1.5.2]{DNR})
\begin{equation}
A^\circ=\{f \in A^*=Hom_K(A,K) ~ |~  Ker( f)~ \text{contains an ideal of finite codimension}  \}
\end{equation} Accordingly, we have the following result.

\begin{cor}\label{C7.11bq} (a) Let $\mathcal{C}:\mathscr{X} \longrightarrow Coalg$ be a coalgebra representation and $\mathcal C^*:\mathscr X^{op}\longrightarrow Alg$ be its linear dual representation. Then, $Com^{cs}\text{-}\mathcal{C} \simeq \mathcal{C}^*\text{-}Rat^{tr}$ and $Com^{tr}\text{-}\mathcal{C} \simeq \mathcal{C}^*\text{-}Rat^{cs}$.

\smallskip
(b) Let $\mathcal{A}:\mathscr{X}^{op} \longrightarrow Alg$ be an algebra representation and $\mathcal A^\circ:\mathscr X\longrightarrow Coalg$ be its finite dual representation. Then,  $Com^{cs}\text{-}\mathcal{A}^\circ \simeq \mathcal{A}\text{-}Rat^{tr}$ and $Com^{tr}\text{-}\mathcal{A}^\circ \simeq \mathcal{A}\text{-}Rat^{cs}$.

\end{cor}

\begin{proof}
Both (a) and (b) follow directly from Theorem \ref{Theorem7.10yc}.
\end{proof}

Our next objective is to give sufficient conditions for $Com^{cs}\text{-}\mathcal{C}\simeq \mathcal{A}\text{-}Rat^{tr}$ to be a torsion class in $\mathcal{A}\text{-}Mod^{tr}$. For this, we 
begin by extending some of the classical theory of rational pairings of coalgebras and algebras using torsion theory (see \cite{Lin}, \cite{NT}). Let $(C,A,\varphi)$ be a rational pairing. We note (see \cite[Proposition 2.2]{Rad}) that for any left $A$-module $N$, the rational submodule $R_\varphi(N)$ may also be expressed as
\begin{equation}\label{7.6eqtu}
R_\varphi(N)=\{\mbox{$n\in N$ $\vert$ $Ann(n)$ is a closed and cofinite left ideal in $A$}\}
\end{equation} 
Here $C^*=Hom_K(C,K)$ carries the finite topology (see, for instance \cite[Section 1.2]{DNR}) and $A$ is equipped with the topology that makes the  algebra morphism $A \longrightarrow C^*$ continuous. We also recall here that a subspace $V\subseteq A$ is said to be cofinite if the quotient $A/V$ is finite dimensional as a vector space. The following argument appears in essence in \cite[Proposition 22]{Lin}, which we extend here to rational pairings of coalgebras and algebras. 

\begin{lem}\label{P7.1ws} Let $\varphi: C\otimes A\longrightarrow K$ be a rational pairing. Suppose that $R_\varphi(A)$ is dense in $A$. Then, for any $N\in {_A}\mathbf M$, we have $R_\varphi(N/R_\varphi(N))=0$.
\end{lem}

\begin{proof} 
We choose $n + R_\varphi(N) \in R_\varphi(N/R_\varphi(N))$ and $a \in R_\varphi(A)$. Then, $Ann(n+R_\varphi(N) )$ and $Ann(a)$ are cofinite and closed left ideals in $A$.  Since $Ann(a)\subseteq Ann(an)$,  it follows from \cite[$\S$ 1.3 (b)]{Rad}  that $Ann(an)$ is also cofinite and closed. Hence, $an\in R_\varphi(N)$ and therefore $a \in Ann(n+ R_\varphi(N))$. This shows that $R_\varphi(A) \subseteq Ann(n+ R_\varphi(N))$. Since $R_\varphi(A)$ is dense in $A$ and $Ann(n+R_\varphi(N) )$ is closed, we get $Ann(n+R_\varphi(N) )=A$, i.e., $n + R_\varphi(N) =0$.
\end{proof}

\begin{thm}\label{P7.2wd} (see \cite{Lin}) Let $\varphi: C\otimes A\longrightarrow K$ be a rational pairing. Then, the following are equivalent.

\smallskip
(a) For any $N\in {_A}\mathbf M$, $R_\varphi(N/R_\varphi(N))=0$. 

\smallskip
(b) The full subcategory  of rational modules is a torsion class in ${_A}\mathbf M$.

\end{thm}

\begin{proof} It is immediate that (b) $\Rightarrow$ (a). 
To show that (a) $\Rightarrow$ (b) we consider the two full subcategories $\mathcal T$ and $\mathcal F$ of $_A{\mathbf M}$ given by 
\begin{equation}
Ob(\mathcal T):=\{\mbox{$N\in {_A}\mathbf M$ $\vert$ $R_\varphi(N)=N$ }\}\qquad Ob(\mathcal F):=\{\mbox{$N\in {_A}\mathbf M$ $\vert$ $R_\varphi(N)=0$ }\}
\end{equation}
From the adjoint pair $(I_\varphi,R_\varphi)$, it is clear that $_A\mathbf M(N_1,N_2)=0$ for any $N_1\in \mathcal T$ and $N_2\in \mathcal F$. Finally, for any $N\in {_A}\mathbf M$, we have a short exact sequence
\begin{equation}
0\longrightarrow R_\varphi(N)\longrightarrow N\longrightarrow N/R_\varphi(N)\longrightarrow 0
\end{equation} From \eqref{7.6eqtu}, it is clear that $R_\varphi(N)\in \mathcal T$. Using (a), we see that $  N/R_\varphi(N)\in \mathcal F$ and this proves (b). 
\end{proof}

\begin{Thm}\label{torclassh} Let $(\mathcal{C},\mathcal{A},\Phi)$ be a rational pairing. Suppose that for each $x\in Ob(\mathscr X)$ and
$N\in {_{\mathcal A_x}}\mathbf M$, we have $R_{\varphi_x}(N/R_{\varphi_x}(N))=0$. Then, the full subcategory $\mathcal{A}\text{-}Rat^{tr}$ of rational modules is a torsion class in $\mathcal{A}\text{-}Mod^{tr}$.
\end{Thm}

\begin{proof} Since $\mathcal{A}\text{-}Mod^{tr}$ is an abelian category that is both complete and cocomplete, it suffices (see \cite[$\S$ 1.1]{BeR}) to show that  the full subcategory $\mathcal{A}\text{-}Rat^{tr}$ is closed under coproducts, quotients and extensions. Since $R_{\varphi_x}(N/R_{\varphi_x}(N))=0$ for every $x\in Ob(\mathscr X)$ and $N\in {_{\mathcal A_x}}\mathbf M$, we know 
from Proposition \ref{P7.2wd} that rational $\mathcal A_x$-modules form a torsion class in ${_{\mathcal A_x}}\mathbf M$. 

\smallskip We consider now a short exact sequence 
\begin{equation}
0\longrightarrow \mathcal M'\longrightarrow \mathcal M\longrightarrow \mathcal M''\longrightarrow 0
\end{equation} in $\mathcal{A}\text{-}Mod^{tr}$. Suppose that $\mathcal M\in \mathcal{A}\text{-}Rat^{tr}$. Then, each $\mathcal M_x\longrightarrow 
\mathcal M''_x$ is an epimorphism in ${_{\mathcal A_x}}\mathbf M$. In particular, each $\mathcal M''_x$ is a rational $\mathcal A_x$-module and we get 
$\mathcal M''\in \mathcal{A}\text{-}Rat^{tr}$. 

\smallskip
Similarly, if $\mathcal M'$, $\mathcal M''\in \mathcal{A}\text{-}Rat^{tr}$, we see that each $0\longrightarrow \mathcal M'_x\longrightarrow \mathcal M_x\longrightarrow \mathcal M''_x\longrightarrow 0$ is exact in ${_{\mathcal A_x}}\mathbf M$ with both $\mathcal M'_x$, $\mathcal M''_x$ rational. Again, since rational $\mathcal A_x$-modules are closed under extensions in ${_{\mathcal A_x}}\mathbf M$, it follows that each $\mathcal M_x$ is rational, i.e.,
$\mathcal M\in  \mathcal{A}\text{-}Rat^{tr}$. By similar reasoning, we see that $ \mathcal{A}\text{-}Rat^{tr}$ is also closed under coproducts. This proves the result.
\end{proof}

\begin{rem} \emph{Since the submodules of a rational module are always rational (see \eqref{7.6eqtu}), we note that the condition in Theorem \ref{torclassh} makes  $\mathcal{A}\text{-}Rat^{tr}$ a hereditary torsion class in $\mathcal{A}\text{-}Mod^{tr}$.}
\end{rem}

\begin{cor} Let $(\mathcal{C},\mathcal{A},\Phi)$ be a rational pairing. Suppose that for each $x\in Ob(\mathscr X)$, $R_{\varphi_x}(\mathcal A_x)$ is dense in $\mathcal A_x$. Then, the full subcategory $\mathcal{A}\text{-}Rat^{tr}$ of rational modules is a torsion class in $\mathcal{A}\text{-}Mod^{tr}$.
\end{cor}
\begin{proof}
This follows directly from Lemma \ref{P7.1ws} and Theorem \ref{torclassh}.
\end{proof}

\begin{cor}
Let $\mathcal{C}:\mathscr{X} \longrightarrow Coalg$ be a representation taking values in right semiperfect $K$-coalgebras. Then, $Com^{cs}\text{-}\mathcal C$ forms  a torsion class in $\mathcal{C^*}\text{-}Mod^{tr}$.
\end{cor}
\begin{proof}
Let $C$ be a coalgebra and $\varphi:C\otimes C^*\longrightarrow K$ the canonical pairing. If $C$ is right semiperfect, we know from \cite[Theorem 23]{Lin} that
$R_\varphi(N/R_\varphi(N))=0$ for any left $C^*$-module $N$. The result is now clear from Corollary \ref{C7.11bq} and Theorem \ref{torclassh}.
\end{proof}

Let $C$ be a $K$-coalgebra. We note that any $C$-contramodule $(M,\pi^C_M:Hom_K(C,M)\longrightarrow M)$ may be treated as a $C^*$-module by considering
$C^*\otimes M\longrightarrow Hom_K(C,M)\xrightarrow{\pi^C_M}M$. We recall from \cite[Theorem 3.11]{BBW} that this determines a functor 
$
	\mathbf{M}_{[C,-]} \longrightarrow \mathbf{M}_{C^*}
$.

\begin{thm}
	Let $\mathcal{C}:\mathscr{X} \longrightarrow Coalg$ be a coalgebra representation. Then, we have a functor $Cont^{tr}\text{-}\mathcal{C} \longrightarrow Mod^{cs}\text{-}\mathcal{C}^*$.
\end{thm}
\begin{proof}
	Let $\mathcal{M} \in Cont^{tr}\text{-}\mathcal{C}$. Then, $\mathcal{M}_x\in \mathbf M_{[\mathcal C_x,\_\_]}$ may be treated as a $\mathcal C_x^*$-module for each 
	$x\in Ob(\mathscr X)$. Additionally, for each $\alpha \in \mathscr{X}(x,y)$, the  morphism ${_\alpha}\mathcal{M}:\mathcal{M}_y \longrightarrow \alpha_\bullet\mathcal{M}_x $  in $\mathbf M_{[\mathcal C_y,\_\_]}$ induces a morphism in $\mathbf M_{\mathcal C_y^*}$. This proves the result.
\end{proof}

For a $K$-coalgebra $C$ and a $C$-bicomodule $N$, we recall (see \cite{semicontra})  that the functor $\mathbf{M}^{{C}}(N,{-}):	\mathbf{M}^{{C}}\longrightarrow Vect$ takes values in $ \mathbf{M}_{[{C},\_\_]}$. Moreover, the functor 
$
		\mathbf{M}^{{C}}(N,{-}):	\mathbf{M}^{{C}}\longrightarrow \mathbf{M}_{[{C},\_\_]} 
$ has a left adjoint 
$
	{-}\boxtimes_C N: \mathbf{M}_{[{C},\_\_]} \longrightarrow 	\mathbf{M}^{{C}}
$
known as the contratensor product (see \cite{semicontra}). Explicitly, for any $M\in \mathbf{M}_{[{C},\_\_]}$, we have 
\begin{equation}\label{conten7}M\boxtimes_CN:=  Coeq\left(\begin{tikzcd}
			Hom_K(C,M)\otimes N \ar[r,shift left=.75ex," "]
			\ar[r,shift right=.75ex,swap," "]
			&
			M\otimes N
		\end{tikzcd}\right)\end{equation} where the two maps in \eqref{conten7} are induced 
by the structure maps $Hom_K(C,M)\longrightarrow M$ and $N\longrightarrow C\otimes N$ of $M$ and $N$ respectively.

\smallskip
Let $\alpha:C\longrightarrow D$ be a morphism of cocommutative coalgebras. Let $M$ be a $C$-contramodule and $N$ be a $C$-comodule. From the construction of the contratensor product in \eqref{conten7}, we note that there is a canonical morphism 
\begin{equation}\label{7.17uc}
\alpha_\bullet M\boxtimes_D\alpha^*N \longrightarrow \alpha^*(M\boxtimes_CN)
\end{equation} in $\mathbf M^D$.

\begin{lem}\label{L7.20g} Let $\alpha:C\longrightarrow D$ be a quasi-finite morphism of cocommutative coalgebras. Then, 

\smallskip
(a) For any $M\in \mathbf M_{[D,\_\_]}$ and $N\in \mathbf M^D$, we have a  natural isomorphism in $\mathbf M^C$: 
\begin{equation}\label{7.15vq} \alpha^!(M\boxtimes_DN)\cong (\alpha^\bullet M)\boxtimes_C(\alpha^!N)\end{equation}

\smallskip
(b) For any $N\in \mathbf M^D$ and $P\in \mathbf M^C$, we have a natural isomorphism in $\mathbf M_{[D,\_\_]}$
\begin{equation}\label{rv7.16h}
\mathbf M^D(N,\alpha^*P)\cong \alpha_\bullet\mathbf M^C(\alpha^!N,P)
\end{equation} 

\end{lem} 

\begin{proof} (a) We first set $M=D^*$ in $\mathbf M_{[D,\_\_]}$. From \cite[$\S$ 1]{BPS}, we know that $D^*\boxtimes_DN\cong N$ in $\mathbf M^D$. Since $\alpha^\bullet D^*=C^*$ in $\mathbf M_{[C,\_\_]}$, \eqref{7.15vq} reduces to 
\begin{equation} \label{7.16vq} \alpha^!(M\boxtimes_DN)=\alpha^!(D^*\boxtimes_DN)\cong \alpha^!N\cong C^*\boxtimes_C(\alpha^!N)=(\alpha^\bullet D^*)\boxtimes_C(\alpha^!N)\cong (\alpha^\bullet M)\boxtimes_C(\alpha^!N)
\end{equation} We also know (see \cite[$\S$ 1]{BPS}) that $\_\_\boxtimes_D\_\_$ preserves colimits in both arguments. Since $\alpha^\bullet$ and $\alpha^!$ are both left adjoints, it follows from \eqref{7.16vq} that \eqref{7.15vq} holds for any free contramodule $M={D^*}^{(I)}$, i.e., any direct sum of copies of $D^*$. Since any $M\in \mathbf M_{[D,\_\_]}$ can be expressed as a cokernel $M=Cok(f:F_2\longrightarrow F_1)$ 
of a morphism of free contramodules and both  $\alpha^\bullet$ and $\alpha^!$  preserve colimits, we obtain $\alpha^!(
M\boxtimes_DN)\cong (\alpha^\bullet M)\boxtimes_C(\alpha^!N)$. By lifting morphisms in $ \mathbf M_{[D,\_\_]}$ to morphisms of their free resolutions, we see that this isomorphism does not depend on the choice of $M=Cok(f:F_2\longrightarrow F_1)$. 

\smallskip
(b) This follows from Yoneda lemma and the isomorphism in part (a), by noting that for any $M\in  \mathbf M_{[D,\_\_]}$, we have
\begin{equation}
\begin{array}{ll}
 \mathbf M_{[D,\_\_]}(M,\mathbf M^D(N,\alpha^*P))&\cong \mathbf M^D(M\boxtimes_DN,\alpha^*P)\\ 
 &\cong \mathbf M^C(\alpha^!(M\boxtimes_DN),P) \cong \mathbf M^C((\alpha^\bullet M)\boxtimes_C(\alpha^!N),P)\\ &\cong 
 \mathbf M_{[C,\_\_]} (\alpha^\bullet M,\mathbf M^C(\alpha^!N,P)) \cong \mathbf M_{[D,\_\_]}(M,\alpha_\bullet \mathbf M^C(\alpha^!N,P))\\
 \end{array}
\end{equation}\end{proof}

\begin{Thm}\label{T7.19fe}  Let $\mathcal{C}:\mathscr{X} \longrightarrow Coalg$ be a coalgebra representation taking values in cocommutative coalgebras. Let $\mathcal N\in Com^{tr}_c$-$\mathcal C$ be a cartesian trans-comodule over $\mathcal C$. 

\smallskip
(a) We have a functor $F:Cont^{tr}$-$\mathcal{C}\longrightarrow Com^{tr}$-$\mathcal{C}$ defined by setting 
\begin{equation}
F(\mathcal M)_x:=\mathcal M_x\boxtimes_{\mathcal C_x}\mathcal N_x\qquad\forall\textrm{ }x\in Ob(\mathscr X), \textrm{ }\mathcal M\in Cont^{tr}\text{-}\mathcal C
\end{equation} 
(b) Suppose that $\mathcal{C}:\mathscr{X} \longrightarrow Coalg$ is quasi-finite. Then, we have a functor $G:Com^{tr}$-$\mathcal{C}\longrightarrow Cont^{tr}$-$\mathcal{C}$ defined by setting
\begin{equation}
G(\mathcal P)_x:=\mathbf M^{\mathcal C_x}(\mathcal N_x,\mathcal P_x) \qquad \forall\textrm{ }x\in Ob(\mathscr X), \textrm{ }\mathcal P\in Com^{tr}\text{-}\mathcal C
\end{equation} 
In that case, $(F,G)$ is a pair of adjoint functors. 
\end{Thm} 

\begin{proof}
(a) We consider $\alpha\in \mathscr X(y,x)$. Then, we have the following composition in $\mathbf M^{\mathcal C_x}$
\begin{equation*}
\begin{CD}
_\alpha F(\mathcal M):F(\mathcal M)_x = \mathcal M_x\boxtimes_{\mathcal C_x} \mathcal N_x @>_{\alpha}\mathcal M\boxtimes_{\mathcal C_x}{_\alpha\mathcal N}>>(\alpha_\bullet\mathcal M_y)\boxtimes_{\mathcal C_x}(\alpha^\ast\mathcal N_y)@>>> \alpha^*(\mathcal M_y\boxtimes_{\mathcal C_y}\mathcal N_y)=\alpha^* F(\mathcal M)_y
\end{CD}
\end{equation*} where the last morphism follows from \eqref{7.17uc}. Accordingly, we have a functor $F:Cont^{tr}$-$\mathcal{C}\longrightarrow Com^{tr}$-$\mathcal{C}$. 

\smallskip
(b) Let  $\alpha\in \mathscr X(y,x)$. Using the adjoint pair $(\alpha^!,\alpha^*)$, we have a canonical morphism $u_\alpha(\mathcal P):\mathcal P_x\longrightarrow \alpha^\ast\alpha^!
\mathcal P_x$ in $\mathbf M^{\mathcal C_x}$. Accordingly, we have the following composition in $\mathbf M_{[\mathcal C_x,\_\_]}$
\begin{equation*}
\begin{CD}
\mathbf M^{\mathcal C_x}(\mathcal N_x,\mathcal P_x)@>\mathbf M^{\mathcal C_x}(\mathcal N_x,u_\alpha(\mathcal P))>>\mathbf M^{\mathcal C_x}(\mathcal N_x, \alpha^\ast\alpha^!\mathcal P_x)@>\cong>> \alpha_\bullet \mathbf M^{\mathcal C_y}(\alpha^!\mathcal N_x,\alpha^!\mathcal P_x)@>\alpha_\bullet \mathbf M^{\mathcal C_y}(\alpha^!\mathcal N_x,{^\alpha\mathcal P})>>  \alpha_\bullet \mathbf M^{\mathcal C_y}(\alpha^!\mathcal N_x, \mathcal P_y)
\end{CD}
\end{equation*} where the isomorphism in the middle is obtained from Lemma \ref{L7.20g}(b). Since $\mathcal N\in Com^{tr}_c$-$\mathcal C$ is cartesian, we put $\alpha^!\mathcal N_x=\mathcal N_y$. The above therefore gives us a morphism 
\begin{equation*}
\begin{CD}
_\alpha G(\mathcal P): G(\mathcal P)_x=\mathbf M^{\mathcal C_x}(\mathcal N_x,\mathcal P_x) @>>>  \alpha_\bullet \mathbf M^{\mathcal C_y}(\alpha^!\mathcal N_x, \mathcal P_y)=\alpha_\bullet \mathbf M^{\mathcal C_y}(\mathcal N_y, \mathcal P_y)=\alpha_\bullet G(\mathcal P)_y
\end{CD}
\end{equation*} corresponding to $\alpha\in \mathscr X(y,x)$. Hence, we have a functor $G:Com^{tr}$-$\mathcal{C}\longrightarrow Cont^{tr}$-$\mathcal{C}$. The adjointness of $F$ and $G$ follows from the fact that $\mathbf M^{\mathcal C_x}(\mathcal M_x\boxtimes_{\mathcal C_x}\mathcal N_x,\mathcal P_x)\cong \mathbf M_{[\mathcal C_x,\_\_]}(\mathcal M_x,
\mathbf M^{\mathcal C_x}(\mathcal N_x,\mathcal P_x))$ for each $x\in Ob(\mathscr X)$. 
\end{proof}

\end{document}